\documentclass{amsart}
\usepackage[dvipdfmx]{graphicx}
\usepackage{amscd}
\usepackage{color}
\usepackage{amsmath,amssymb,amsthm}

\newcommand{\p}{$(p-1)q\equiv \pm1$ $(mod\, p)$}
\newcommand{\np}{$(p-1)q\not\equiv \pm1$ $(mod\, p)$}
\newtheorem{Def}{Definition}[section]
\newtheorem{thm}{Theorem}[section]
\newtheorem{lem}{Lemma}[section]
\newtheorem{prop}{Proposition}[section]
\newtheorem{claim}{Claim}[section]
\newtheorem{Rem}{Remark}[section]

\begin{document}

%\title{Decomposition of lens spaces with three handlebodies}
\title[Decompositions of the 3-sphere and lens spaces with three handlebodies]{Decompositions of the 3-sphere and lens spaces\\ with three handlebodies}
\author{Yasuyoshi Ito and Masaki Ogawa}
\date{}
\maketitle
\begin{abstract}
In this paper, we consider decompositions of 3-manifolds with three handlebodies. We classify such decompositions of the 3-sphere and lens spaces with small genera. These decompositions admit operations called stabilizations.
We also determine whether these decompositions are are obtained by stabilization from another decomposition.
\end{abstract}

\maketitle
\section{Introduction}
	In this paper, we consider decompositions of 3-manifolds with handlebodies. A Heegaard splitting is a decomposition of a 3-manifold with two handlebodies. Heegaard splittings, which have been studied for many years, play an important role in the study of 3-manifolds.  We consider decompositions of 3-manifolds with three handlebodies. Such a decomposition is more general than a Heegaard splitting. We classify such decompositions of the 3-sphere and lens spaces up to ambient isotopy in this paper.

	%In this paper, we will consider the decompositions of 3-manifolds with three handlebodies. It is a generalization of a Heegaard splitting. Gomez-Larrra\~naga  characterized  decompositions of 3-manifolds with three handlebodies of genera at most one \cite{G}. 
	In dimension 4, Gay and Kirby introduced  a trisection of a 4-manifold, i.e., a decomposition of  a 4-manifold  into  three  four-dimensional handlebodies \cite{GK}. They showed that any closed, orientable, smooth 4-manifold admits a trisection.  In \cite{Koenig},  Koenig defined a trisection of  a 3-manifold. A trisection of a 3-manifold is a decomposition that uses three handlebodies and satisfies some properties. In this paper, we will consider a slightly generalized version, i.e, we do not assume that  the intersection of two handlebodies is connected. If a 3-manifold $M$ is decomposed into three handlebodies of genera $g_1$, $g_2$, and $g_3$, respectively, then we call this decomposition a type-$(g_1, g_2, g_3)$ decomposition. G\'omez-Larra\~naga et al. also studied such decompositions and characterized  the decompositions of 3-manifolds with three handlebodies of genera at most one \cite{G}. In \cite{G}, G\'omez-Larra\~naga showed that the 3-sphere has type-$(0, 0, 0)$, type-$(0, 0, 1)$, type-$(0, 1, 1)$, and type-$(1, 1, 1)$ decompositions  and lens spaces have type-$(0, 0, 1)$, type-$(0, 1, 1)$, and type-$(1, 1, 1)$ decompositions. 
	
	 Several results on the classification of Heegaard splittings of 3-manifolds are known. Waldhausen proved that the Heegaard splittings of the 3-sphere of the same genus are unique up to isotopy \cite{W}. Further, Bonahon and Otal proved that two Heegaard splittings of lens spaces of the same genus are unique up to isotopy \cite{BO}. 
	 We classify the such decompositions of the 3-sphere and lens spaces up to isotopy  if the genera of hadlebodies at most one in Theorem \ref{s3001}, \ref{lensclass011}, \ref{class111s3} and  \ref{class111}.
	 In the cases of handlebody decompositions, it was found to have multiple isotopy classes.
	 
	 It is known that a Heegaard splitting admits a stabilization operation. Reidemeister and Singer showed  that  two Heegaard splittings of the same 3-manifold will be isotopic after performing some stabilizations \cite{R, S}. Koenig defined the stabilizations of a trisection and showed a stably equivalent theorem \cite{Koenig}. In \cite{Poly}, it has been shown that the handlebody decompositions of the same 3-manifold are stably equivalent. We will characterize handlebody decompositions of the 3-sphere and lens spaces with genera at most one that can be obtained from others by stabilizations in Theorem \ref{stab_2}.
	 
	 In Section 2, we introduce a handlebody decomposition and list the results of this paper.  After that, we review some results on a handlebody decomposition in Section 3 and classify handlebody decompositions of the 3-sphere and lens spaces in Section 4. 
	
%=========================================================================================================================================================================
\section{Contents and Results}
	In this section, we introduce the notion of a handlebody decomposition of a 3-manifold.
	Koenig defined a trisection of a 3-manifold in \cite{Koenig}. The following definition of a handlebody decomposition is slightly more general than that of a trisection of a 3-manifold.
	The results on classifications are stated in Theorem \ref{s3001}, \ref{lensclass011}, \ref{class111s3} and  \ref{class111}.
\subsection{Definition of handlebody decomposition}
	\begin{Def}[A handlebody decomposition]
		Let $M$ be a closed, orientable 3-manifold and $H_i$ be a genus-$g_i$ handlebody  for $i=1, 2, 3$.  Suppose that M has a decomposition $M=H_1\cup H_2\cup H_3$. Then, we call this decomposition a handlebody decomposition if it satisfies the following conditions:
			\begin{enumerate}
			\renewcommand{\theenumi}{(\arabic{enumi})}
			\item[(1)] $H_i\cap H_j=\partial H_i\cap \partial H_j$ is a  (possibly disconnected) compact surface with a boundary denoted by $F_{ij}$; and
			\item[(2)] $H_1\cap H_2\cap H_3$ is a set of simple closed curves in M. We call a component of this set a branched locus.
		\end{enumerate}
		We refer to the union $F_{12}\cup F_{13}\cup F_{23}$ as a branched surface of the handlebody decomposition. We call this  decomposition  a type-$(g_1, g_2, g_3; b)$ decomposition, where $b$ is the number of branched loci. We sometimes omit $b$.
	\end{Def}
	
	\begin{Rem}
	\begin{enumerate}
	\renewcommand{\theenumi}{(\arabic{enumi})}
		\item[(1)] The above-mentioned definition of a handlebody decomposition corresponds to the one in \cite{Poly} such that its partition does not have a vertex and the number of handlebodies is three.
		\item[(2)] The difference between the definition of a trisection of 3-manifolds \cite{Koenig} and the above-mentioned definition is the connectedness of the intersection of two handlebodies. In our definition, we do not assume the connectedness. 
	\end{enumerate}
	\end{Rem}
	
	\subsection{Results on classifications}
	
	Our first result is the classification of the decompositions of the 3-sphere and lens spaces. In this paper, $D^2$, $D$, and $D_i$ denote a two-dimensional disk, $A$ and $A_i$ denote an annulus, $P$ denotes a thrice-punctured two-dimensional sphere, and $T^{\circ}$ denotes a once-punctured torus. Let $X$ and $Y$ be  topological spaces. The notation $X\cong Y$ implies that there exists an orientation-preserving or orientation-reversing homeomorphism between $X$ and $Y$. A  closed, orientable 3-manifold is called a lens space if it admits a genus-one Heegaard splitting and it is not homeomorphic to the 3-sphere or $S^2\times S^1$. Let $V_1$ and $V_2$ be solid tori and $M=V_1\cup_\varphi V_2$ be a lens space, where $\varphi$ is an orientation-reversing homeomorphism from $\partial V_2$ to $\partial V_1$.  We suppose that $m_2$ is a meridian of $\partial V_2$.
	Then, choosing a fixed longitude and meridian, $l_1$ and $m_1$, respectively, of $\partial V_1$, we write $\varphi_\ast(m_2)=pl_1+qm_2$, where $p$ and $q$ are coprime integers and $\varphi_\ast$ is an isomorphism between $H_1(\partial V_2)$ and $H_1(\partial V_1)$. 
	Then, we denote $M=L(p, q)$. It is well known that $L(p, q)\cong L(-p, q)\cong L(p, -q)\cong L(-p, -q)$. Furthermore, lens spaces $L(p, q)$ and $L(p', q')$ are homeomorphic if and only if $p=p'$ and either $q\equiv\pm q'$ or $qq'\equiv \pm 1$ $(mod\, p)$. This has been shown by Brody in \cite{B}. Hence, we assume that  $p>0, q>0$, and $q<p$ in this paper.

	\begin{thm}\label{s3001}
			A type-$(0,0,1)$ decomposition of the 3-sphere and a lens space $M$ has exactly two branched loci and it satisfies $F_{12} \cong D_{1} \cup D_{2}$ and $F_{23} \cong F_{31} \cong A$. 
			In fact, the 3-sphere and lens spaces have a type-$(0, 0, 1)$ decomposition that satisfies $F_{12} \cong D_{1} \cup D_{2}$ and $F_{23} \cong F_{31} \cong A$.
			
			Furthermore, there is exactly one embedding of the branched surface of the type-$(0, 0, 1)$ decomposition up to ambient isotopy if $M$ is the 3-sphere or a lens space with \p. Otherwise,  there are exactly two embedding of the branched surface of the type-$(0, 0, 1)$ decomposition up to ambient isotopy.
	\end{thm}
	
The following theorem is a characterization of type-$(0, 1, 1)$ decompositions of the 3-sphere and lens spaces. 	
\begin{thm}\label{s3011}
A type-$(0,1,1)$ decomposition of the 3-sphere and a lens space  $M$ satisfies one of the following: 
			\begin{enumerate}
			\renewcommand{\theenumi}{(\arabic{enumi})}
				\item[(1)] it has exactly one branched locus and satisfies $F_{12} \cong F_{31} \cong D$ and $F_{23} \cong T^{\circ}$. 
				\item[(2)] it has exactly three branched loci and satisfies $F_{12} \cong F_{31} \cong D \cup A $ and $F_{23} \cong P$. 
			\end{enumerate}
			In fact, the 3-sphere and lens spaces admit each type-$(0, 1, 1)$ decomposition.
			%has ... . (branched surfaces are one loop or three loops.)
\end{thm}
The following theorem is a classification of a type-$(0,1,1)$ decomposition of the 3-sphere and a lens space up to ambient isotopy. 
% and use diffeotopy of lens space. 
\begin{thm}\label{lensclass011}
Let $M$ be the 3-sphere or a lens space.
If $M$ is the 3-sphere or a lens space $L(p, q)$ with $p=2$,  the type-$(0, 1, 1)$ decompositions of $M$ can be classified up to ambient isotopy as follows.
\begin{enumerate}
\renewcommand{\theenumi}{(\arabic{enumi})}
\item[(1)] there is exactly one ambient isotopy class of the branched surface of the type-$(0, 1, 1)$ decomposition of conclusion (1) of Theorem \ref{s3011}.
\item[(2)]  there is exactly one ambient isotopy class of the branched surface of the type-$(0, 1, 1)$ decomposition of conclusion (2) of Theorem \ref{s3011}.
\end{enumerate}
On the other hand, if $M$ is a lens space $L(p, q)$ with $p\neq 2$, the type-$(0, 1, 1)$ decompositions of $M$ can be classified up to ambient isotopy as follows.
\begin{enumerate}
\item[(3)]   there is exactly one  isotopy class of the branched surface of the type-$(0, 1, 1)$ decomposition of conclusion (1) of Theorem \ref{s3011}.
\item[(4)]  there are exactly two isotopy classes  of the branched surface of the type-$(0, 1, 1)$ decomposition of conclusion (2) of Theorem \ref{s3011} if $M$ is $L(p, q)$ with \p. Otherwise,  there are exactly  four isotopy classes of the branched surface of the type-$(0, 1, 1)$  decomposition of conclusion (2) of Theorem \ref{s3011}.
\end{enumerate}
\end{thm}

A type-$(1, 1, 1)$ decomposition of the 3-sphere and a lens space is characterized as follows.
	\begin{thm}\label{thm 3}
		A type-$(1,1,1)$ decomposition of the 3-sphere and a lens space $M$ satisfies one of the following, where $\{i, j, k\}=\{1, 2, 3\}$.
			\begin{enumerate}
			\renewcommand{\theenumi}{(\arabic{enumi})}
				\item[(1)] 
					it has exactly two branched loci and satisfies $F_{ij} \cong F_{jk} \cong F_{ki} \cong A$.
					%(Figure \ref{111deco1}).
				\item[(2)]  it has exactly two branched loci and satisfies
					$F_{ij}\cong D \cup T^{\circ}$ and $F_{jk} \cong F_{ki} \cong A $.
					%They are longitudal in $\partial T_1$ and inessential in $\partial T_2$ and $\partial T_3$.
					%(Figure \ref{111deco2}).
				\item[(3)]  it has exactly four branched loci and satisfies
					$F_{ij} \cong D \cup P$ and $F_{jk} \cong F_{ki} \cong A_{1} \cup A_{2}$.
					%They are longitudal in $\partial T_1$, two are longitudal and the others are inessential and parallel in $\partial T_2$ and $T_3$.
					%(Figure \ref{111deco3}).
				\item[(4)]  it has exactly four branched loci and satisfies
					$F_{ij} \cong A_{1} \cup A_{2}$ and $F_{jk} \cong F_{ki} \cong D \cup P$.
					%, two of them are longitudal and the others are inessential in each $\partial T_i$.
					%Two longitudal loops are parallel in $\partial T_1$ and $\partial T_2$, and not parallel in  $\partial T_3$.
					%(Figure \ref{111deco4}).
				\item[(5)]  it has exactly four branched loci and satisfies
					$F_{ij} \cong F_{jk} \cong F_{ki} \cong D \cup P$.
					%, two of them are longitudal and the others are inessential in each $\partial T_i$.
					%Any pair of them are not parallel in  $\partial T_i$'s.
					%(Figure \ref{111deco5}).
				\end{enumerate}
			Furthermore, the following holds. 
				\begin{enumerate}
				 \item[(6)] $M \cong L(4,1)$. In this case,  $M$ also  has a decomposition that satisfies $F_{12}\cong F_{13}\cong F_{23} \cong A\cup A$ and has four branched loci.
		\end{enumerate}
		In fact, if $M$ is not $L(4, 1)$, $M$ admits type-$(1, 1, 1)$ decompositions that satisfy (1), (2), (3), (4), and (5) above. In addition, if  $M\cong L(4, 1)$, $M$ admits type-$(1, 1, 1)$ decompositions that satisfy  (1), (2), (3), (4), (5), and (6) above.
		
		\end{thm}
	\begin{Rem}
		A type-$(1, 1, 1)$ decomposition of conclusion (1) whose branched loci are not meridional in the boundary of each handlebody is  derived from a Seifert fibered structure of the 3-sphere or a lens space, i.e., each handlebody is a fibered torus. 
		Furthermore,  the type-$(1, 1, 1)$ decomposition of Theorem \ref{thm 3} (1) whose boundary has meridional branched loci  has at most one handlebody.  
	\end{Rem}
	
	We also classify the type-$(1, 1, 1)$ decompositions of the 3-sphere and lens spaces that satisfy cases (2), (3), (4), and (5) of Theorem \ref{thm 3}. Theorem \ref{class111s3} deals with the classification of type-$(1, 1, 1)$ decompositions of the 3-sphere and Theorem \ref{class111} deals with the classification of type-$(1, 1, 1)$ decompositions of  lens spaces. 
	
		\begin{thm}\label{class111s3}
			There is exactly one ambient isotopy class of embeddings of  the branched surface of  a type-$(1, 1, 1)$ decomposition of the 3-sphere that satisfies one of conclusions (2), (3), (4), or (5) of Theorem \ref{thm 3}.
		\end{thm}
		
\begin{thm}\label{class111}
		Let $M$ be a lens space $L(p, q)$.
If $M$ satisfies $p=2$,  type-$(1, 1, 1)$ decompositions of $M$ can be classified as follows.
			\begin{enumerate}
			\renewcommand{\theenumi}{(\arabic{enumi})}
			\item[(1)] there are exactly two isotopy classes of  the branched surface of a handlebody decomposition of conclusion (2) of Theorem \ref{thm 3}.
			\item[(2)] there are exactly two isotopy classes of  the branched surface of a handlebody decomposition of conclusion (3) of Theorem \ref{thm 3}.
			\item[(3)] there are exactly two isotopy classes of  the branched surface of a handlebody decomposition of conclusion (4) of Theorem \ref{thm 3}.
			\item[(4)] there is exactly one isotopy class of  the branched surface of a handlebody decomposition of conclusion (5) of Theorem \ref{thm 3}.
			\end{enumerate}
		On the other hand, if $M$ satisfies $p\neq 2$, type-$(1, 1, 1)$ decompositions of $M$ can be classified as follows.		
			\begin{enumerate}
			\item[(5)] there are exactly two isotopy classes of the branched surface of a decomposition of conclusion (2) of Theorem \ref{thm 3}  if $M$ satisfies \p. Otherwise,  there are exactly  four isotopy classes of the branched surface of a decomposition of conclusion (2) of Theorem \ref{thm 3}.
			\item[(6)] there are exactly three isotopy classes of the branched surface of a decomposition of conclusion (3) of Theorem \ref{thm 3} if $M$ satisfies \p. Otherwise,  there are exactly six isotopy classes of the branched surface of a decomposition of conclusion (3) of Theorem \ref{thm 3}.
			\item[(7)] there are exactly four isotopy classes of the branched surface of a decomposition of conclusion (4) of Theorem \ref{thm 3} if $M$ satisfies \p.  Otherwise, there are exactly eight isotopy classes of the branched surface of a decomposition of conclusion (4) of Theorem \ref{thm 3}.
			\item[(8)] there is exactly one isotopy class of the branched surface of a decomposition of conclusion (5) of Theorem \ref{thm 3}  if $M$ satisfies  \p. Otherwise,  there are exactly two ambient isotopy classes of the branched surface of a decomposition of conclusion (5) of Theorem \ref{thm 3}.
			\end{enumerate}
		\end{thm}
		
		\subsection{Results on stabilizations of a handlebody decomposition}
 The stabilization of a Heegaard splitting  is an operation that  increases the genera of handlebodies. Koenig defined  the stabilizations of a trisection of a 3-manifold \cite{Koenig}. Similarly, a handlebody decomposition admits the following stabilizations. Koenig showed that any two trisections of the same 3-manifold are isotopic to each other after some stabilizations.  We show that the decompositions of the 3-sphere with handlebodies of genera at most one are stabilized from a type-$(0, 0, 0)$ decomposition. In addition, we show that the handlebody decompositions of  lens spaces of  genera at most one are stabilized from a type-$(0, 0, 1)$ decomposition except Theorem \ref{thm 3} (6).  First, we define the stabilizations of  handlebody decompositions.

	\begin{Def}[stabilization]
Let $M = H_{i} \cup H_{j} \cup H_{k}$ be a type-$(g_{i}, g_{j}, g_{k})$ decomposition.

\begin{enumerate}
\renewcommand{\theenumi}{(\arabic{enumi})}
\item[(1)] The following operation is called a {\it type-0 stabilization} (Figure \ref{typeii}).
We take two points on the interior of $F_{ij}$ and connect them by a properly embedded boundary-parallel arc $\alpha$ in $H_{i}$. 
Let $N(\alpha)$ be the regular neighborhood of $\alpha$ in $H_i$. 
Define a new handlebody decomposition $M = H'_{i} \cup H'_{j} \cup H'_{k}$ by $H'_{i} = H_{i} \setminus int(N(\alpha))$, $H'_{j} = H_{j}\cup N(\alpha)$, and $H'_{k}=H_{k}$. 
Then, the triple $(g_{i}, g_{j}, g_{k})$ is changed into $(g_{i}+1, g_{j}+1 , g_{k})$ and the number of components of branched loci is not changed by this operation. 

\item[(2)] The following operation is called a {\it type-1 stabilization} (Figure \ref{typeia}). 
We take two points on the branched loci and connect them by an arc $\alpha$ on $F_{jk}$. Let $N(\alpha)$ be the regular neighborhood of $\alpha$ in $M$. 
Define a new handlebody decomposition $M = H'_{i} \cup H'_{j} \cup H'_{k}$, where $H'_{i} = H_{i} \cup N(\alpha)$, $H'_{j} = H_{j}\setminus int(N(\alpha))$ and $H'_{k} = H_{k}\setminus int(N(\alpha))$. 
Then, the triple $(g_{i}, g_{j}, g_{k})$ is changed into $(g_{i}+1, g_{j}, g_{k})$ and the number of components of branched loci is changed by $1$.
Conversely, if there exists a non-separating disk $D_{i} \subset H_{i}$ whose boundary intersects the set of branched loci at exactly two points transversely, then $D_{i}$ can be canceled by an inverse operation of a type-1 stabilization. 
We call this operation a {\it type-1 destabilization}.

%\item[(2)]The following operation is called a {\it type 2 stabilization} (Figure \ref{typeib}).
%We take each point on $F_{ij}$ and $F_{ki}$, and connect them by a properly embedded boundary parallel arc $\alpha$ in $H_{i}$. Suppose that $\alpha \cong [0,\frac{1}{2}] \cup [\frac{1}{2}, 1]$. Let $N(\alpha) \cong N([0,\frac{1}{2}]) \cup N([\frac{1}{2}, 1])$ be the regular neighborhood of $\alpha$ in $M$. Define a new handlebody decomposition $M = H'_{i} \cup H'_{j} \cup H'_{k}$ where $H'_{i} \coloneqq H_{i} \setminus int(N(\alpha))$, $H'_{j} \coloneqq H_{j} \cup N([0,\frac{1}{2}])$ and $H'_{k} \coloneqq H_{k} \cup N([\frac{1}{2}, 1])$.
%Then the triple $(g_{i}, g_{j}, g_{k})$ is changed into $(g_{i}+1, g_{j}, g_{k})$ and the number of components of branched loci is increased by 1. 
%Conversely, if there exists a disk $D_{jk} \subset F_{jk}$ whose boundary is essential (resp. inessential) on $\partial H_{i}$ (resp. $H_{j}$ and $H_{k}$), then the triple $(g_{i},g_{j},g_{k})$ is changed into $(g_{i}-1,g_{j},g_{k})$ by an inverse operation of a type 2 stabilization. We call this operation a {\it type 2 destabilization}.

\end{enumerate}
%The operation as in Figure \ref{typeii} is called a {\it type 0 (de)stabilization}.
\begin{figure}[httb]
\centering
\includegraphics[scale=0.8]{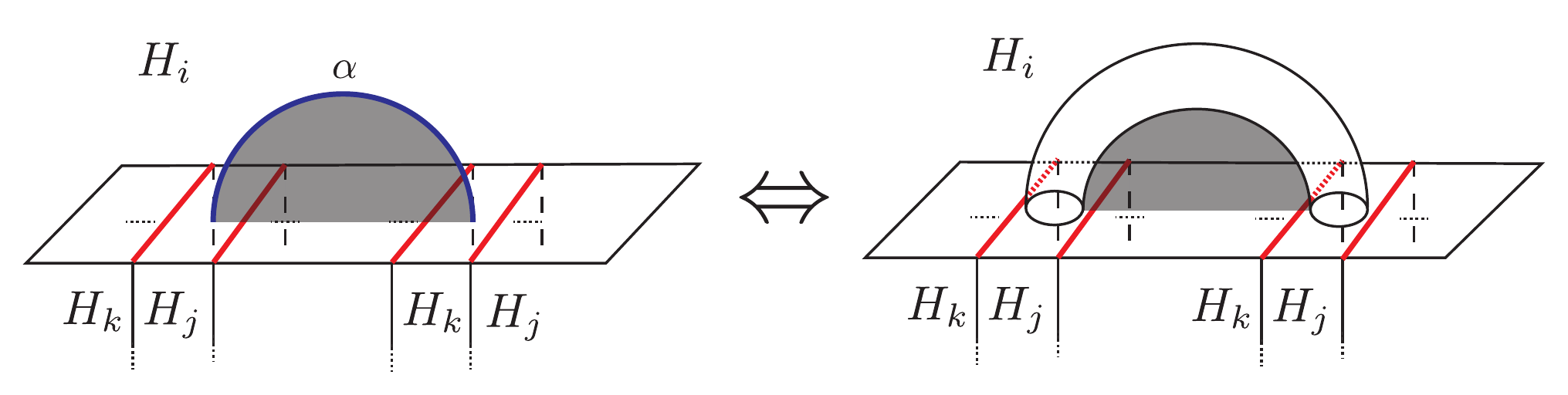}
\caption
{A type-0 stabilization along the arc $\alpha$. A triple $(g_{i}, g_{j}, g_{k})$ is changed into $(g_{i}+1, g_{j}+1 , g_{k})$.}
\label{typeii}
\end{figure}

\begin{figure}[httb]
\centering
\includegraphics[scale=0.8]{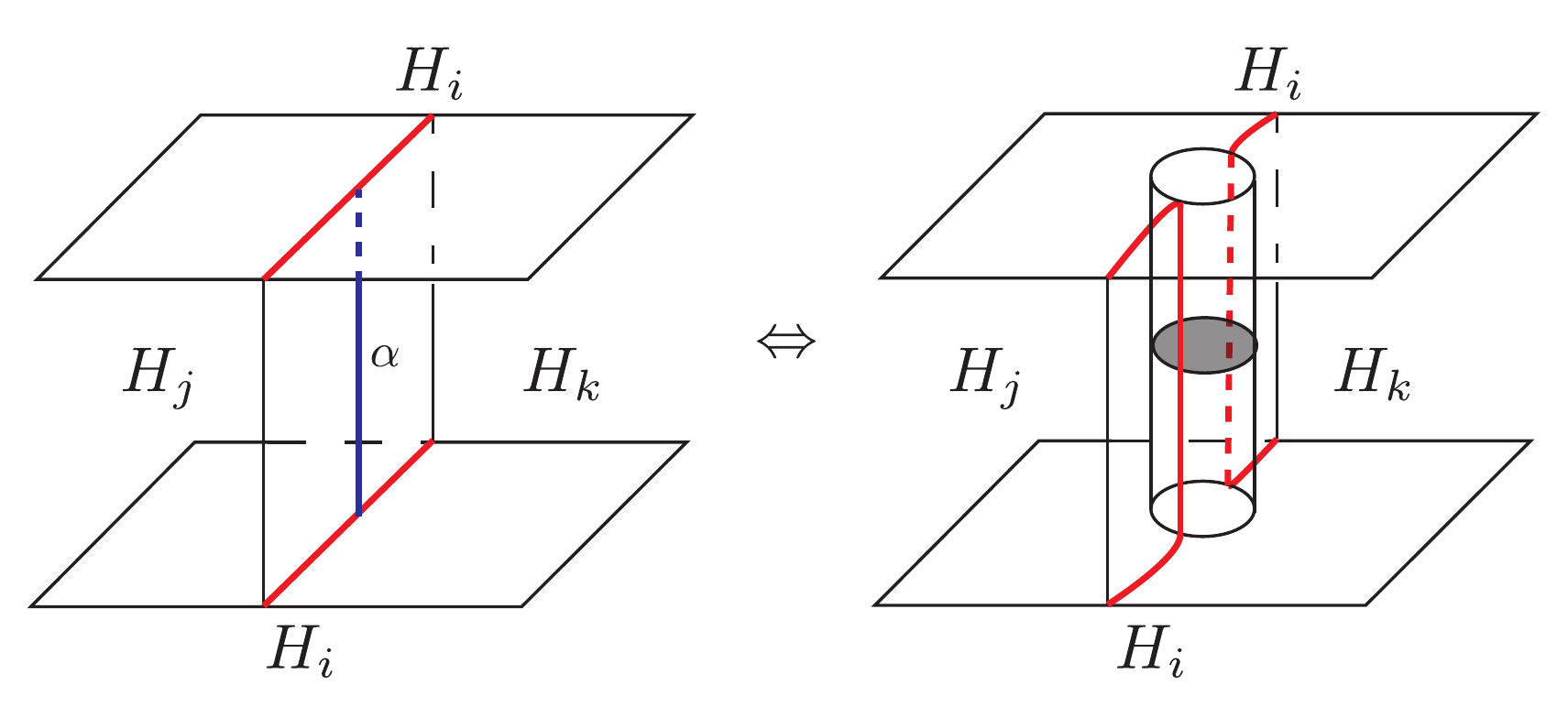}
\caption
{A type-1 stabilization along the arc $\alpha$. A triple $(g_{i}, g_{j}, g_{k})$ is changed into $(g_{i}+1, g_{j}, g_{k})$.}
\label{typeia}
\end{figure}

%\begin{figure}[H]
%\centering
%\includegraphics[scale=0.8]{Type-1_b.pdf}
%\caption
%{A type 2 stabilization along the arc $\alpha$. A triple $(g_{i},g_{j},g_{k})$ is changed into $(g_{i}-1,g_{j},g_{k})$.}
%\label{typeib}
%\end{figure}

 If a handlebody decomposition can be obtained from another handlebody decomposition by a finite sequence of the stabilizations defined above, we say that it is stabilized.
\end{Def}

The result is the following theorem.
	\begin{thm}\label{stab_2}
	A type-$(0, 0, 1)$ decomposition of the 3-sphere and type-$(0, 1, 1)$ and  type-$(1, 1, 1)$ decompositions of the 3-sphere and lens spaces satisfy following.
	\begin{enumerate}
	\renewcommand{\theenumi}{(\arabic{enumi})}
			\item[(1)]  a type-$(0,0,1)$ decomposition of the 3-sphere is obtained from a type-$(0,0,0)$ decomposition by a type-1 stabilization.
			\item[(2)]  a type-$(0,1,1)$ decomposition of the 3-sphere and lens spaces is obtained from a type-$(0,0,1)$ decomposition by a type-1 stabilization.
			\item[(3)]  a type-$(1,1,1)$ decomposition of the 3-sphere and  lens spaces is obtained from a type-$(0,0,1)$ decomposition by a sequence of  type-1 stabilizations. On the other hand, the decomposition in Theorem \ref{thm 3} (6) is not stabilized. 
\end{enumerate}
	\end{thm}
	
	 The remainder of the paper is organized as follows. In Section 3, we review some known results. In Section 4, we classify  handlebody decompositions of the 3-sphere and lens spaces.  In Section 5, we show Theorem \ref{stab_2}.

\section{Known results on  handlebody decompositions}
Gomez-Larra\~naga studied handlebody decompositions in \cite{G}. He characterized 3-manifolds that admit decompositions with handlebodies of genera at most one.  Let $\mathbb B$ be a connected sum of a finite number of  $S^2$-bundles over $S^1$'s,  let $\mathbb L$ and $\mathbb L_i$ be lens spaces, and let $\mathbb S(3)$ be a Seifert manifold  with at most three exceptional fibers. He showed the following theorem in \cite{G}.
\begin{prop}[\cite{G}]\label{GTh}Let $M$ be a closed 3-manifold.
\begin{enumerate}
\renewcommand{\theenumi}{(\arabic{enumi})}
\item[(1)] 
$M$ has a type-$(0,0,0)$ decomposition if and only if $M$ is homeomorphic to $\mathbb B$.
\item[(2)] 
$M$ has a type-$(0,0,1)$ decomposition if and only if $M$ is homeomorphic to $\mathbb B \# \mathbb L$.
\item[(3)] 
$M$ has a type-$(0,1,1)$ decomposition if and only if $M$ is homeomorphic to $\mathbb B \# \mathbb L$ or $\mathbb B \# \mathbb L_{1} \# \mathbb L_{2}$.
\item[(4)] 
$M$ has a type-$(1,1,1)$ decomposition if and only if $M$ is homeomorphic to $\mathbb B \# \mathbb L$ or $\mathbb B \# \mathbb L_{1} \# \mathbb L_{2}$ or $\mathbb B \# \mathbb L_{1} \# \mathbb L_{2} \# \mathbb L_{3} $ or $\mathbb B \# \mathbb S(3)$.
\end{enumerate}

\end{prop}
To show this proposition, Gomez-Larra\~naga applied the following lemmas. In this paper, we use these lemmas frequently.
\begin{lem}[\cite{G}, lemma 1]\label{lem1}
				Let $M=H_1\cup H_2\cup H_3$ be a handlebody decomposition of a 3-manifold $M$.
				For $\{i,j,k\}=\{1,2,3\}$, we have
				$\chi(F_{ij})=\frac{1}{2}(\chi (\partial H_i)+\chi(\partial H_j)-\chi(\partial H_k))$, where $\chi(X)$ denotes the Euler characteristic of $X$.
		\end{lem}

		\begin{lem}[\cite{G}, lemma 3]\label{lem2}
				Suppose that $F_{12}$ has at least two connected components and one of these is a disk $D$ such that $\partial D$ is inessential in $\partial H_{3}$. Then, $M = M' \# \mathbb{B}$, where $M' = H'_{1}\cup H'_{2}\cup H'_{3}$ with $H'_{i} \cong H_{i}$.
		\end{lem}
Furthermore, 3-manifolds with the following handlebody decompositions are characterized.
\begin{prop}[\cite{G2}]%他の著者(1994)]
%Mnの説明
Let $M$ be a closed 3-manifold and $M_n$ be a closed 3-manifold that admits a genus-$n$ Heegaard splitting.
\begin{enumerate}
\renewcommand{\theenumi}{(\arabic{enumi})}
\item[(1)] 
$M$ has a type-$(0,0,n)$ decomposition if and only if $M$ is homeomorphic to $\mathbb B\# M_n$.
\item[(2)] 
$M$ has a type-$(0,1,n)$ decomposition if and only if $M$ is homeomorphic to $\mathbb B \# \mathbb L\# M_n$.
\end{enumerate}
\end{prop}

%======================Handlebody decomposition of 3-sphere=====================================

\section{Classification of handlebody decompositions of the 3-sphere and lens spaces}

\subsection{Diffeotopy group of lens spaces}
In this section, we review the diffeotopy groups of lens spaces. The diffeotopy  group $\mathcal{D}(M)$ of a 3-manifold $M$ is the quotient of the diffeomorphism group Diff$(M)$ by its normal subgroup Diff$_0(M)$ of diffeomorphisms isotopic to the identity. 
The diffeotopy groups of lens spaces have been studied in \cite{H-R, B2}. 
To consider the embedding of branched surfaces of handlebody decompositions of the 3-sphere and lens spaces, we need the  diffeotopy  group of lens spaces.

Let $V_1\cup V_2$ be a genus-one Heegaard splitting of a lens space $L(p, q)$. 
We assume that $V_i=S^1\times D^2\subset \mathbb{C}\times \mathbb{C}$ for $i=1, 2$.
Then, we consider the following three diffeomorphisms preserving the Heegaard surface $\partial V_1=\partial V_2$ (see \cite{vs}).
\begin{enumerate}
\renewcommand{\theenumi}{(\arabic{enumi})}
\item[$\tau$]: an involution fixing each solid torus, defined as $\tau(u, v)=(\bar{u}, \bar{v})$ in each solid torus , where $\bar{z}$ denotes the complex conjugated of $z$;
\item[$\sigma_+$]: an involution exchanging $V_1$ and $V_2$, given by $\sigma_+:V_i\ni (u, v)\mapsto (u, v)\in V_j$ for $i\neq j$;
\item[$\sigma_-$]: a diffeomorphism exchanging $V_1$ and $V_2$, given by $\sigma_-: V_1\ni(u, v)\mapsto (\bar{u}, v)\in V_2$ and  $\sigma_-:  V_2\ni(u, v)\mapsto (u, \bar{v})\in V_1$.
\end{enumerate}
\begin{thm}[\cite{H-R, B2}]\label{Diff}
	The diffeotopy group of lens space $L(p, q)$ is isomorphic to the following:
	\begin{enumerate}
	\renewcommand{\theenumi}{(\arabic{enumi})}
	\item[(1)]  $\mathbb{Z}_2$, with generator $\sigma_-$,  if $p=2$;
	\item[(2)]  $\mathbb{Z}_2$, with generator $\tau$, if $q\equiv\pm 1$ mod $p$ and $p\neq 2$;
	\item[(3)]  $\mathbb{Z}_2\oplus \mathbb{Z}_2$, with generator $\tau$ and $\sigma_+$, if $q^2\equiv +1$ mod $p$ and $q\not\equiv \pm 1$ mod $p$;
	\item[(4)]  $\mathbb{Z}_4$, with generator $\sigma_-$, if $q^2\equiv -1$ mod $p$ and $p\neq 2$;
	\item[(5)]  $\mathbb{Z}_2$, with generator $\tau$, if $q^2\not\equiv \pm 1$ mod $p$.
	\end{enumerate}
\end{thm}
\begin{Rem}
	The restriction $\tau |_{\partial V_1}$ is a hyperelliptic involution in the mapping class group of a torus $\partial V_1$.  
	We note that $\sigma_-$ is an orientation-reversing involution (see \cite{H-R, B2}).
\end{Rem}
The following lemma can easily be obtained from Theorem \ref{Diff}.
\begin{lem}\label{diff}
	Let $M$ be the 3-sphere or a lens space $L(p, q)$ with a genus-one Heegaard splitting $V_1\cup V_2$.
	$M$ is the 3-sphere or a lens space $L(p, q)$ with $p=2$ if and only if $M$ admits an ambient isotopy $F: M \times [0, 1]\to M$ that satisfies the following:
	\begin{enumerate}
	\renewcommand{\theenumi}{(\arabic{enumi})}
		\item[(1)]  $F(V_1, 1)=V_1$ and $F(V_2, 1)=V_2$.
		\item[(2)]  let $f_t(x)=F(x, t)$. Then, $f_1|_{\partial V_1}: \partial V_1\to \partial V_1$ is a hyperelliptic involution in a mapping class group of a torus $\partial V_1$.
	\end{enumerate}
\end{lem}
\begin{proof}
	Let $M$ be the 3-sphere.
	The core of $V_1$ is the unknot in the 3-sphere.
	Since a mirror image of the unknot in the 3-sphere is ambient isotopic to itself, there is an ambient isotopy $F:M\times [0, 1]\to M$ such that $F(M, 1): M\to M$ is $\tau$ defined above.
	Hence, $M$ admits an ambient isotopy $F: M \times [0, 1]\to M$ that satisfies the assumption.
	
	Let $M$ be a lens space $L(p, q)$ with $p=2$.
	Let $\tau$ be a self-diffeomorphism of $M$ that is defined above.
	By Theorem \ref{Diff}, a diffeotopy group of $M$ is generated by $\sigma_-$.
	Hence, $\tau$ is isotopic to $\sigma_-$ or the identity.
	Since $\tau$ is an orientation-preserving involution and $\sigma_-$ is an orientation-reversing involution, $\tau$ is isotopic to the identity.
	This implies that there is an ambient isotopy $F: M \times [0, 1]\to M$ that satisfies the assumption.

	Next, we suppose that $M$ admits an ambient isotopy $F$.
	We note that $f_1: M\to M$ is isotopic to $\tau$.
	$\tau$ is not the identity in $\mathcal{D}(L(p, q))$ if $p\neq 2$ by Theorem \ref{Diff}.
	Hence, $M$ is the 3-sphere or a lens space $L(p, q)$ with $p=2$.

\end{proof}

\subsection{Handlebody decompositions of the 3-sphere and lens spaces}

	In this section, we classify the  handlebody decompositions of the 3-sphere and lens spaces.  The following proposition  is shown in the proof of Proposition 1 in \cite{G}.

	\begin{prop} \label{s3000}
		A type-$(0, 0, 0)$ decomposition of the 3-sphere has  exactly one branched locus and satisfies $F_{12} \cong F_{23} \cong F_{31} \cong D$.
	\end{prop}

Before proving Theorem \ref{s3001}, we prepare some lemmas. The following lemmas are applied not only to Theorem \ref{s3001} but also to Theorems \ref{lensclass011},  \ref{class111s3}, and \ref{class111}.
Lemma \ref{longi} can be obtained by observing the number of  intersections of a torus knot and the boundary of  a meridian disk of a solid torus. If the torus knot $T(p, q)$ satisfies $p=1$, there exists an  annulus embedded in a solid torus whose boundary is the core of a solid torus and $T(p, q)$.
\begin{lem}\label{longi}
Let $V$ be a solid torus and $C$ be a core of $V$. Then, the torus knot $T(p, q)$ in $\partial V$ is isotopic in $V$ to $C$ if and only if $p=\pm 1$
\end{lem}
%\begin{proof}
%		If $p=1$, there is an annulus embedded in $V$ its boundary is $T(p, q)$ and $C$.%アニュラスの存在をいった方が良い？
		
%		If $p\neq 1$, then $T(p, q)$ intersects meridian of solid torus $V$. Hence $T(p, q)$ is not isotopic to $C$ since there is a meridian disk of $V$ which intersects $C$ exactly once.
%\end{proof}
\begin{figure}[httb]
\centering
\includegraphics[scale=0.8]{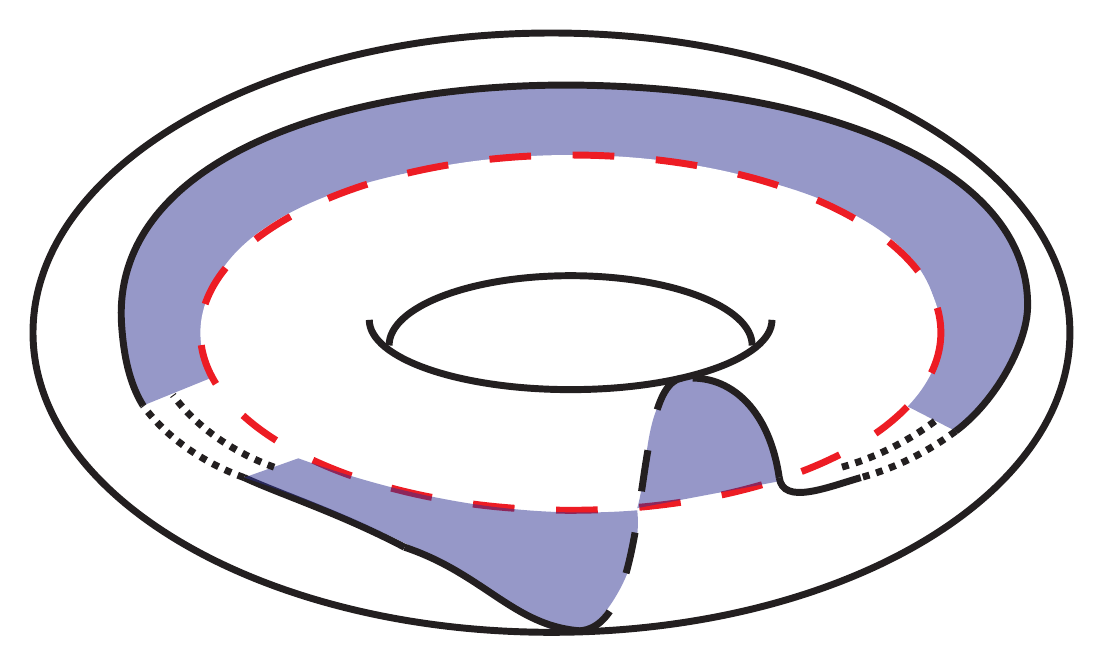}
\vspace{-3mm}
\caption{The annulus whose boundary is the core of a solid torus and a torus knot $T(1, q)$.}
\label{001decomposition}
\end{figure}

We use the following lemma to classify the handlebody decompositions of the 3-sphere and lens spaces. 
\begin{lem}\label{core}
	Let $M$ be the 3-sphere or a lens space $L(p, q)$ and $V_1\cup_\varphi V_2$ be a genus-one Heegaard splitting of $M$, where $\varphi: \partial V_2\to \partial V_1$ is an orientation-reversing homeomorphism. The core of $V_1$ is isotopic to the core of $V_2$ in $M$ if and only if $M$ is either the 3-sphere or $L(p, q)$ with \p.%mod p で1の方が良いかも
\end{lem}

\begin{proof}
	We consider the 3-sphere as  $L(1, 0)$.
	Since $\varphi$ is an orientation-reversing homeomorphism, $\varphi_\ast: H_1(\partial V_2)\to H_1(\partial V_1)$ is an isomorphism between the first homology group of a torus.
	Let $m_i$ and $l_i$ be a meridian and a longitude of $\partial V_i$, respectively.
	Then, the image of $m_2$ and $l_2$ under $\varphi_\ast$ can be written as follows:
			\[\varphi_\ast(m_2)=pl_1+qm_1,\]
			\[\varphi_\ast(l_2)=rl_1+sm_1,\]
	where $r$ and $s$ are integers. Since $\varphi$ is an orientation-reversing homeomorphism, $p$, $q$, $r$, and $s$ satisfy $ps-qr=-1$.
	By Lemma \ref{longi}, the core of $V_2$ is isotopic to a simple closed curve $c$ represented by $\left[c\right]=l_2+nm_2$, where $n$ is an integer.
	Then, the image of $c$ under $\varphi$ is a simple closed curve represented by $\varphi_\ast (\left[c\right])=\left[\varphi(c)\right]=(pn+r)l_1+(qn+s)m_1$.
	By Lemma \ref{longi}, $\varphi(c)$ is isotopic to the core of $V_1$ if and only if $|pn+r|=1$.
	
	Suppose that  the core of $V_1$ is isotopic to the core of $V_2$.
	Then, by Lemma \ref{longi}, we can obtain $|pn+r|=1$.
	From the equation $ps-qr=-1$, we obtain $r=(1+ps)/q$.
	Then, we obtain
	\[
		pn+\frac{1+ps}{q}=\pm 1.
	\]
	Hence,
	\[
		p(qn+s)=\pm q - 1.
	\]
	Then, we obtain
	\[
		p=\frac{-1\pm q}{qn+s},
	\]
	where $n$ and $s$ are integers. 
	This implies that $-1+ q$ or $-1-q$ must be a multiple of $p$.
	Since $0< q <p$ is satisfied, $-1+q$ is not a multiple of $p$.
	Hence, we can assume $-1-q$ that is a multiple of $p$.
	Since $0< q <p$ is satisfied, we obtain $-q\leq-1-q<0$.
	Since $-1-q$ is a multiple of $p$, $-1-q=-p$ is satisfied.
	Hence, we can obtain $q=p-1$.
	Since  lens spaces $L(p, q)$ and $L(p', q')$ are homeomorphic if and only if $p=p'$ and either $q\equiv\pm q'$ or $qq'\equiv \pm 1$ $(mod\, p)$, $L(p, p-1)$ is homeomorphic to $L(p, q)$ if and only if $(p-1)q\equiv \pm1$ $(mod\, p)$.
	Hence, if the core of $V_1$ is isotopic to the core of $V_2$, $L(p, q)$ satisfies $(p-1)q\equiv \pm1$ $(mod\, p)$. 
	
	Conversely, if $(p-1)q\equiv \pm1$ $(mod\, p)$, $L(p, q)\cong L(p, p-1)$.
	We can assume that $M=L(p, p-1)$.
	Suppose that $n=s-r$.
	Since $pn+r=p(s-r)+r=ps-(p-1)r=-1$, the core of $V_1$ is isotopic to the core of $V_2$ by Lemma \ref{longi}.
\end{proof}

Next, we show the following lemma.  We use this lemma to determine whether a 3-manifold with a given handlebody decomposition is a lens space.

\begin{lem}\label{b-1}
	Let $M=H_1\cup H_2\cup H_3$ be a type-$(g_1, g_2, 1; b)$ decomposition.
	If $F_{12}$ has a disk component $D$ such that  $\partial D$ is essential and not meridional in $\partial H_3$, then $M$ satisfies $M\cong M_1 \# M_2$, where $M_1$ has a type-$(g_1, g_2, 0; b-1)$ decomposition and $M_2$ is a lens space or $M$
	has a type-$(g_1, g_2, 0; b-1)$ decomposition.
\end{lem}
\begin{proof}
	Let $N(D)$ be a regular neighborhood of $D$ in $H_1\cup H_2$. 
	$N(D)$ is a 2-handle attached to $H_3$ along $N(\partial D)$, where $N(\partial D)$ is a regular neighborhood of  $\partial D$ in $\partial H_3$.
	Since $H_3$ is a solid torus and $\partial D$ is essential in $\partial H_3$, $L=N(D)\cup H_3$ is a punctured lens space or a 3-ball.
	Then, $M\cong M_1\# M_2$, where $M_2$ is  a capping off of $L$.
	$M_1$ is obtained by capping off $(H_1-N(D))\cup (H_2- N(D))$.
	Let us define $H_1'=H_1-N(D)$ and $H_2'=H_2- N(D)$.
	Then, $H_1'\cup H_2'\cup H_3'$ is a type-$(g_1, g_2, 0)$ decomposition of $M'$, where $H_3$ is a 3-ball.
	The number of branched loci of the type-$(g_1, g_2, 0)$ decomposition of $M'$ above is equal to the number of components of $\partial (H_1'\cap H_2' )$.
	 Since $\partial (H_1'\cap H_2' )=\partial F_{12} - \partial D$, the number of branched loci of a type-$(g_1, g_2, 0)$ decomposition of $M'$ is $b-1$.
\end{proof}

%================================S^3 type-(0,0,1)================================================
\setcounter{section}{2}
\setcounter{thm}{0}
We restate Theorem \ref{s3001}. To show this theorem, we shall take a look at  the proof of Proposition 2 in \cite{G} in detail and apply the lemmas presented above.
	\begin{thm}\label{s3001}
			A type-$(0,0,1)$ decomposition of the 3-sphere and a lens space $M$ has exactly 2 branched loci and it satisfies $F_{12} \cong D_{1} \cup D_{2}$ and $F_{23} \cong F_{31} \cong A$. 
			In fact, the 3-sphere and lens spaces have a type-$(0, 0, 1)$ decomposition that satisfies $F_{12} \cong D_{1} \cup D_{2}$ and $F_{23} \cong F_{31} \cong A$.
			
			Furthermore, there is a unique ambient isotopy class of the branched surface of the decomposition if $M$ is the 3-sphere or a lens space with \p. Otherwise,  there are exactly two ambient isotopy classes of the branched surface of the decomposition.
	\end{thm}
\setcounter{section}{4}

		\begin{proof}
			Let $M$ be the 3-sphere or a lens space that has a type-$(0,0,1; b)$ decomposition $H_1\cup H_2\cup H_3$. By Lemma \ref{lem1}, $\chi(F_{12})=2$ and  $\chi(F_{13})=\chi(F_{23})=0$.
			  By the condition $\chi(F_{12})=2$, $F_{12}$ has  at least two disk components. Let  $D_1$  and  $D_2$ be the disk components of $F_{12}$. If $\partial D_1$ is inessential in $\partial H_3$,  $M\cong M' \# S^2\times S^1$ by Lemma \ref{lem2}. This contradicts the assumption that $M$ is the 3-sphere or  a lens space. Then, we can assume that $\partial D_1$ is essential in $\partial H_3$. Similarly, $\partial D_2$ is essential in $\partial H_3$. Let $N(D_1)$ be a regular neighborhood of $\partial D_1$ in $H_1\cup H_2$. After attaching $N(D_1)$ to $H_3$ as a 2-handle along $N(\partial D_1)$ in $\partial H_3$, $N(D_1)\cup H_3$ will be a punctured lens space $L$. Then, $M\cong M''\# Cap( L )$, where $M''$ has a type-$(0,0,0; b-1)$ decomposition by Lemma \ref{b-1} and $Cap(L)$ is a capping off of $L$. 
				
 			Since $M$ is the 3-sphere or a lens space, $M''\cong S^3$ and $Cap(L)\cong M$. Since $M''$ has a type-$(0,0,0; b-1)$ decomposition and $M''\cong S^3$, $b-1=1$; hence, $b=2$ by Proposition \ref{s3000}. If $F_{12}$ has another component, then the number of branched loci is at least 2. This contradicts Proposition \ref{s3000}. This discussion implies that  $F_{12}$ has exactly two disks as its components.  Then, $Cl(\partial H_1-(D_1\cup D_2))\cong A$ and $Cl(\partial H_2-(D_1\cup D_2))\cong A$ are satisfied. This implies that $F_{13}\cong F_{23}\cong A$.
			
			Next, we shall show that the 3-sphere and lens space $M$ has a type-$(0, 0, 1)$ decomposition that satisfies $F_{12} \cong D_{1} \cup D_{2}$ and $F_{23} \cong F_{31} \cong A$.
			$M$ has a genus-one Heegaard splitting $V_1\cup V_2$.
			Let $H_1$ and $H_2$ be 3-balls.
			Suppose that $H_1\cap H_2\cong D_1\cup D_2$.
			Then, $H_1\cup H_2$ is a solid torus.
			Hence, $(H_1\cup H_2) \cup H_3$ is a genus-one Heegaard splitting, where $H_3$ is a solid torus.
			Therefore, $M$ has a type-$(0, 0, 1)$ decomposition that satisfies $F_{12} \cong D_{1} \cup D_{2}$ and $F_{23} \cong F_{31} \cong A$.
			
			Next, we classify the type-$(0, 0, 1)$ decomposition of $M$ as stated above up to ambient isotopy. To classify the  type-$(0, 0, 1)$ decomposition of $M$ , we have to prove only the following claim.
			
			\begin{claim}\label{uni001}
				A type-$(0, 0, 1)$ decomposition of the 3-sphere or lens space is unique up to ambient isotopy if $M$ is  $L(p, q)$ with \p. Otherwise, the decomposition of $M$ has exactly two isotopy classes.
			\end{claim}
			
			\begin{proof}[Proof of Claim \ref{uni001}]
				Let $H_1\cup H_2\cup H_3$ and $H_1'\cup H_2'\cup H_3'$ be the type-$(0, 0, 1)$ decompositions of the 3-sphere or lens space.
				Furthermore, we denote $F_{ij}=H_i\cap  H_j$ and $F'_{ij}=H'_i\cap  H'_j$.
				Then, $(H_1\cup H_2)\cup H_3$ and $(H_1'\cup H_2')\cup H_3'$ are genus-one Heegaard splittings.
				Since the Heegaard splittings of the 3-sphere and lens spaces are unique up to isotopy, $\partial H_3$ is isotopic to $\partial H_3'$.
				Let $V_1\cup V_2$ be a genus-one Heegaard splitting of the 3-sphere or lens space such that $V_1=H_3$.
				Then, $F_{12}$ has two disks properly embedded in a solid torus $V_2$ whose boundaries are essential in $\partial V_2$.
				Hence, the components of $F_{12}$ are meridian disks of $V_2$.
				Therefore, the embedding of $F_{12}$ into $V_2$ is unique up to ambient isotopy in $V_2$ since a meridian disk of a solid torus is unique up to isotopy.
				If $V_1=H_3'$, we can ambient isotope $F_{12}$ to $F_{12}'$ by  keeping $\partial H_3'=H_3$.
				Hence, $F_{12}\cup F_{13}\cup F_{23}$ is ambient isotopic to $F_{12}'\cup F_{13}'\cup F_{23}'$.
				If $V_2=H_3'$ and $M$ is homeomorphic to the 3-sphere or a lens space  $L(p, q)$ with  \p, we can isotope $H'_3$ to $H_3$ by Lemma \ref{core}.
				Hence, if $M$ is $L(p, q)$ with \p, $F_{12}\cup F_{13}\cup F_{23}$ is ambient isotopic to $F_{12}'\cup F_{13}'\cup F_{23}'$.
				On the other hand, if $V_2=H_3'$ and $M$ is not homeomorphic to $L(p, q)$ with \np, we cannot isotope $H_3'$ to $H_3$ by Lemma \ref{core}.
				Then, we cannot ambient isotope $F_{12}\cup F_{13}\cup F_{23}$ to $F_{12}'\cup F_{13}'\cup F_{23}'$.
				Hence, if $M$ is $L(p, q)$ with \np, the type-$(0, 1, 1)$ decomposition of a lens space has exactly two ambient isotopy classes. %such that  $H_3=V_1$ and $H_3=V_2$.
			\end{proof}
			By Claim \ref{uni001}, this completes the proof of Theorem  \ref{s3001}. 
		\end{proof}

\begin{figure}[httb]
\centering
\includegraphics[scale=0.3]{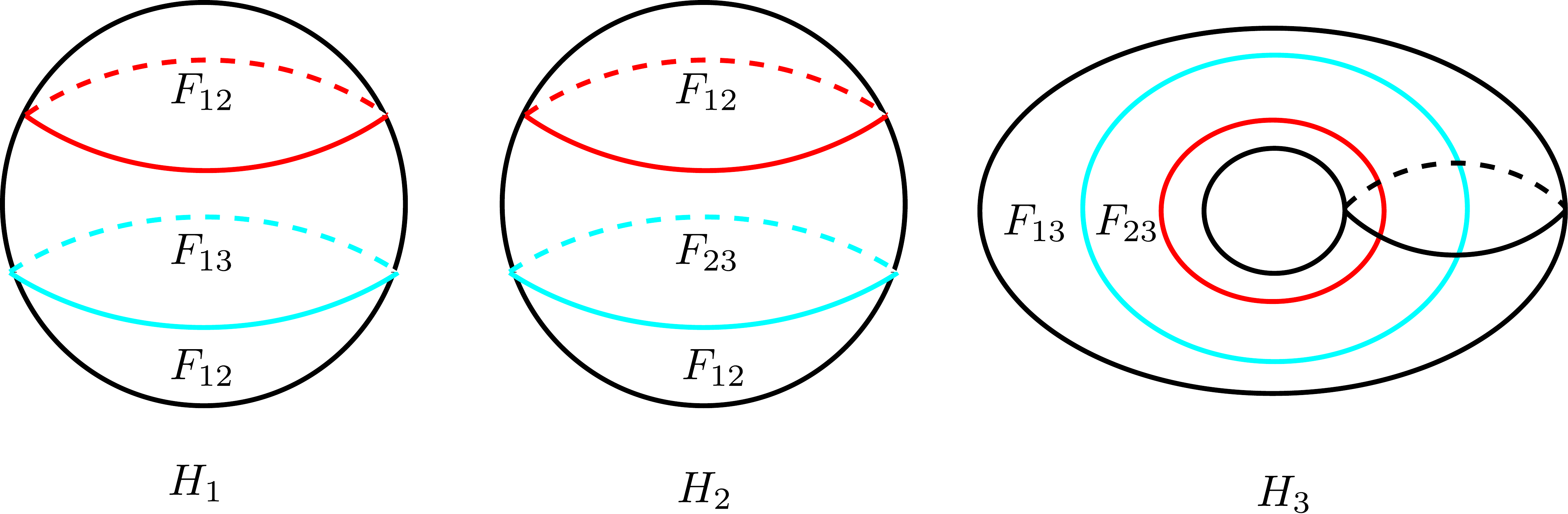}
\vspace{-3mm}
\caption{A type-$(0,0,1)$ decomposition of $S^{3}$ whose branched loci are two loops.}
\label{001decomposition}
\end{figure}

%================================S^3 type-(0,1,1)================================================

\setcounter{section}{2}

Before classifying the type-$(0, 1, 1)$ decompositions of the 3-sphere and lens spaces, we characterize them. We restate Theorem \ref{s3011}.
\begin{thm} \label{s3011}
A type-$(0,1,1)$ decomposition of the 3-sphere and a lens space  $M$ satisfies one of the following. 
			\begin{enumerate}
			\renewcommand{\theenumi}{(\arabic{enumi})}
				\item[(1)]  it has exactly one branched locus and satisfies $F_{12} \cong F_{31} \cong D$ and $F_{23} \cong T^{\circ}$ (Figure \ref{0111}).
				\item[(2)]  it has exactly three branched loci and satisfies $F_{12} \cong F_{31} \cong D \cup A $ and $F_{23} \cong P$ (Figure \ref{0112}).
			\end{enumerate}
			In fact, the 3-sphere and lens spaces admit each type-$(0, 1, 1)$ decomposition above.
			 %has ... . (branched surfaces are one loop or three loops.)
\end{thm}
\setcounter{section}{4}
\begin{proof}
Let $M$ be the 3-sphere or a lens space. Suppose that $M$ has a type-$(0, 1, 1; b)$ decomposition $H_{1} \cup H_{2} \cup H_{3}$, where $H_{1}$ is a 3-ball and $H_{2}$ and $H_3$ are  solid tori.
Then, $\chi(\partial H_{1})=2$ and $\chi(\partial H_{2})=\chi(\partial H_{3})=0$. 
By Lemma \ref{lem1}, $\chi(F_{12})=\chi(F_{31})=1$ and $\chi(F_{23})=-1$. 
Hence, each of $F_{12}$ and $F_{13}$ has at least one disk component.

\begin{claim}\label{011.a}
The surface $F_{23}$ has no disk component.
\end{claim}
\begin{proof}[Proof of Claim \ref{011.a}]
Suppose that $F_{23}$ has a disk component. Since  $\chi(F_{23})= -1$, $F_{23}$ has other components. Then, $M \ncong S^{3}$ and $M$ is not homeomorphic to a lens space by Lemma \ref{lem2}. This is a contradiction.
\end{proof}

\begin{claim}\label{011.b}
The number of disk components of  $F_{12}$ (resp.  $F_{13}$)  is one.
\end{claim}

\begin{proof}[Proof of Claim \ref{011.b}]
We show that the  number of disk components of  $F_{12}$ is one.
Suppose that $F_{12}$ has at least two disk components $D_{1}$ and $D_{2}$ .
By the condition $\chi(F_{12})=1$, $F_{12}$ has a planar surface component whose Euler number is less than $-1$. 
Then, the number of branched loci $b$ is greater than $4$.
If $\partial D_{1}$ is inessential in $\partial H_{2}$, then $M \ncong S^{3}$ by Lemma \ref{lem2}. This is a contradiction.
Hence, we can assume that $\partial D_{1}$ is essential in $\partial H_{3}$. 
Let $N(D_{1})$ be a regular neighborhood of $D_1$ in $H_{1} \cup H_{2}$.
Then, $N(D_1)\cup H_3$ is a punctured lens space $L$.
Hence, $M\cong M'\#Cap(L)$, where $M'$ has a type-$(0, 0, 1; b-1)$ decomposition by Lemma \ref{b-1}.
Since $b$ is greater than $4$, $b-1$ is greater than $3$.
This contradicts the assumption that the 3-sphere and lens spaces do not have a type-$(0, 0, 1; b)$ decomposition such that $b>2$ by Theorem \ref{s3001}.
Hence, the number of disk components of $F_{12}$ is one.
Similarly, we can show that the  number of disk components of  $F_{13}$ is one.
\end{proof}
\begin{claim}\label{011.c}
	The number of non-disk components of $F_{12}$ is at most one. If $F_{12}$ has a non-disk component, it is an annulus.
 \end{claim}
\begin{proof}[Proof of Claim \ref{011.c}]
	If $F_{12}$ has a planar surface component that is not an annulus, then $F_{12}$ has at least two disk components by $\chi(F_{12})=1$.
	This contradicts Claim \ref{011.b}.
	Then, if a component of $F_{12}$ is not a disk, it is an annulus.
	Suppose that  $F_{12}$ has at least two annuli components. 
	Then, the handlebody decomposition has at least 5 branched loci. 
	Let $D$ be a disk component of $F_{12}$ and $N(D)$ be a regular neighborhood of $D$ in $H_{1}\cup H_2$. 
	Then, $N(D)\cup H_3$ is a punctured lens space $L$.
	Hence, $M\cong M'\# Cap(L)$, where $M'$ has a type-$(0, 0, 1; b)$ decomposition with $b\geq4$ by Lemma \ref{b-1}.
	By Theorem \ref{s3001}, the 3-sphere and lens spaces do not admit a type-$(0, 0, 1; b)$ decomposition with $b\geq4$.
	Therefore, this contradicts the assumption.
\end{proof}
By Claims \ref{011.a}, \ref{011.b}, and \ref{011.c}, there are two cases $F_{12}\cong D^2$ and $F_{12}\cong D^2\cup A$.
If $F_{12}\cong D^2$, then $F_{13}\cong D^2$ by $ F_{13}\cong \partial H_1- F_{12}\cong D^2$.
Hence, $F_{23}\cong \partial H_2- F_{12}\cong T^\circ$.
Then, this decomposition corresponds to the decomposition of case (1).
Suppose that $F_{12}\cong D^2\cup A$. If  the core of an annulus component of $F_{12}$  is inessential in $\partial H_2$, then $F_{23}\cong A\cup T^\circ$ or $F_{23}\cong D \cup (T^\circ - D^2)$.
If $F_{23}\cong A\cup T^\circ$, $\partial D_{12}$ is inessential in $\partial H_3$, where $D_{12}$ is a disk component of $F_{12}$. 
Then, $M \ncong S^{3}$ and $M$ is not homeomorphic to a lens space by Lemma \ref{lem2}. This is a contradiction.
If $F_{23}\cong D \cup (T^\circ - D^2)$, this contradicts Claim \ref{011.a}.
Then, we can assume that the core of an annulus component of $F_{12}$ is essential in $\partial H_2$.
Hence, $F_{23}\cong \partial H_2- F_{12} \cong A- D^2\cong P$.
This decomposition corresponds to the decomposition of case (2).

			Next, we shall prove that $M$ admits a type-$(0, 1, 1)$ decomposition that satisfies cases (1) and (2).  
			 We can assume that $M$ has a genus-one Heegaard splitting $V_1\cup V_2$.
			 We take a point on a genus-one Heegaard surface. 
			 Let $N(p)$ be a regular neighborhood of $p$ in $M$.
			 Then, $N(p)\cup (V_1-N(p))\cup (V_2-N(p))$ is a type-$(0, 1, 1)$ decomposition that satisfies case (1). 
			 Hence, $M$ admits a type-$(0, 1, 1)$ decomposition that satisfies case (1). 
			 Let $D$ be a meridian disk of $V_1$.
			 Then, $V_1- N(D)$ is a 3-ball $B^3$, where $N(D)$ is a regular neighborhood of $D$ in $V_1$.
			 Let $A$ be an annulus in $B^3$ such that one of the components of $\partial A$ is parallel to $\partial D$ in $\partial V$.
			Then, $A$ cuts open $B^3$ into a 3-ball $H_1$ and a solid torus $H_2$ (see Figure \ref{ann011}).
			Hence, $D\cup A$ cuts open $V_1$ into $H_1$ and $H_2$. We denote $V_2=H_3$. Then, $H_1\cup H_2\cup H_3$ is a type-$(0, 1, 1)$ decomposition of $M$ that satisfies case (2).
\end{proof}

Before classifying the type-$(0, 1, 1)$ decompositions of the 3-sphere and lens spaces stated above, we present some lemmas. The following lemma shows that we can consider the type-$(0, 1, 1)$ decompositions of the 3-sphere and lens spaces as a genus-one Heegaard splitting. 
\begin{lem}\label{heegaard011lens}
	Let $H_1\cup H_2\cup H_3$ be a type-$(0, 1, 1)$ decomposition of the 3-sphere or a lens space  $M$  that satisfies either case (1) or case (2) of Theorem \ref{s3011}. Then, the following hold:
	\begin{enumerate}
	\renewcommand{\theenumi}{(\arabic{enumi})}
	 \item[(1)] if $H_1\cup H_2\cup H_3$ is a type-$(0, 1, 1)$ decomposition of $M$  that satisfies the conclusion of Theorem \ref{s3011} (1), both $\partial H_2$ and $\partial H_3$ are genus-one Heegaard surfaces of $M$.
	\item[(2)] if $M$ is a lens space and $H_1\cup H_2\cup H_3$ is a type-$(0, 1, 1)$ decomposition of  $M$  that satisfies the conclusion of Theorem \ref{s3011} (2), exactly either $\partial H_2$ or $\partial H_3$ is a genus-one Heegaard surface of $M$
	\item[(3)] if $M$ is the 3-sphere and $H_1\cup H_2\cup H_3$ is a type-$(0, 1, 1)$ decomposition of  $M$  that satisfies the conclusion of Theorem \ref{s3011} (2), both $\partial H_2$ and $\partial H_3$ are genus-one Heegaard surfaces of $M$.
	\end{enumerate}
\end{lem}
\begin{proof}
	First, we shall prove (1).
	Suppose that $H_1\cup H_2\cup H_3$ satisfies the conclusion of Theorem \ref{s3011} (1).
	Then, both $(H_1\cup H_2)\cup H_3$ and $(H_1\cup H_3)\cup H_2$ are genus-one Heegaard splittings of $M$.
	Hence, both $\partial H_2$ and $\partial H_3$ are genus-one Heegaard surfaces.
	
	Next, we shall prove (2).
	Suppose that $M$ is a lens space and $H_1\cup H_2\cup H_3$ is a type-$(0, 1, 1)$ decomposition of  $M$  that satisfies the conclusion of Theorem \ref{s3011} (2).
	Let $D_1$ be a disk component of $F_{12}$.
	If $\partial D_1$ is inessential in $\partial H_3$, $M$ has $S^2\times S^1$ as a connected summand by Lemma \ref{lem2}.
	This contradicts the assumption that $M$ is a lens space.
	Then, we can assume that $\partial D_1$ is essential in $\partial H_3$.
	If $N(\partial D_1)\cup H_3$ is a punctured $M$, then $\partial H_3$ is a genus-one Heegaard surface and $(H_1-N(D_1))\cup (H_2-N(D_1))$ is a 3-ball.
	Hence, if $\partial H_3$ is a Heegaard surface, there is a disk in $H_1-N(D_1)$ that intersects a meridian disk of $H_2$ exactly once; hence, $\partial H_2$ is not a genus-one Heegaard surface.
	Suppose that $N(\partial D_1)\cup H_3$ is a 3-ball.
	Then, $\partial H_3$ is not a genus-one Heegaard surface of $M$ and $(H_1-N(D_1))\cup (H_3\cup N(D_1))\cup (H_2-N(D_1))$ is a type-$(0, 0, 1)$ decomposition of $M$.
	Now, $((H_1-N(D_1))\cup (H_3\cup N(D_1)))$ is a solid torus by Theorem \ref{s3001}.
 	Hence, $((H_1-N(D_1))\cup (H_3\cup N(D_1))) \cup (H_2-N(D_1))$ is a Heegaard splitting.
	Therefore, $\partial H_2$ is a genus-one Heegaard surface of $M$ since $\partial H_2$  is isotopic to $(H_2-N(D))$.
		
	Finally, we shall prove (3).
	 Suppose that $M$ is the 3-sphere and $H_1\cup H_2\cup H_3$ is a type-$(0, 1, 1)$ decomposition of  $M$  that satisfies the conclusion of Theorem \ref{s3011} (2).
	Let $D_2$ be a disk component of $F_{12}$.
	If $\partial D_2$ is inessential in $\partial H_3$, $M$ has a $S^2\times S^1$ as a connected summand by Lemma \ref{lem2}.
	This contradicts the assumption that $M$ is the 3-sphere.
	Then, we can assume that $\partial D_2$ is essential in $\partial H_3$.
	If $N(\partial D_2)\cup H_3$ is a punctured $M$, then $\partial H_3$ is a genus-one Heegaard surface and $(H_1-N(D_2))\cup (H_2-N(D_2))$ is a 3-ball.
	Then, $(H_1-N(D_2))\cup (H_3\cup N(D_2)) \cup (H_2-N(D_2))$ is a type-$(0, 0, 1; 2)$ decomposition of $M$.
	Now, $((H_1-N(D_2))\cup (H_3\cup N(D_2)))$ is a solid torus by Theorem \ref{s3001}.
 	Hence, $((H_1-N(D_2))\cup (H_3\cup N(D_2))) \cup (H_2-N(D_2))$ is a Heegaard splitting.
	Therefore, $\partial H_2$ is a genus-one Heegaard surface of $M$ since $\partial H_2$  is isotopic to $(H_2-N(D))$.
	This completes the proof.
\end{proof}

If $H_1\cup H_2\cup H_3$ is a type-$(0, 1, 1)$ decomposition of the 3-sphere or a lens space, $H_1\cup H_2$ or $H_1\cup H_3$ is a solid torus of a genus-one Heegaard splitting of the 3-sphere or lens space by Lemma \ref{heegaard011lens}.
Hence, we shall consider  the embedding of $F_{23}$ or $F_{13}$ into a solid torus $H_1\cup  H_2$ or $H_1\cup H_3$, respectively.

\begin{lem}\label{ann}
	A properly embedded annulus $A$ in a 3-ball $B_1$ that cuts open $B_1$ into a solid torus $V$ and a 3-ball $B_2$ is unique up to ambient isotopy.
\end{lem}
\begin{proof}
	The 3-ball $B_2$  can be assumed as a 2-handle in a 3-ball $B_1$ attached to a solid torus $V$. 
	Let $\alpha$ be a core of a 2-handle $B_2$. 
	Then, the pair $(B_1, \alpha)$ is a 1-string free tangle.
	This is unique up to ambient isotopy in a 3-ball $B_1$. Then, its regular neighborhood is also unique up to ambient isotopy.
	Hence, an annulus $A$ is unique up to ambient isotopy.
\end{proof}
\begin{lem}\label{ann2}
	Let $V$ be a solid torus, $D$ be a meridian disk of $V$, and $A$ be an annulus properly embedded in $V$, which does not intersect $D$.
	Then, there are two embeddings of $D\cup A$ up to ambient isotopy if $D\cup A$ satisfies the following.
	\begin{enumerate}
	\renewcommand{\theenumi}{(\arabic{enumi})}
	\item[(1)] $D\cup A$ cuts open $V$ into a 3-ball $B$ and a solid torus $V'$ such that each of them does not have a self-intersection.
	\item[(2)] $\partial (D\cup A)$ cuts open $\partial V$ into an annulus $A'$, a disk $D'$, and a thrice-punctured sphere $P$.
	\end{enumerate}
	Furthermore, these ambient isotopy classes are taken to each other by hyperelliptic involution on $V$.
\end{lem}
\begin{proof}
		If each of the components of $\partial A$ is essential in $\partial V$,  it is meridional since $A$ does not intersect the meridian disk $D$.
		Suppose that each of the components of $\partial A$ is essential in $\partial V$.
		Then, $\partial (D\cup A)$ cuts open $\partial V$ into two annuli.
		This contradicts assumption (2).
		Suppose that each of the components of $\partial A$ is inessential.
		Then, $B$ intersects itself at $D$ since $\partial V'$ does not contain $D$.
		This contradicts assumption (1).
		Hence, we can assume that one of the components of $\partial A$ is inessential and the other is essential in $\partial V$.
		Since $D\cup A$ satisfies assumption (2), we can assume that the component of $\partial A$ that is inessential in $\partial V$ is $\partial D'$.
		Let $C$ be a component of $\partial A$ that is essential in $\partial V$.
		Then, $C$ is meridional in $\partial V$ since $A$ does not intersect $D$.
		Now, $C$ and $\partial D$ cut open $\partial V$ into two annuli $S_1$ and $S_2$.
		There are two cases: either  $D'\subset S_1$ or $D'\subset S_2$.
		
		Then, we shall show that a properly embedded annulus $A$ in $V$ that satisfies assumptions (1) and (2) as well as the assumption that $D'$ is in $S_1$ is unique up to isotopy.
		Since $D$ is unique up to ambient isotopy in $V$, we must consider the ambient isotopy class of $A$ in $V-N(D)$.
		Since $D$ is an intersection of $B$ and $V'$, $\partial D$ is inessential in $\partial V'$.
		Hence, $V'-N(D)$ is a solid torus and $B-N(D)$ is a 3-ball.
		Therefore, $A$ is a properly embedded annulus in a 3-ball $V-N(D)$ that cuts open $V-N(D)$ into a 3-ball $B-N(D)$ and a solid torus $V'-N(D)$.
		By Lemma \ref{ann}, $A$ is unique up to ambient isotopy in $V-N(D)$.
		Similarly, we can show that a properly embedded annulus  in $V$ that satisfies assumptions (1) and (2) as well as the assumption that $D'$ is in $S_2$ is unique up to isotopy.
		
		Figure \ref{ann011} (b) can be obtained by $\pi$-rotating Figure \ref{annuli_2} (a)  (see Figure \ref{ann011}).
		On the other hand, if they are ambient isotopic in $V$, there exists an ambient isotopy $F:V\times I\to V$ such that 
		\[
				F(S_1, 1)=S_2, F(S_2, 1)=S_1, F(\partial D, 1)=\partial D, F(C, 1)=C.
		\]
		We define  $f_t(x)=F(x, t)$. 
		Suppose that $f_1|_{\partial V}: \partial V\to \partial V$ is an identity in the mapping class group of a torus.
		Since $F(\partial D, 1)=\partial D$ and $F(C, 1)=C$,  an image of $S_1$ under $F(\partial V, 1)$ must be $S_1$.
		This contradicts the definition of $F$.
		Then, $f_1|_{\partial V}: \partial V\to \partial V$ is not an identity.
		Since the order of $f_1|_{\partial V}: \partial V\to \partial V$ in the mapping class group of a torus is 2, $f_1|_{\partial V}: \partial V\to \partial V$ is a hyperelliptic involution.
		Therefore, these are not isotopic in $V$ to each other since a hyperelliptic involution is not isotopic to the identity in the mapping class group of a torus.
		This completes the proof of the lemma (see Figure \ref{ann011}).
			\begin{figure}[httb]
			\centering
			\includegraphics[scale=0.8]{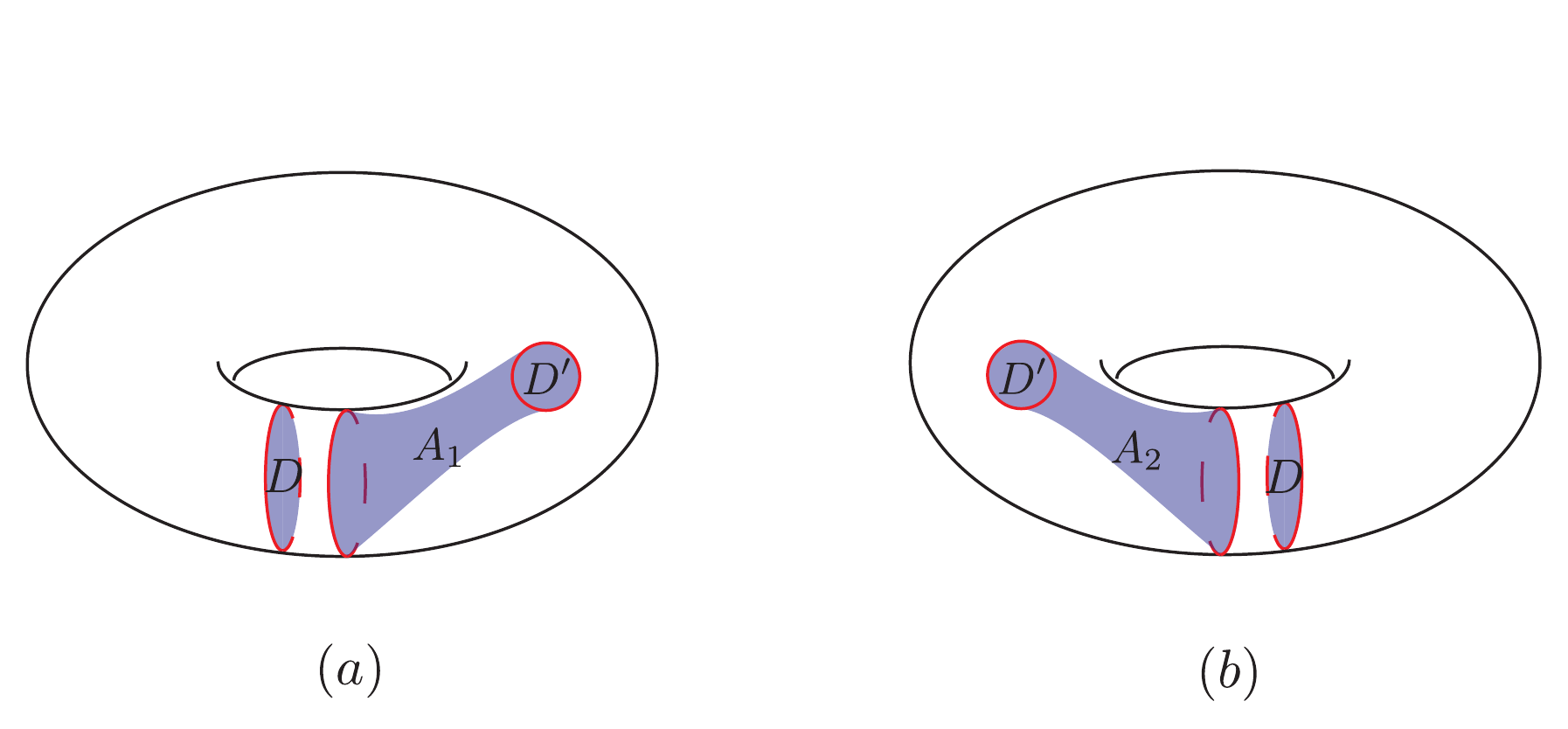}
			\vspace{-3mm}
			\caption{A properly embedded annulus and a disk in $V$ that satisfies the assumption in Lemma \ref{ann2}. (a) in a figure is taken to (b) by the hyperelliptic involution.}
			\label{ann011}
			\end{figure}
\end{proof}

%\begin{Rem}
%	Let $H_1\cup H_2\cup H_3$ be the type-$(0, 1, 1)$ decomposition of a lens space $M$ which satisfies the conclusion (2) of Theorem \ref{s3011}. 
%	If $\partial H_i$ is Heegaard surface, then $\partial H_j$  can not be a Heegaard surface of $M$ for $\{i, j\}=\{2, 3\}$.
%\end{Rem}
Now, we shall start classifying the type-$(0, 1, 1)$ decompositions of the 3-sphere and lens spaces. We can obtain Theorem  \ref{lensclass011} from Propositions \ref{011iso1}, \ref{011iso2}, and \ref{011iso3}.
\begin{prop}\label{011iso1}
A type-$(0, 1, 1)$ decomposition of the 3-sphere or a lens space that satisfies case (1) is unique up to ambient isotopy.
\end{prop}
\begin{proof}[Proof of Claim \ref{011iso1}]
	Let $H_1\cup H_2\cup H_3$ and $H_1'\cup H_2'\cup H_3'$ be type-$(0, 1, 1)$ decompositions of the 3-sphere or a lens space that satisfies case (1), where $H_1$ and $H_1'$ are homeomorphic to a 3-ball.
	Furthermore, we denote $F_{ij}=H_i\cap  H_j$ and $F'_{ij}=H'_i\cap  H'_j$.
	Then, $(H_1\cup H_2)\cup H_3$ and $(H_1'\cup H_2')\cup H_3'$ are genus-one Heegaard splittings.
	Since  genus-one Heegaard splittings of the 3-sphere and lens spaces are unique up to isotopy, $\partial H_3$ is isotopic to $\partial H_3'$.
	Let $V_1\cup V_2$ be a genus-one Heegaard splitting of the 3-sphere or a lens space such that $V_1=H_3$.
	Then, $F_{12}$ is a properly embedded disk  in $V_2$ such that its boundary is inessential in $\partial V_2$.
	If $V_1=H'_3$, $F'_{12}$ is a properly embedded disk in $V_2$ and $\partial F'_{12}$ is inessential in $\partial V_2$.
	Hence, $F_{12}$ is ambient isotopic to $F'_{12}$ in $V_2$.
	Therefore, $F_{12}\cup F_{13}\cup F_{23}$ is ambient  isotopic to $F'_{12}\cup F'_{13}\cup F'_{13}$.
	Next, we suppose that $V_2=H'_3$.
	Then, $F_{13}$ is isotopic to $F_{12}'$ in $V_1$.
	Since $F_{12}\cup F_{23}$ is also a Heegaard surface, $F_{12}\cup F_{23}$ is ambient isotopic to $\partial H_3'$. 
	Therefore, $F_{12}\cup F_{13}\cup F_{23}$ is ambient isotopic to $F'_{12}\cup F'_{13}\cup F'_{13}$.
\end{proof}
\begin{prop}\label{011iso2}
Let $M$ be the 3-sphere or a lens space with a genus-one Heegaard splitting of $V_1\cup V_2$.
If $M$ is the 3-sphere or a lens space $L(p, q)$ with $p=2$,  there is exactly one ambient isotopy class of the branched surface of the type-$(0, 1, 1)$ decomposition of conclusion (2) of Theorem \ref{s3011}.
 \end{prop}
 \begin{proof}
	Let $H_1\cup H_2\cup H_3$ and $H_1'\cup H_2'\cup H_3'$ be a type-$(0, 1, 1)$ decomposition of the 3-sphere or a lens space that satisfies case (2), where $H_1$ and $H_1'$ are homeomorphic to a 3-ball.
	Furthermore, we denote $F_{ij}=H_i\cap  H_j$ and $F'_{ij}=H'_i\cap  H'_j$.
	By Lemma \ref{heegaard011lens}, $\partial H_2$ or $\partial H_3$ is a Heegaard surface.
	Since $F_{12}\cong F_{13}$, we have to consider only the case that $\partial H_3$ is a Heegaard surface.
	Similarly, we suppose that $\partial H_3'$ is a Heegaard surface.
	Then, $\partial H_3$ and $\partial H_3'$ are genus-one Heegaard surfaces of $M$.
	Since Heegaard surfaces of the 3-sphere and lens space are unique up to isotopy, we can assume that $\partial H_3=\partial H_3'$.
	Suppose that $V_2=H_3$.
	Then, $F_{12}$ is properly embedded in $V_1$.
	$F_{12}$ satisfies $\partial V_1- \partial F_{12}\cong A\cup D \cup P$ and $F_{12}$ cuts open $V_1$ into a 3-ball $H_1$ and a solid torus $H_2$.
	If an annulus component $A_1$ of $F_{12}$ satisfies the condition that $\partial A_1$ is inessential in $\partial V$, a solid torus $H_2$ intersects $V_2=H_3$ at an annulus.
	This contradicts $H_2\cap H_3\cong P$.
	Hence, one of the components of $\partial A$ is meridional and the other is inessential in $\partial V_1$.
	Then, $F_{12}$ is the union of a disk and an annulus properly embedded in a solid torus $V_1$ that satisfies the assumption of Lemma \ref{ann2}.
	By Lemma \ref{diff}, $M$ admits an ambient isotopy $F: M\times [0, 1]\to M$ such that $F( V_1, 1)=V_1$ and $f_1|_{\partial V_1}:\partial V_1\to \partial V_1$ is a hyperelliptic involution in the mapping class group of a torus $\partial V_1$, where $f_1(x)=F(x, 1)$.
	Since $M$ admits an ambient isotopy $F: M\times [0, 1]\to M$ such that $F( V_1, 1)=V_1$ and $f_1|_{\partial V_1}:\partial V_1\to \partial V_1$ is a hyperelliptic involution in the mapping class group of a torus $\partial V_1$, the two ambient isotopy classes of the embedding $F_{12}$ in $V_1$ are ambient isotopic in $M$ by taking $F_{13}\cup F_{23}$ to itself.
	Hence, the embedding of $F_{12}\cup F_{13}\cup F_{23}$ is unique up to ambient isotopy in $M$.
	If $M$ is a lens space $L(p, q)$ with $p=2$, it satisfies \p.
	Hence, if $V_2=H'_3$, we can isotope $H'_3$ to $H_3$ by Lemma \ref{core}.
	Then,  $F_{12}\cup F_{13}\cup F_{23}$ is ambient isotopic to $F_{12}'\cup F_{13}'\cup F_{23}'$.
	This implies that the embedding of $F_{12}\cup F_{13}\cup F_{23}$ into $M$ is unique up to ambient isotopy.
\end{proof}

 \begin{prop}\label{011iso3}
 Let $M$ be  a lens space $L(p, q)$.
Suppose that $M$ satisfies $p\neq 2$. Then, there are exactly two isotopy classes  of the branched surface of the type-$(0, 1, 1)$ decomposition of conclusion (2) of Theorem \ref{s3011} if $M$ is homeomorphic to the 3-sphere or $L(p, q)$ with \p. Otherwise,   there are exactly  four isotopy classes of the branched surface of the type-$(0, 1, 1)$  decomposition of conclusion (2) of Theorem \ref{s3011}.
 \end{prop}
\begin{proof}
	Let $H_1\cup H_2\cup H_3$ and $H_1'\cup H_2'\cup H_3'$ be the type-$(0, 1, 1)$ decomposition of the 3-sphere or a lens space that satisfies case (2), where $H_1$ and $H_1'$ are homeomorphic to a 3-ball.
	Furthermore, we denote $F_{ij}=H_i\cap  H_j$ and $F'_{ij}=H'_i\cap  H'_j$.
	By Lemma \ref{heegaard011lens}, $\partial H_2$ or $\partial H_3$ is a Heegaard surface.
	If $\partial H_2$ (resp. $\partial H_3$) is a Heegaard surface, we shall consider the embedding of $F_{13}$( resp. $F_{12}$) into a solid torus.
	Since $F_{12}\cong F_{13}$, we have to consider only the case that $\partial H_3$ is a Heegaard surface.
	Similarly, we suppose that $\partial H_3'$ is a Heegaard surface.
	Then, $\partial H_3$ and $\partial H_3'$ are genus-one Heegaard surfaces of $M$.
	Since Heegaard surfaces of the 3-sphere and lens space are unique up to isotopy, we can assume that $\partial H_3=\partial H_3'$.
	Let $V_1\cup V_2$ be a genus-one Heegaard splitting of $M$. 
	Then, we can suppose that $V_1=H_3$.
	Then, $F_{12}$ is embedded in $V_2$.
	$F_{12}$ satisfies $\partial V_2- \partial F_{12}\cong A\cup D \cup P$ and $F_{12}$ cuts open $V_2$ into a 3-ball $H_1$ and a solid torus $H_2$.
	If an annulus component $A_1$ of $F_{12}$ satisfies the condition that $\partial A_1$ is inessential in $\partial V$, a solid torus $H_2$ intersects $V_1=H_3$ at an annulus.
	This contradicts $H_2\cap H_3\cong P$.
	Hence, one of the components of $\partial A$ is meridional and the other is inessential in $\partial V_2$.
	Then, $F_{12}$ is the union of a disk and an annulus properly embedded in a solid torus $V_2$ that satisfies the assumption of Lemma \ref{ann2}.
	By Lemma \ref{diff}, $M$ does not admit an ambient isotopy $F: M\times [0, 1]\to M$ such that $F( V_2, 2)=V_2$ and $f_1|_{\partial V_2}:\partial V_2\to \partial V_2$ is a hyperelliptic involution in the mapping class group of a torus $\partial V_2$, where $f_1(x)=F(x, 1)$.
	Hence, the two ambient isotopy classes of the embedding $F_{12}$ in $V_2$ cannot be ambient isotopic in $M$ by taking  $F_{13}\cup F_{23}$ to itself.
	Then, if $F_{12}'$ is not isotopic to $F_{12}$ in $V_2$, they are not isotopic to each other in $M$ by taking $F_{13}\cup F_{23}$ (resp. $F_{13}'\cup F_{23}'$) to itself.
	If $V_2=H'_3$ and $M$ is $L(p, q)$ with \p, we can isotope $H'_3$ to $H_3$ by Lemma \ref{core}.
	Hence, if $M$ is  $L(p, q)$ with \p, a type-$(0, 1, 1)$ decomposition of the 3-sphere or a lens space that satisfies case (2) has two isotopy classes.
	If $V_2=H'_3$ and $M$ is $L(p, q)$ with \np, we cannot isotope $H'_3$ to $H_3$ by Lemma \ref{core}.
	Hence, if $M$ is $L(p, q)$ with \np, a type-$(0, 1, 1)$ decomposition of the 3-sphere or a lens space that satisfies case (2) has four isotopy classes.
\end{proof}

By Propositions \ref{011iso1}, \ref{011iso2}, and \ref{011iso3}, we can obtain Theorem  \ref{lensclass011}. We restate  Theorems  \ref{lensclass011}.
\setcounter{section}{2}
\begin{thm}\label{lensclass011}
Let $M$ be the 3-sphere or a lens space.
If $M$ is the 3-sphere or a lens space $L(p, q)$ with $p=2$,  the type-$(0, 1, 1)$ decompositions of $M$ can be classified up to ambient isotopy as follows.
\begin{enumerate}
\renewcommand{\theenumi}{(\arabic{enumi})}
\item[(1)]  there is exactly one ambient isotopy class of the branched surface of the type-$(0, 1, 1)$ decomposition of conclusion (1) of Theorem \ref{s3011}.
\item[(2)]  there is exactly one ambient isotopy class of the branched surface of the type-$(0, 1, 1)$ decomposition of conclusion (2) of Theorem \ref{s3011}.
\end{enumerate}
On the other hand, if $M$ is a lens space $L(p, q)$ with $p\neq 2$, the type-$(0, 1, 1)$ decompositions of $M$ can be classified up to ambient isotopy as follows.
\begin{enumerate}
\item[(3)]   there is exactly one  isotopy class of the branched surface of a type-$(0, 1, 1)$ decomposition of conclusion (1) of Theorem \ref{s3011}.
\item[(4)]  there are exactly two isotopy classes  of the branched surface of the type-$(0, 1, 1)$ decomposition of conclusion (2) of Theorem \ref{s3011} if $M$ is $L(p, q)$ with \p. Otherwise,  there are exactly  four isotopy classes of the branched surface of the type-$(0, 1, 1)$  decomposition of conclusion (2) of Theorem \ref{s3011}.
\end{enumerate}
\end{thm}

\setcounter{section}{4}
\begin{figure}[httb]
\centering
\includegraphics[scale=0.3]{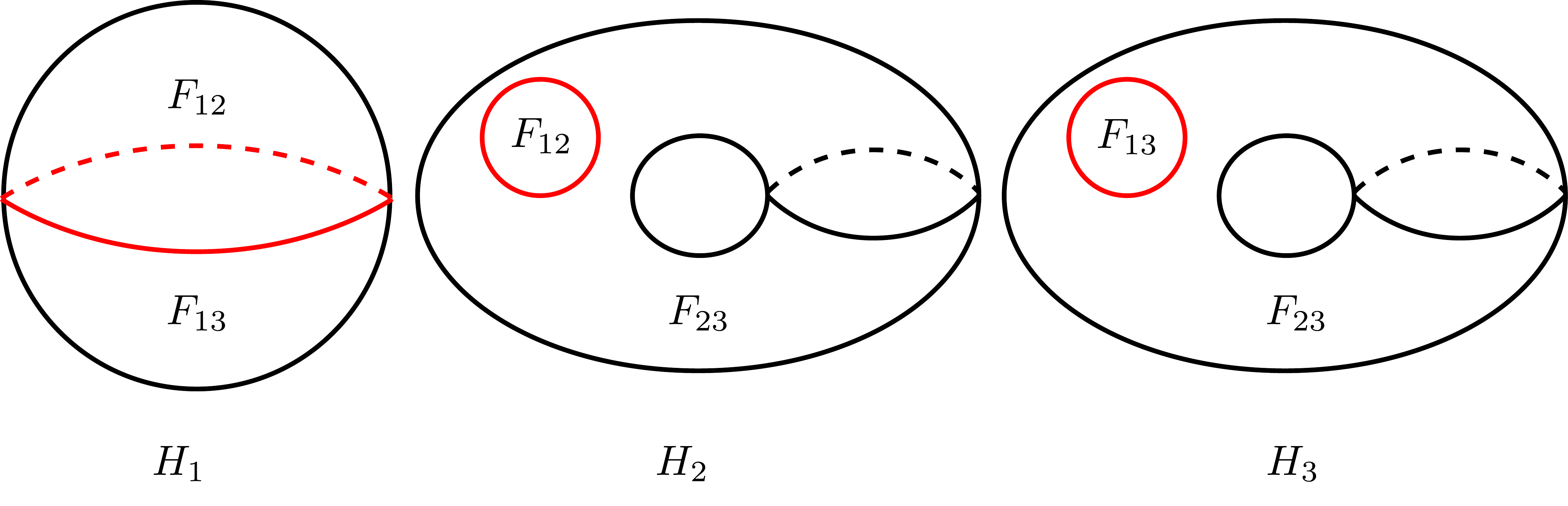}
\vspace{-3mm}
\caption{A type-$(0,1,1)$ decomposition whose branched locus is one loop.}
\label{0111}
\end{figure}
%\end{figure}
%\begin{figure}
\begin{figure}[httb]
\centering
\includegraphics[scale=0.3]{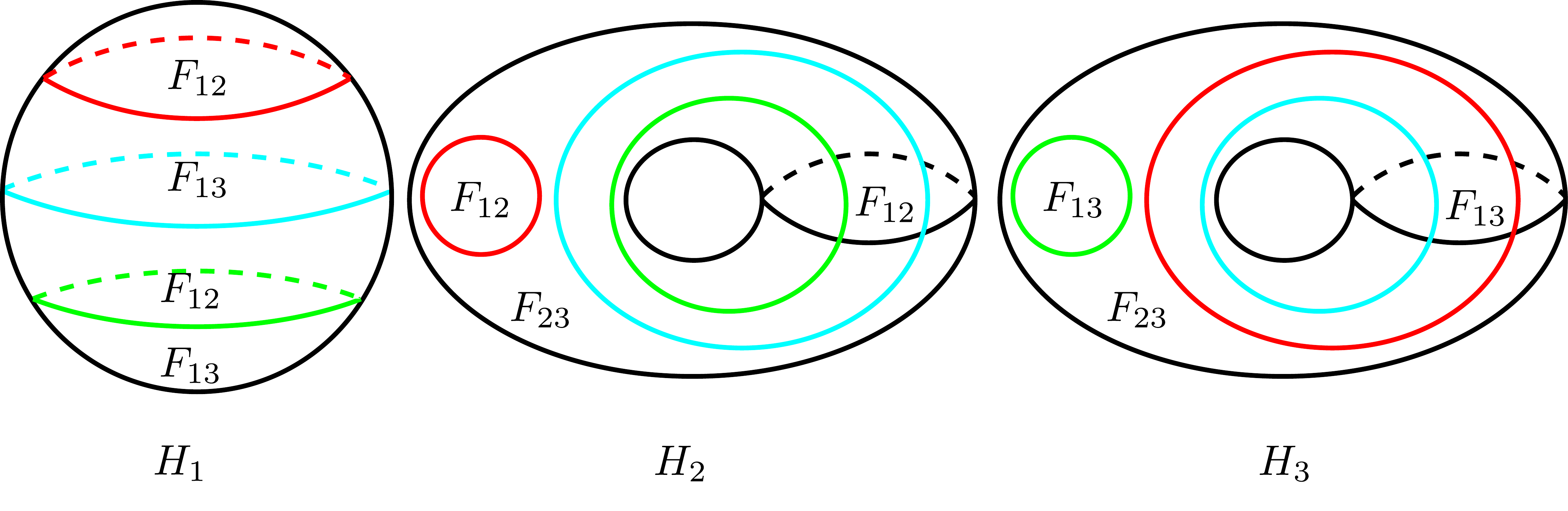}
\vspace{-3mm}
\caption{A type-$(0,1,1)$ decomposition whose branched loci are three loops.}
\label{0112}
\end{figure}

		Finally, we consider the type-$(1, 1, 1)$ decompositions of the 3-sphere and  lens spaces. To characterize a decomposition, we prepare two propositions about the Seifert fibered structure of lens spaces.
				\setcounter{section}{4}
		 \begin{prop}[Lemma 4.1 of \cite{S2}]\label{Seifert1}
       		The base space of a Seifert fibration of any lens space is $S^2$ or $\mathbb{R}P^2$. If the base space is $S^2$, there are at most two singular fibers; if the base space is $\mathbb{R}P^2$, there is no singular fiber.
		
       \end{prop}
       In this paper, two  3-manifolds are homeomorphic if  there exists an orientation-preserving or orientation-reversing homeomorphism between them. Then, Proposition 4.2 of \cite{S2} can be written as follows. 
       \begin{prop}[Proposition 4.2 of \cite{S2}]\label{Seifert2}
       		A lens space that fibers over $\mathbb{R}P^2$ is homeomorphic to $L(4, 1)$. Each of these lens spaces  admits a unique Seifert fibered structure with base space $\mathbb{R}P^2$.
       \end{prop}
	\setcounter{section}{2}
	\setcounter{thm}{3}
	We restate  Theorem \ref{thm 3}.
	\begin{thm}
		A type-$(1,1,1)$ decomposition of the 3-sphere and a lens space $M$ satisfies one of the following, where $\{i, j, k\}=\{1, 2, 3\}$.
			\begin{enumerate}
			\renewcommand{\theenumi}{(\arabic{enumi})}
				\item [(1)] 
					it has exactly two branched loci and satisfies $F_{ij} \cong F_{jk} \cong F_{ki} \cong A$.
					%(Figure \ref{111deco1}).
				\item[(2)]  it has exactly two branched loci and satisfies
					$F_{ij}\cong D \cup T^{\circ}$ and $F_{jk} \cong F_{ki} \cong A $.
					%They are longitudal in $\partial T_1$ and inessential in $\partial T_2$ and $\partial T_3$.
					%(Figure \ref{111deco2}).
				\item[(3)]  it has exactly four branched loci and satisfies
					$F_{ij} \cong D \cup P$ and $F_{jk} \cong F_{ki} \cong A_{1} \cup A_{2}$.
					%They are longitudal in $\partial T_1$, two are longitudal and the others are inessential and parallel in $\partial T_2$ and $T_3$.
					%(Figure \ref{111deco3}).
				\item[(4)]  it has exactly four branched loci and satisfies
					$F_{ij} \cong A_{1} \cup A_{2}$ and $F_{jk} \cong F_{ki} \cong D \cup P$.
					%, two of them are longitudal and the others are inessential in each $\partial T_i$.
					%Two longitudal loops are parallel in $\partial T_1$ and $\partial T_2$, and not parallel in  $\partial T_3$.
					%(Figure \ref{111deco4}).
				\item[(5)]  it has exactly four branched loci and satisfies
					$F_{ij} \cong F_{jk} \cong F_{ki} \cong D \cup P$.
					%, two of them are longitudal and the others are inessential in each $\partial T_i$.
					%Any pair of them are not parallel in  $\partial T_i$'s.
					%(Figure \ref{111deco5}).
				\end{enumerate}
			Furthermore, the following holds. 
				\begin{enumerate}
				 \item[(6)] $M \cong L(4,1)$. In this case,  $M$ also  has a decomposition that satisfies $F_{12}\cong F_{13}\cong F_{23} \cong A\cup A$ and has four branched loci.
		\end{enumerate}
		In fact, if $M$ is not homeomorphic to $L(4, 1)$, $M$ admits type-$(1, 1, 1)$ decompositions that satisfy (1), (2), (3), (4), and (5) above. In addition, if  $M\cong L(4, 1)$, $M$ admits type-$(1, 1, 1)$ decompositions that satisfy  (1), (2), (3), (4), (5), and (6) above.
		\end{thm}
	
	\setcounter{section}{4}
			\begin{proof}
		Let $M$ be the 3-sphere or a lens space with a type-$(1, 1, 1;  b)$ decomposition $H_1 \cup H_2\cup H_3$. By Lemma \ref{lem1}, we shall obtain
 				\[
 					\chi(F_{12})=0\quad, \quad \chi(F_{13})=0\quad, \quad \chi(F_{23})=0.
 				\]
		\begin{claim}\label{disk one}
				The surface $F_{ij}$ has at most one disk component. 
		\end{claim}
			\begin{proof}[Proof of Claim 4.3]
				We show the case where $F_{ij}=F_{12}$.
				Suppose that $F_{12}$ has at least two disk components. 
				Since $\chi(F_{12})=0$,  $F_{12}$ has a component that is not a disk.
				Let $D_1$ and $D_2$ be the disk components of $F_{12}$ and $N(D_i)$ be a regular neighborhood of $D_i$ for $i=1, 2$.
				Since $H_3\cup N(D_1)$ is a punctured lens space $L$, $M\cong M'\# Cap(L)$, where $Cap(L)$ is a capping off of $L$.
				The connected summand $M'$ has a type-$(0, 1, 1; b-1)$ decomposition $H_1'\cup H_2'\cup H_3'$, where $H_1'=Cl( H_1-N(D_1))$, $H_2'=Cl(H_2-N(D_1))$ and $H_3'=B^3$ by Lemma \ref{b-1}.
				Since $F_{12}'=H_1'\cup H_2'$ is $F_{12}-D_1$, $F_{12}'$ has a disk component $D_2$; in addition, $F_{12}'$ has at least two components. 
				Since $H_3'$ is a 3-ball, $\partial D_2$ is inessential in $\partial H_3'$. By Lemma \ref{lem2}, $M'\cong S^2\times S^1\# M''$. 
				This contradicts the assumption that $M$ is the 3-sphere or a lens space.	Similarly, $F_{13}$ and $F_{23}$ have at most one disk component.		
			\end{proof}
			
			\begin{claim}\label{claim1.1}
 				If $F_{12}$ has a disk component, each of $F_{13}$ and $F_{23}$ is homeomorphic to   $A$ or $D^2\cup P$ or $A\cup A$.
			\end{claim}
				\begin{proof}[Proof of Claim \ref{claim1.1}]
					If  $b\geq 5$ and $F_{12}$ has a disk component, then $M\cong M' \# \mathbb{L}$, 
					where $M'$ has a type-$(0, 1, 1; 4)$ decomposition.
					Since the 3-sphere and lens spaces do not have a type-$(0, 1, 1; 4)$ decomposition, 
					this contradicts the assumption that  $M$ is the 3-sphere or a lens space.
					Suppose that $b\leq 4$ and $F_{12}$ has only one disk component.
					Since $\chi(F_{12})=0$, $F_{12}$ satisfies one of the following.
					\begin{enumerate}
					\renewcommand{\theenumi}{(\arabic{enumi})}
						\item[(1)]  $F_{12}\cong D^2 \cup T^{\circ} \cup A$.
						\item[(2)]  $F_{12}\cong D^2 \cup P$.
						\item[(3)]  $F_{12}\cong D^2 \cup T^{\circ}$.
					\end{enumerate}
					
					Suppose that $F_{12}\cong D^2 \cup T^{\circ} \cup A$.
					Let $D$ be a disk component of $F_{12}$ and $N(D)$ be a regular neighborhood of $D$.
					Then, $N(D)\cup H_3$ is a punctured lens space $L$ and $M\cong M'\# Cap(L)$.
					We define $H_1'$, $H_2'$, and $H_3'$ as $H_1'=Cl(H_1-N(D))$, $H_2'=Cl(H_2-N(D))$, and $H_3=B^3$.
					 Then, $H_1'\cup H_2'\cup H_3$ is a type-$(0, 1, 1; 3)$ decomposition of $M'$ by Lemma \ref{b-1}.
					 Further, $F_{13}'=H_1'\cap H_3' \cong D^2\cup A$ is satisfied. 
					 Let $D'$ be a disk component of $F_{13}'$.
					 Since the branch locus that is a boundary of a disk component in $\partial H_2$ is inessential in $\partial H_2$, $\partial D'$ is also inessential in $\partial H_2$.
					 By Lemma \ref{lem2}, $M'$ has  $S^2\times S^1$ as a connected summand. 
					 This contradicts the assumption that $M$ is the 3-sphere or a lens space.
					 Hence, $F_{12}$ is homeomorphic to  $D^2 \cup P$ or $D^2 \cup T^{\circ}$.
					 
					 If $F_{12}\cong T^\circ \cup D^2$, then $F_{13}=\partial H_1- (T^\circ \cup D^2)\cong A$ and $F_{23}=\partial H_2-(T^\circ \cup D^2)\cong A$.
					 Suppose that $F_{12}$ is homeomorphic to  $F_{12}\cong D^2 \cup P$.
					 If $\partial P$ is inessential in $\partial H_1$ or $\partial H_2$, there are two disk components in $F_{12}$, $F_{13}$, or $F_{23}$.
					 This contradicts Claim \ref{disk one}.
					Hence, one of the components of $\partial P$ is essential. 
					Therefore, two of the components $\partial P$ are essential in $\partial H_1$ and $\partial H_2$.
					Then, $Cl(\partial H_1 - P)$ is homeomorphic to $D^2\cup A$.
					Let $D_1$ be a disk component of $Cl(\partial H_1 - P)$ and $A_1$ be an annulus component of $Cl(\partial H_1 - P)$.
					If $D$ is included in $D_1$, $F_{13}\cong A\cup A$, where $D$ is a disk component of $F_{12}$. 
					On the other hand, if $D$ is included in $A_1$, $F_{13}\cong D^2 \cup P$, where $D$ is a disk component of $F_{12}$.
					Similarly, $F_{23}$ is homeomorphic to $A\cup A$ or $D^2\cup P$.
				\end{proof}
				\begin{claim}\label{claim2.1}
		 		Suppose that $F_{12}$ does not have a disk component.  Then, there are at most two annuli components in $F_{12}$.
		 \end{claim}
		 	\begin{proof}[Proof of Claim \ref{claim2.1}]
					We note that if $F_{12}$ does not have a disk component, each of the component $F_{12}$ is annulus since $\chi(F_{ij})=0$.
					We can assume that the core of $F_{ij}$ is essential in the boundary of each handlebodies and $F_{12}$ has at least three annulus components.
					Then, $F_{13}$  and $F_{23}$ consist of annuli. 
					We note that the number of components of $F_{ij}$ are equal. 
					
					If $\partial F_{ij}$ is not meridional in $\partial H_k$ for $k=1, 2, 3$, 
					then $M$ is a Seifert manifold with at most three singular fibers. 
					Since $M$ is the 3-sphere or a lens space, the base space is  $S^2$ or $\mathbb{R}P^2$ by Proposition \ref{Seifert1}. Then, the intersection of two fibered tori is exactly one annulus or exactly two annuli. 
					This contradicts the assumption.  
					
					Hence, we can assume that  $\partial F_{12}$ is meridional in $\partial H_1$. 
					Let $D$ be a meridian disk of $H_1$ such that  $\partial D \subset F_{12}$. 
					Let $N(D)$ be a regular neighborhood of $D$ in $H_1$. 
					The union $N(D)\cup H_2$ is a punctured lens space or a 3-ball $L$. 
					Then, $M\cong M' \# Cap(L)$, where $M'$ is a 3-manifold that has a type-$(0, 0, 1; b)$ decomposition with $b\geq 5$. 
					By  the condition $b\geq 5$ and  Proposition \ref{s3011}, $M'$ is not homeomorphic to the 3-sphere or a lens space.
					This contradicts the assumption that $M$ is the 3-sphere or a lens space.
					This completes the proof of Claim \ref{claim2.1}.
			\end{proof}
				By Claims \ref{claim1.1} and \ref{claim2.1}, we have the following cases after renaming the subscripts of the handlebodies.
		\begin{enumerate}
		\renewcommand{\theenumi}{(\arabic{enumi})}
				\item[(1)]  
					it has exactly two branched loci and satisfies $F_{ij} \cong F_{jk} \cong F_{ki} \cong A$.
					%(Figure \ref{111deco1}).
				\item[(2)]  It has exactly two branched loci and satisfies
					$F_{ij}\cong D \cup T^{\circ}$ and $F_{jk} \cong F_{ki} \cong A $.
					%They are longitudal in $\partial T_1$ and inessential in $\partial T_2$ and $\partial T_3$.
					%(Figure \ref{111deco2}).
				\item[(3)]  it has exactly four branched loci and satisfies
					$F_{ij} \cong D \cup P$ and $F_{jk} \cong F_{ki} \cong A_{1} \cup A_{2}$.
					%They are longitudal in $\partial T_1$, two are longitudal and the others are inessential and parallel in $\partial T_2$ and $T_3$.
					%(Figure \ref{111deco3}).
				\item[(4)]  it has exactly four branched loci and satisfies
					$F_{ij} \cong A_{1} \cup A_{2}$ and $F_{jk} \cong F_{ki} \cong D \cup P$.
					%, two of them are longitudal and the others are inessential in each $\partial T_i$.
					%Two longitudal loops are parallel in $\partial T_1$ and $\partial T_2$, and not parallel in  $\partial T_3$.
					%(Figure \ref{111deco4}).
				\item[(5)]  it has exactly four branched loci and satisfies
					$F_{ij} \cong F_{jk} \cong F_{ki} \cong D \cup P$.
					%, two of them are longitudal and the others are inessential in each $\partial T_i$.
					%Any pair of them are not parallel in  $\partial T_i$'s.
					%(Figure \ref{111deco5}).
				\item[(6)]  $M \cong L(4,1)$. In this case,  $M$ also  has a decomposition that satisfies $F_{12}\cong F_{13}\cong F_{23} \cong A\cup A$ and it has four branched loci.
				\end{enumerate}
			We complete the proof of the first half of the statement.
			%We will show lens space has a decomposition which satisfies an above condition.
			Hereafter, we shall show that $M$ has a type-$(1, 1, 1)$ decomposition that satisfies cases  (2), (3), (4), or  (5), and  if $M\cong L(4, 1)$, $M$ has another decomposition that satisfies case (6). 
			
			\begin{claim}\label{claim3.1}
			  the 3-sphere and any lens spaces have a type-$(1, 1, 1)$ decomposition that satisfies  cases  (2), (3), (4), and (5).
			  \end{claim}
			  \begin{proof}[Proof of Claim \ref{claim3.1}]
				Let $M$ be the 3-sphere and a lens space.
				Then, $M$ has a genus-one Heegaard splitting $V_1\cup V_2$.
				Let $D$ be a meridian disk of $V_2$.
				Then, we can construct a  punctured torus $T^\circ$ properly embedded  in a 3-ball  $V_1-N(D)$, which cuts open $V_1-N(D)$ into two solid tori $H_1$ and $H_2$ (see Figure \ref{punctorus}).
				Then, $H_1\cup H_2\cup V_1$ is a type-$(1, 1, 1)$ decomposition that satisfies case (2).
				
				We can construct $D^2\cup P$ embedded in $V_2$ such that $D^2\cup P$ cuts open $V_2$ into two solid tori $H_1$ and $H_2$ and $\partial (D^2\cup P)$ cuts open $\partial V_2$ into four annuli (see Figure \ref{punc3eps}).
				Then, $H_1\cup H_2\cup V_1$ is a type-$(1, 1, 1)$ decomposition that satisfies case (3).
				
				In addition, we can construct $D^2\cup P$ embedded in $V_2$ such that $D^2\cup P$ cuts open $V_2$ into two solid tori $H_1$ and $H_2$ and $\partial (D^2\cup P)$ cuts open $\partial V_2$ into two annuli, a disk, and a thrice-punctured sphere (see Figure \ref{punc_1eps}).
				Then, $H_1\cup H_2\cup V_1$ is a type-$(1, 1, 1)$ decomposition that satisfies case (4).
				
				Furthermore, we can construct $D^2\cup P$ embedded in $V_2$ such that $D^2\cup P$ cuts open $V_2$ into two solid tori $H_1$ and $H_2$ and $\partial (D^2\cup P)$ cuts open $\partial V_2$ into two disks and two thrice-punctured spheres (see Figure \ref{punc4eps}).
				Then, $H_1\cup H_2\cup V_1$ is a type-$(1, 1, 1)$ decomposition that satisfies case (5).						
			\end{proof}
			%Also, all of type-$(0, 1, 1; 3)$ decompositions of $M'$ are obtained by  above discussion satisfy $F_{12}\cong F_{13}\cong D^2\cup A$ and $F_{23}\cong P$.
			%This case corresponds the cases (1), (2) and (4).
			%The 3-sphere have a type-$(0, 1, 1; 1)$ decomposition and a type-$(0, 1, 1; 3)$ decomposition which satisfy those requirements.
			%This discussion implies that the 3-sphere and lens spaces has a type-$(1, 1, 1)$ decomposition which satisfies any of the case (1), (2), (4) or (5).
			\begin{claim}\label{claim4.1}
			 $M$ has a type-$(1, 1, 1)$ decomposition that satisfies case (1). A type-$(1, 1, 1)$ decomposition that satisfies case (1) has two cases: one is  the case where branched loci are meridional in the boundary of a handlebody and the other is the case where branched loci are not meridional in the boundary of a handlebody. 
			\end{claim}
			\begin{proof}
			$M$ has a Seifert fibered structure over $S^2$ at most two singular fiber by Proposition \ref{Seifert1}.
			 Let $H_i$ be a fibered solid torus  for $i=1, 2, 3$. 
			 Then, we can construct the Seifert fibered structure over $S^2$ by attaching $H_i$ to each other along the annuli. 
			 A component of the intersection of $H_i\cap H_j$ is an annulus since the base space is $S^2$.
			 Then, $H_1\cup H_2\cup H_3$ is a type-$(1, 1, 1)$ decomposition of $M$ that satisfies case (1).
			 Hence, $M$ has a type-$(1, 1, 1)$ decomposition that satisfies $F_{ij}\cong A$. 
			\end{proof}
						
			By Claims \ref{claim3.1} and \ref{claim4.1}, $M$ has a type-$(1, 1, 1)$ decomposition that satisfies cases (1), (2), (3), (4), and (5).
			
			 Finally, we show that $M$ has a type-$(1, 1, 1)$ decomposition that satisfies case (6) if $M$ is homeomorphic to $L(4, 1)$. 
			 Let $M$ be a lens space $L(4, 1)$.
			 Then, $M$ has a Seifert fibered structure over $\mathbb{R}P^2$ with no singular fiber by Proposition \ref{Seifert2}.
			 Let $H_i$ be a fibered solid torus; its core is a regular fiber for $i=1, 2, 3$. 
			 Then, we can construct the Seifert fibered structure over $\mathbb{R}P^2$ with no singular fiber.
			 The components of the intersection of $H_i\cap H_j$ are two annuli since the base space is $\mathbb{R}P^2$.
			 Then, $H_1\cup H_2\cup H_3$ is a type-$(1, 1, 1)$ decomposition of $M$ that satisfies case (6).
			 Hence, $M$ has a type-$(1, 1, 1)$ decomposition that satisfies $F_{ij}\cong A\cup A$.

		\end{proof} 
		
		\begin{Rem}
				%\item Two cases of type-$(1, 1, 1)$ decomposition of conclusion (5) of Theorem \ref{thm 3} are not isotopic each other. %There are two type-$(1, 1, 1)$ decomposition of the 3-sphere and lens spaces which satisfies the conclusion (5) of Theorem \ref{thm 3} up to isotopy. One is derived from Heegaard splitting and another one is derived from Seifert fibered structure. 
			 In this paper, we consider handlebody decompositions of the 3-sphere and lens spaces. A similar consequence is obtained when a 3-manifold is a Seifert fibered space that has at most three singular fibers. By looking at the details of the proof of Theorem 1 in \cite{G}, we have the following.  If the base space of Seifert fibered space $\mathbb{S}(3)$ is a genus-$g$ orientable closed surface, a type-$(1, 1, 1)$ decomposition of $\mathbb{S}(3)$ satisfies $F_{ij}\cong A_1\cup\cdots\cup A_{2g+1}$. If the base space of $\mathbb{S}(3)$ is a genus-$g$ non-orientable closed surface, a type-$(1, 1, 1)$ decomposition of $\mathbb{S}(3)$ satisfies $F_{ij}\cong A_1\cup\cdots \cup A_{g+1}$.
		
		\end{Rem}

		Before classifying a type-$(1, 1, 1)$ decomposition that satisfies conclusion (2), (3), (4), or (5) of Theorem \ref{thm 3}, we prepare some lemmas.
		By Lemmas  \ref{heegaard} and \ref{heegaard2} below, we can consider a type-$(1, 1, 1)$ decomposition that satisfies conclusion (2), (3), (4), or (5) of Theorem \ref{thm 3} as a genus-one Heegaard splitting.
			\begin{lem}\label{heegaard}
				Let $H_1\cup H_2\cup H_3$ be a type-$(1, 1, 1)$ decomposition of a lens space $M$ that satisfies conclusion (2), (3), (4), or (5) of Theorem \ref{thm 3}. Then, the following are satisfied.
				\begin{enumerate}
				\renewcommand{\theenumi}{(\arabic{enumi})}
				 \item[(1)]  if $H_1\cup H_2\cup H_3$ is a type-$(1, 1, 1)$ decomposition of $M$  that satisfies conclusion (3), (4), or (5) of Theorem \ref{thm 3}, exactly one of $\partial H_i$ is a genus-one Heegaard surface of $M$, where $i\in \{1, 2, 3\}$.
				 \item[(2)]  if $H_1\cup H_2\cup H_3$ is a type-$(1, 1, 1)$ decomposition of $M$  that satisfies conclusion (2) of Theorem \ref{thm 3}, at most two of $\partial H_i$ are genus-one Heegaard surfaces of $M$, where $i\in \{1, 2, 3\}$.
				 \end{enumerate}
			\end{lem}
			
			\begin{proof}
				Suppose that  $H_1\cup H_2\cup H_3$ satisfies conclusion (2) of  Theorem \ref{thm 3}.
				We can assume that $F_{12}\cong D\cup T^\circ$ without loss of  generality.
				$N(D)\cup H_3$ is a punctured $M$ or a 3-ball, where $D$ is a disk component of $F_{12}$ and $N(D)$ is a regular neighborhood of $D$ in $H_1\cup H_2$.
				Suppose that  $N(D)\cup H_3$ is a punctured $M$.
				Then, $\partial H_3$ is a Heegaard surface and $(H_1-N(D))\cup (H_2-N(D))$ is a 3-ball.
				If $\partial (H_1-N(D))$ is a genus-one Heegaard surface of $M$, the exterior of $ (H_1-N(D)$ may be a solid torus.
				This contradicts the assumption that $N(D)\cup H_3$ is a punctured $M$.
				Hence, $\partial (H_1-N(D))$ is not a genus-one Heegaard surface of $M$. 
				Similarly, $\partial (H_2-N(D))$ is not a genus-one Heegaard surface of $M$.
				This implies that if $\partial H_3$ is a genus-one Heegaard surface of $M$, both $\partial H_1$ and $\partial H_2$ are not genus-one Heegaard surfaces of $M$.
				If $N(D)\cup H_3$ is a 3-ball, $(H_1-N(D))\cup (H_2-N(D))\cup (N(D)\cup H_3)$ is a type-$(0, 1, 1; 1)$ decomposition of $M$.
				Then, if $N(D)\cup H_3$ is a 3-ball, both $\partial H_2$ and $\partial H_1$ are Heegaard surfaces by  Lemma \ref{heegaard011lens}.
								
				Next, we suppose that  $H_1\cup H_2\cup H_3$ satisfies conclusion (3) of  Theorem \ref{thm 3}.
				We can assume that $F_{12}\cong D\cup P$ without loss of  generality.
				$N(D)\cup H_3$ is a punctured $M$ or a 3-ball, where $D$ is a disk component of $F_{12}$ and $N(D)$ is a regular neighborhood of $D$ in $H_1\cup H_2$.
				Suppose that  $N(D)\cup H_3$ is a punctured $M$.
				Then, $\partial H_3$ is a Heegaard surface and $(H_1-N(D))\cup (H_2-N(D))$ is a 3-ball.
				If $\partial (H_1-N(D))$ is a genus-one Heegaard surface of $M$, the exterior of $ (H_1-N(D)$ may be a solid torus.
				This contradicts the assumption that $N(D)\cup H_3$ is a punctured $M$.
				Hence, $\partial (H_1-N(D))$ is not a genus-one Heegaard surface of $M$. 
				Similarly, $\partial (H_2-N(D))$ is not a genus-one Heegaard surface of $M$.
				This implies that if $\partial H_3$ is a genus-one Heegaard surface of $M$, both $\partial H_1$ and $\partial H_2$ are not genus-one Heegaard surfaces of $M$.
				Suppose that $N(D)\cup H_3$ is a 3-ball and $(H_1-N(D))\cup (H_2-N(D))\cup (N(D)\cup H_3)$ is a type-$(0, 1, 1; 3)$ decomposition of $M$.
				By Lemma \ref{heegaard011lens}, either $\partial (H_1-N(D))$ or $\partial (H_2-N(D))$ is a genus-one Heegaard surface of $M$.
				This implies that either $\partial H_1$ or $\partial H_2$ is a genus-one Heegaard surface of $M$.
				
				Next, we suppose that  $H_1\cup H_2\cup H_3$ satisfies conclusion (4) of  Theorem \ref{thm 3}.
				We can assume that $F_{12}\cong F_{23}\cong D\cup P$ without loss of generality.
				$N(D)\cup H_3$ is a punctured $M$ or a 3-ball, where $D$ is a disk component of $F_{12}$ and $N(D)$ is a regular neighborhood of $D$ in $H_1\cup H_2$.
				Suppose that  $N(D)\cup H_3$ is a punctured $M$.
				Then, $\partial H_3$ is a Heegaard surface and $(H_1-N(D))\cup (H_2-N(D))$ is a 3-ball.
				If $\partial (H_1-N(D))$ is a genus-one Heegaard surface of $M$, the exterior of $ (H_1-N(D)$ may be a solid torus.
				This contradicts the assumption that $N(D)\cup H_3$ is a punctured $M$.
				Hence, $\partial (H_1-N(D))$ is not a genus-one Heegaard surface of $M$. 
				Similarly, $\partial (H_2-N(D))$ is not a genus-one Heegaard surface of $M$.
				This implies that if $\partial H_3$ is a genus-one Heegaard surface of $M$, both $\partial H_1$ and $\partial H_2$ are not genus-one Heegaard surfaces of $M$.
				Suppose that $N(D)\cup H_3$ is a 3-ball and $(H_1-N(D))\cup (H_2-N(D))\cup (N(D)\cup H_3)$ is a type-$(0, 1, 1; 3)$ decomposition of $M$.
				By Lemma \ref{heegaard011lens}, either $\partial (H_1-N(D))$ or $\partial (H_2-N(D))$ is a genus-one Heegaard surface of $M$.
				This implies that either $\partial H_1$ or $\partial H_2$ is a genus-one Heegaard surface of $M$.
				
				Finally, we suppose that  $H_1\cup H_2\cup H_3$ satisfies conclusion (5) of  Theorem \ref{thm 3}.
				We can assume that $F_{12}\cong D\cup P$ without loss of generality.
				$N(D)\cup H_3$ is a punctured $M$ or a 3-ball, where $D$ is a disk component of $F_{12}$ and $N(D)$ is a regular neighborhood of $D$ in $H_1\cup H_2$.
				Suppose that  $N(D)\cup H_3$ is a punctured $M$.
				Then, $\partial H_3$ is a Heegaard surface and $(H_1-N(D))\cup (H_2-N(D))$ is a 3-ball.
				If $\partial (H_1-N(D))$ is a genus-one Heegaard surface of $M$, the exterior of $ (H_1-N(D)$ may be a solid torus.
				This contradicts the assumption that $N(D)\cup H_3$ is a punctured $M$.
				Hence, $\partial (H_1-N(D))$ is not a genus-one Heegaard surface of $M$. 
				Similarly, $\partial (H_2-N(D))$ is not a genus-one Heegaard surface of $M$.
				This implies that if $\partial H_3$ is a genus-one Heegaard surface of $M$, both $\partial H_1$ and $\partial H_2$ are not genus-one Heegaard surfaces of $M$.
				Suppose that $N(D)\cup H_3$ is a 3-ball and $(H_1-N(D))\cup (H_2-N(D))\cup (N(D)\cup H_3)$ is a type-$(0, 1, 1; 3)$ decomposition of $M$.
				By Lemma \ref{heegaard011lens}, either $\partial (H_1-N(D))$ or $\partial (H_2-N(D))$ is a genus-one Heegaard surface of $M$.
				This implies that either $\partial H_1$ or $\partial H_2$ is a genus-one Heegaard surface of $M$.
			\end{proof}

			\begin{lem}\label{heegaard2}
				Let $H_1\cup H_2\cup H_3$ be a type-$(1, 1, 1)$ decomposition of the 3-sphere $M$ that satisfies one of conclusion (2), (3), (4), or (5) of Theorem \ref{thm 3}.
				Then, each of $\partial H_i$ is a genus-one Heegaard surface for $i=1, 2, 3$.
			\end{lem}
			\begin{proof}
				Suppose that  $H_1\cup H_2\cup H_3$ satisfies conclusion (2) of  Theorem \ref{thm 3}.
				We can assume that $F_{12}\cong D\cup T^\circ$ without loss of generality.
				Then, $N(D)\cup H_3$ is a 3-ball since $M$ is the 3-sphere,  where $D$ is a disk component of $F_{12}$ and $N(D)$ is a regular neighborhood of $D$ in $H_1\cup H_2$.
				This implies that $\partial H_3$ is a genus-one Heegaard surface of $M$.
				Then, $(H_1-N(D))\cup (H_2-N(D))\cup (H_3\cup N(D))$ is a type-$(0, 1, 1; 1)$ decomposition of $M$.
				Hence, both $\partial H_1$ and $\partial H_2$ are also genus-one Heegaard surfaces by Lemma \ref{heegaard011lens}.
				
				Next, we suppose that  $H_1\cup H_2\cup H_3$ satisfies conclusion (3), (4), or (5) of  Theorem \ref{thm 3}.
				We can assume that $F_{12}\cong D\cup P$ without loss of generality.
				Then, $N(D)\cup H_3$ is a 3-ball since $M$ is the 3-sphere,  where $D$ is a disk component of $F_{12}$ and $N(D)$ is a regular neighborhood of $D$ in $H_1\cup H_2$.
				This implies that $\partial H_3$ is a genus-one Heegaard surface of $M$.
				Then, $(H_1-N(D))\cup (H_2-N(D))\cup (H_3\cup N(D))$ is a type-$(0, 1, 1; 3)$ decomposition of $M$.
				Hence, both $\partial H_1$ and $\partial H_2$ are also genus-one Heegaard surfaces by Lemma \ref{heegaard011lens}.
			\end{proof}
			
Next, we consider the surfaces properly embedded in a solid torus or a 3-ball. 
If $H_1\cup H_2\cup H_3$ is a type-$(1, 1, 1)$ decomposition of the 3-sphere or a lens space, $H_i\cup H_j$ is a solid torus of a genus-one Heegaard surface of the 3-sphere or a lens space for some $\{i, j\}\subset\{1, 2, 3\}$ by Lemmas \ref{heegaard} and \ref{heegaard2}.
The following lemmas correspond to the embedding of $F_{ij}$ into a solid torus $H_i\cup H_j$ for some $\{i, j\}\subset\{1, 2, 3\}$.

		\begin{lem}\label{annuiso2}
			Let $A_1\cup A_2$ be the union of two disjoint annuli properly embedded in a solid torus $V$. Then, there are exactly two ambient isotopy classes of the embedding $A_1\cup A_2$ into $V$ if $A_1\cup A_2$ satisfies the following.
			\begin{enumerate}
			\renewcommand{\theenumi}{(\arabic{enumi})}
			\item[(1)]  $A_1\cup A_2$ cuts open $V$ into two solid tori $V_1$ and $V_2$.
			\item[(2)]  $\partial (A_1\cup A_2)$ cuts open $\partial V$  into two disks $D_1$, $D_2$ and two thrice-punctured spheres $P_1$, $P_2$ with $\partial D_i\subset \partial P_i$.
			\end{enumerate}
			Furthermore, these ambient isotopy classes are taken to each other by the hyperelliptic involution on $V$.
		\end{lem}
		
		\begin{proof}
			We denote $S_i=D_i\cup P_i$ for $i=1, 2$. 
			By assumption (2), $S_i$ is an annulus.
			Since $\partial V = S_1\cup S_2$, $\partial S_1=\partial S_2$ is essential in $\partial V$.
			Suppose that $\partial A_1=\partial S_1$.
			Then, $\partial A_2=\partial D_1\cup \partial D_2$.
			Hence, a 2-sphere $A_2\cup D_1\cup D_2$ in $V$ bounds a 3-ball since $V$ is irreducible.
			This contradicts the assumption that $A_1\cup A_2$ cuts open $V$ into two solid tori.
			Hence, we can assume that one of the components of $\partial A_1$ is essential in $\partial V$ and the other is inessential in $\partial V$.
			Let $C$ be a component of $\partial A_1$ such that it is essential in $\partial V$.
			Then, $C$ bounds a disk in $V$ since one of the components of $\partial A_1$ bounds a disk  in $\partial V$.
			Hence, $C$ is meridional.
			
			We can suppose that $\partial D_i$ is a component of $\partial A_i$ without loss of generality.
			Let  $C_1$ and $C_2$ be components of $\partial S_1=\partial S_2$.
			Then, there are two cases.
			One is the case where $\partial A_1=C_1\cup \partial D_1$, $\partial A_2=C_2\cup\partial D_2$ for $i\neq j$ and
			the other is the case where $\partial A_1=C_2\cup \partial D_1$, $\partial A_2=C_1\cup\partial D_2$ for $i\neq j$.
			
			Next, we shall show that the union of two annuli $A_1\cup A_2$ embedded in $V$, satisfying assumptions (1) and (2), $A_1=C_1\cup \partial D_1$, and $A_2=C_2\cup\partial D_2$ for $i\neq j$, is unique up to ambient isotopy in $V$.
			We also denote $A_1'\cup A_2'$ as the union of  two disjoint annuli properly embedded in a solid torus $V$, satisfying assumption (2), $A_1'\cup A_2'$ cuts open $V$ into two solid tori $V_1'$ and $V_2'$, $A_1=C_1\cup \partial D_1$, and $A_2=C_2\cup\partial D_2$ for $i\neq j$.
			We note that $A_1\cup D_1$ is a disk whose boundary is $C_1$.
			Hence, we can obtain a meridian disk $D$ in $V$ by slightly isotoping $A_1\cup D_1$ into $V_1$ or $V_2$.
			We can assume $D$ in $V_1$ without loss of generality.
			Since $A_1$ and $A_2$ are disjoint, $D\cap(A_1\cup A_2)=\emptyset$.
			Similarly, we can obtain a meridian disk $D'$ such that $D'\cap(A_1'\cup A_2')=\emptyset$
			Since the meridian disk of a solid torus is unique up to ambient isotopy, we can obtain a meridian disk $D'$ that satisfies $D''\cap(A_1\cup A_2)=\emptyset$ and $D''\cap(A_1'\cup A_2')=\emptyset$.
			A disk $D''$ is isotopic to $D_1$ since $D_1\cup D''\cup A_1$ is a 2-sphere in $V_1$ and $V$ is irreducible.
			Let $B$ be a 3-ball in $V$ such that $\partial B=D_1\cup D''\cup A_1$.
			We note that $V_1-B$ is a solid torus and $B$ is a 2-handle attached to $V_2$.
			If the core of an attaching region of $B$ in $\partial V_2$ is inessential, $V_1-B$ has two boundary components.
			This contradicts the assumption that $V_1-B$ is a solid torus.
			Hence, the core of an attaching region of $B$ is essential in $\partial V_2$.
			We note that the core of an attaching region of $B$ is a longitude in $\partial V_2$ since $V$ is a solid torus.
			Since $D''$ is a meridian disk of $V$, $V-D''$ is a 3-ball.
			 $A_2$ is embedded in a 3-ball $V-D''$ and it cuts open $V-D''$ into a 3-ball $V_2\cup B$ and a solid torus $V_1- B$.
			 Similarly, $A_2'$ is embedded in a 3-ball $V-D''$ and it cuts open $V-D''$ into a 3-ball and a solid torus.
			 Hence, $A_2$ is ambient isotopic to $A_2'$ by Lemma \ref{ann}.
			 Similarly, we can show that $A_1$ is ambient isotopic to $A_1'$.
			 Hence, $A_1\cup A_2$ is unique up to ambient isotopy in $V$.
			 Similarly, we can show that $A_1\cup A_2$ embedded in $V$, satisfying assumptions (1) and (2), $A_1=C_1\cup \partial D_2$, and $A_2=C_2\cup\partial D_1$ for $i\neq j$, is unique up to ambient isotopy in $V$.
			
			Then, $A_1\cup A_2$ satisfying assumptions (1) and (2) is as shown in Figure \ref{annuli_2}.
			Figure \ref{annuli_2} (b) can be obtained by $\pi$-rotating Figure \ref{annuli_2} (a)  (see Figure \ref{annuli_2}).
			Conversely, if these are isotopic each other, there exists an ambient isotopy $F: V \times I\to V$ such that 
			\[
					F(C_1, 1)=C_2, F(C_2, 1)=C_1, F(S_1, 1)=S_1, F(S_2, 1)=S_2.
			\]
			We define $f_t(x)=F(x, t)$.
			Suppose that $f_1|_{\partial V}:\partial V\to \partial V$ is an identity in the mapping class group of a torus.
			Since $F(S_1, 1)=S_1$ and $F(S_2, 1)=S_2$, $F(C_1, 1)=C_1$ and  $F(C_2, 1)=C_2$.
			This contradicts the definition of $F$.
			Hence, $f_1|_{\partial V}:\partial V\to \partial V$ is not an identity in the mapping class group of a torus.
			It is clear that the order of $f_1|_{\partial V}:\partial V\to \partial V$ is 2 in the mapping class group of a torus.
			Hence, $f_1|_{\partial V}:\partial V\to \partial V$ is the hyperelliptic involution. 
			This contradicts the assumption that $F: V \times I\to V$ is an ambient isotopy.
			 \end{proof}
				\begin{figure}[httb]
				\centering
				\includegraphics[scale=0.8]{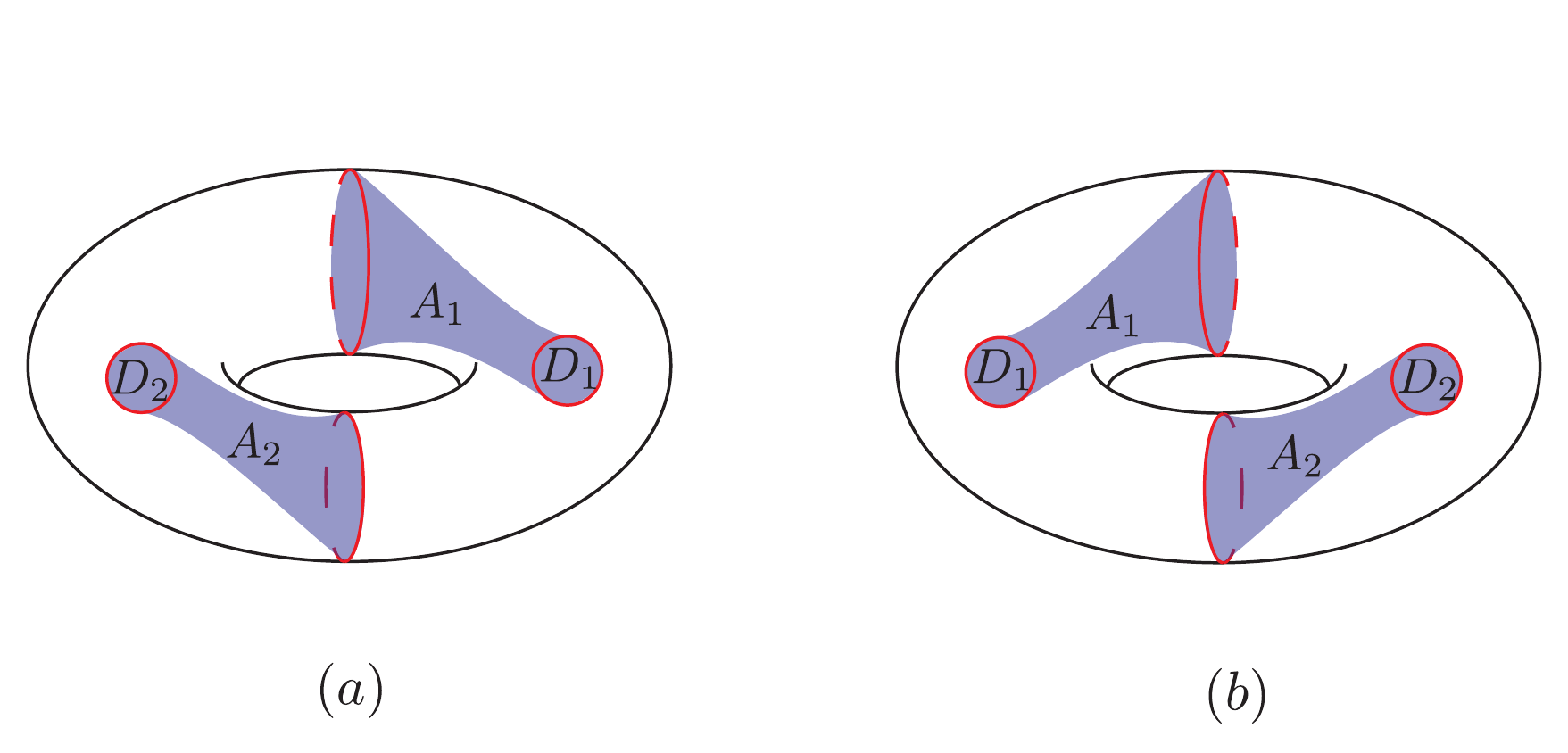}
				\vspace{-3mm}
				\caption{The two embedded annuli in a solid torus satisfying the assumption of Lemma \ref{annuiso2}. (a) in a figure is taken to (b) by the hyperelliptic involution.}
				\label{annuli_2}
				\end{figure}

		\begin{lem}\label{annuiso1}
		Let $A_1\cup A_2$ be the union of  two disjoint annuli properly embedded in a solid torus $V$. Then, there are exactly two ambient isotopy classes of the embedding $A_1\cup A_2$ into $V$ if $A_1\cup A_2$ satisfies the following.
		\begin{enumerate}
		\renewcommand{\theenumi}{(\arabic{enumi})}
			\item[(1)]  $A_1\cup A_2$ cuts open $V$ into two solid tori $V_1$ and $V_2$.
			\item[(2)]  $\partial (A_1\cup A_2)$ cuts open $\partial V$ into a disk $D$, a thrice-punctured sphere $P$, and two annuli $S_1$, $S_2$ such that $\partial S_1$ is inessential in $\partial V$ and $\partial S_2$ is essential in $\partial V$.
		\end{enumerate}
		Furthermore, these ambient isotopy classes are taken to each other by hyperelliptic involution on $V$.
		\end{lem}
		\begin{proof}
				If $\partial A_1=\partial S_1$, $A_1\cup A_2$ cuts open $V$ into three solid tori.
				This contradicts assumption (2).
				Hence, we can assume that one of the components of $\partial A_i$ is essential in $\partial V$ and the other is inessential in $\partial V$ for $i=1, 2$.
				Therefore, the component of $\partial A_i$ that is essential in $\partial V$ is meridional since it bounds a disk $A_i\cup D_i$, where $D_i$ is a disk bounded by the component of $\partial A_i$ that is inessential in $\partial V$.
				We can suppose that $\partial D$ is a component of $\partial A_1$ without loss of generality.
				We denote $\partial S_2=C_1\cup C_2$.
				Then, if $\partial A_1=\partial D\cup C_1$, $\partial A_2=(\partial S_1- \partial D) \cup C_2$.
				On the other hand, if  $\partial A_1=\partial D\cup C_2$, $\partial A_2=(\partial S_1- \partial D) \cup C_1$.
				
				Next, we shall show that the union of two annuli $A_1\cup A_2$ embedded in $V$, satisfying assumptions (1) and (2), $\partial A_1=\partial D\cup C_1$, and $\partial A_2=(\partial S_1- \partial D) \cup C_2$,  is unique up to ambient isotopy in $V$.
				We also denote $A_1'\cup A_2'$ as the union of  two disjoint annuli properly embedded in a solid torus $V$, satisfying assumption (2), $A_1'\cup A_2'$ cuts open $V$ into two solid tori $V_1'$ and $V_2'$, $\partial A_1'=\partial D\cup C_1$, and $\partial A_2'=(\partial S_1- \partial D) \cup C_2$.
				Since $A_1\cup D$ is a disk embedded in $V$ that does not intersect $A_2$, $C_1$ bounds a meridian disk $D'$ of $V$ that does not intersect $A_2$.
				Similarly, we can obtain a meridian disk $D''$ that does not intersect $A_2'$.
				Since a meridian disk of a solid torus is unique up to ambient isotopy, we can isotope $D''\cup A_1\cup A_2$ by taking $D''$ to $D'$.
				Then, $A_1'\cap D'=\emptyset$ and $A_2'\cap D'=\emptyset$.
				We note that $A_1\cup D\cup D'$ is a 2-sphere in $V$.
				Hence, $A_1\cup D\cup D'$ bounds a 3-ball $B$ in $V$ since $V$ is irreducible.
				We can assume that $B$ is in $V_1$ without loss of generality.
				Since $A_1\cup D$ is a disk in  $\partial V_1$, $D'$ is boundary-parallel in $V_1$.
				Hence, $V_1-B$ is a solid torus.
				We can consider $B$ as a 2-handle attached to $V_2$ since $V_2\cap B$ is an annulus $A_1$.
				Hence, $B\cup V_2$ is a 3-ball since $V$ is a solid torus.
				%We note that $V-N(D')=(V_1-B-N(D'))\cup_{A_2}(V_2\cup B-N(D'))$.
				Then, $A_2$ is a properly embedded annulus in a 3-ball  $V - N(D')$ that cuts open $V - N(D')$ into a 3-ball $V_2\cup (B-N(D'))$ and a solid torus $V_1-(B-N(D'))$. 
				Similarly, $A_2'$ is a properly embedded annulus in a 3-ball  $V - N(D')$ that cuts open $V - N(D')$ into a 3-ball and a solid torus.
				Hence, by Lemma \ref{ann}, $A_2$ is  ambient isotopic to $A_2'$.
				We can suppose that $A_2=A_2'$.
				We can assume that $A_2$ cuts open $V-N(D')$ into a 3-ball $B'$ and a solid torus.
				Then, $A_1$ is an annulus properly embedded in a 3-ball $B'$ that cuts open a 3-ball $B$ and a solid torus $V_2$.
				Similarly, $A_1$ is also an annulus properly embedded in a 3-ball $B'$ that cuts open a 3-ball and a solid torus.
				Hence, $A_1$ is ambient isotopic to $A_1'$ in $V$ by Lemma \ref{ann}. %境界は固定していると思っていいので一意性を使える。
				
				Similarly, we can show that the union of two annuli $A_1\cup A_2$ embedded in $V$, satisfying assumptions (1) and (2), $\partial A_1=\partial D\cup C_2$, and $\partial A_2=(\partial S_1- \partial D) \cup C_1$, is unique up to ambient isotopy in $V$.				
				
				Then, $A_1\cup A_2$, satisfying assumptions (1) and (2), is as shown in Figure \ref{annuli_1}.
				Figure \ref{annuli_1} (b) can be obtained by $\pi$-rotating Figure \ref{annuli_1} (a)  (see Figure \ref{annuli_1}).
				Conversely, if these are ambient isotopic to each other in $V$, there exists an ambient isotopy $F: V \times I\to V$ such that 
					\[
							F(C_1, 1)=C_2, F(C_2, 1)=C_1, F(S_1, 1)=S_1, F(S_2, 1)=S_2.
					\]
				We define $f_t(x)=F(x, t)$.
				Suppose that $f_1|_{\partial V}:\partial V\to \partial V$ is an identity in the mapping class group of a torus.
				Since $F(S_1, 1)=S_1$ and $F(S_2, 1)=S_2$, $F(C_1, 1)=C_1$ and  $F(C_2, 1)=C_2$.
				This contradicts the definition of $F$.
				Hence, $f_1|_{\partial V}:\partial V\to \partial V$ is not an identity in the mapping class group of a torus.
				It is clear that the order of $f_1|_{\partial V}:\partial V\to \partial V$ is 2 in the mapping class group of a torus.
				Hence, $f_1|_{\partial V}:\partial V\to \partial V$ is a hyperelliptic involution. 
				This contradicts the assumption that $F: V \times I\to V$ is an ambient isotopy.
				This completes the proof of the lemma.
				\end{proof}
				\begin{figure}[httb]
				\centering
				\includegraphics[scale=0.8]{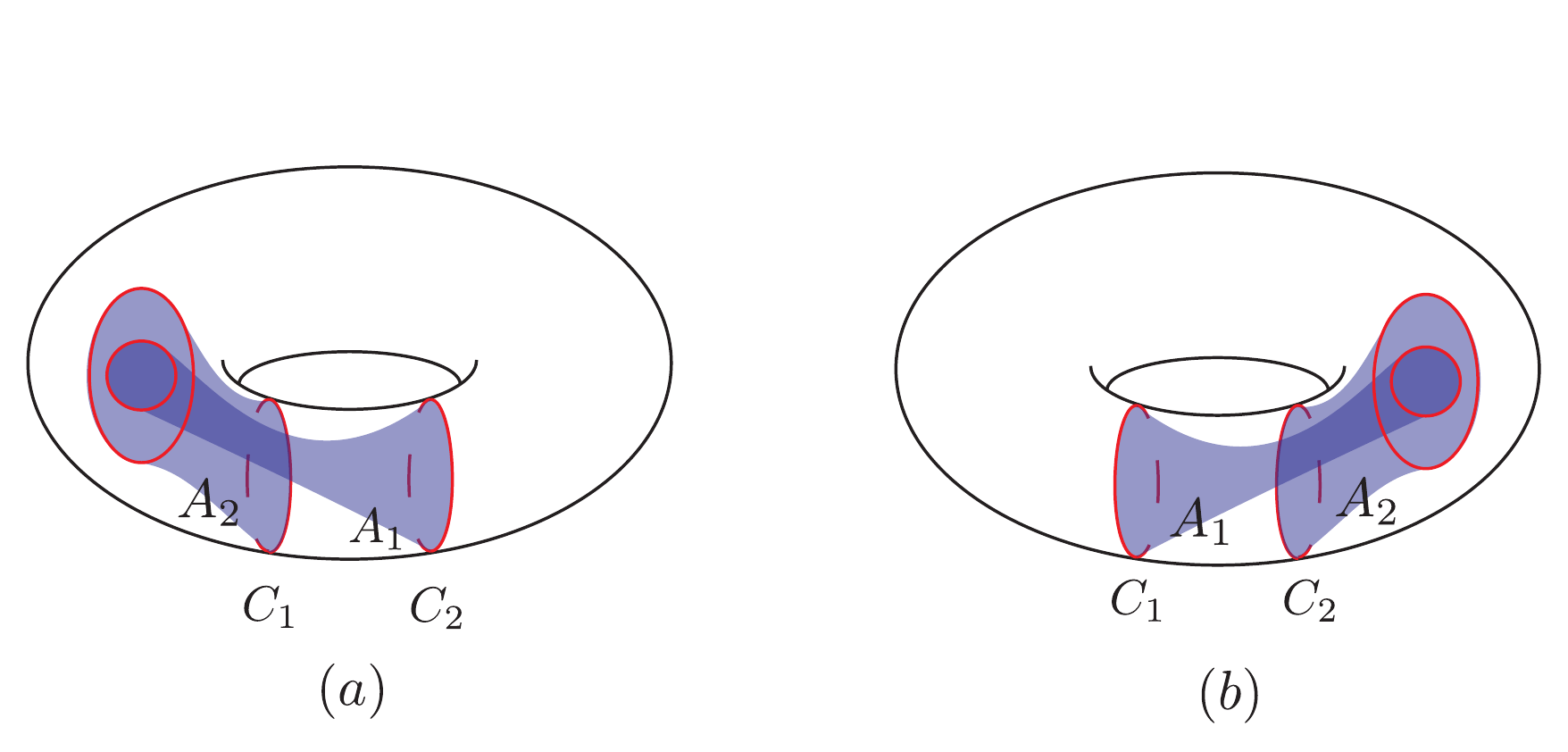}
				\vspace{-3mm}
				\caption{The embedding of two annuli into a solid torus satisfying the assumption of Lemma \ref{annuiso1}. (a) in a figure is taken to (b) by the hyperelliptic involution.}
				\label{annuli_1}
				\end{figure}

			\begin{lem}\label{punc}
			 	 A thrice-punctured sphere $P$ embedded in a 3-ball $B$ is unique up to ambient isotopy  if $P$ cuts open $B$ into two solid tori and $\partial P$ cuts open $\partial B$ into two annuli and two disks.
			\end{lem}
			\begin{proof}
			We assume that  $\partial B - \partial P=A_1\cup A_2\cup D_1\cup D_2$, where $A_i$ is an annulus and $D_i$ a disk for $i=1, 2$. 
			Furthermore, we can assume that $B-P=H_1\cup H_2$, where $H_i$ is a solid torus for $i=1, 2$.
			We take a simple closed curve $C$ in $P$, which is parallel to $\partial D_1$.
			Then, $C$ bounds a disk $D$ in $H_1$ or $H_2$.
			We can assume that $D$ is in $H_1$ without loss of generality.
			Since $D$ is parallel to $\partial D_1$, $D$ is a boundary-parallel disk in $H_1$.
			Then, $C$ cuts open $P$ into $A$ and $P'$, where $A$ is an annulus and $P'$ is a thrice-punctured sphere.
			Hence, $S=D\cup A\cup D_1$ is a 2-sphere embedded in $B$. 
			Then, $S$ bounds a 3-ball $B'$ that can be regarded as a 2-handle  attached to $H_2$.
			If the core of an attaching region of $B'$ is inessential in $\partial H_2$, $H_2\cup B'\cup H_1$ has two boundary components.
			This contradicts the assumption that $H_2\cup B'\cup H_1$ is a 3-ball.
			We can assume that the core of an attaching region of $B'$ is essential in $\partial H_2$
			Hence, $H_2\cup B'$ is a 3-ball since $B$ is a 3-ball.
			We note that $H_1- B'$ is a solid torus since $D$ is a boundary-parallel disk in $H_1$.
			Therefore, $P'\cup D$ is an annulus properly embedded in $B^3$ that cuts open $B^3$ into a solid torus $H_1- B'$ and a 3-ball $H_2\cup B'$.
			Then, the embedding of  $P'$ into $B^3$ is unique up to ambient isotopy by Lemma \ref{ann}.
			Furthermore, $A$ is an annulus properly embedded in a 3-ball  $H_2\cup B'$ that cuts open $H_2\cup B'$ into a solid torus $H_2$ and a 3-ball $B'$.
			Then, the embedding of $A$ into $H_2\cup B'$ is unique up to isotopy  by Lemma \ref{ann}.
			Therefore, the embedding of $P$ into $B^3$ is unique up to isotopy (see Figure \ref{2-sphe}). 
			\end{proof}
				\begin{figure}[httb]
				\centering
				\includegraphics[scale=0.8]{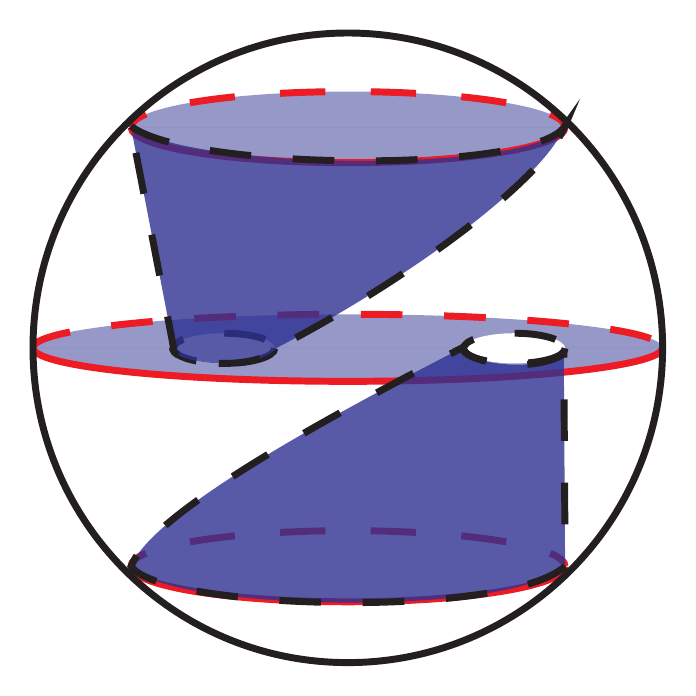}
				\vspace{-3mm}
				\caption{The embedding of a thrice-punctured 2-sphere into a 3-ball}
				\label{2-sphe}
				\end{figure}
				
			\begin{lem}\label{punc1}
				The proper embedding of the punctured torus $T^\circ$ into a 3-ball  $B$ such that $T^\circ$ cuts open $B$ into two solid tori  is unique up to isotopy.
			\end{lem}
			\begin{proof}
				There are meridian disks $D_i$ in $H_i$ for $i=1, 2$ whose boundaries intersect each other exactly once since $B$ is a 3-ball.
				Then, $H_1$ is a 1-handle in $B$ and it is unknotted since there are meridian disks $D_i$ in $H_i$ whose boundaries intersect exactly once.
				Hence, the proper embedding of a punctured torus $T^\circ$ into a 3-ball  $B$ that cuts open $B$ into two solid tori  is  unique up to ambient isotopy.
			\end{proof}
							\begin{figure}[httb]
				\centering
				\includegraphics[scale=0.8]{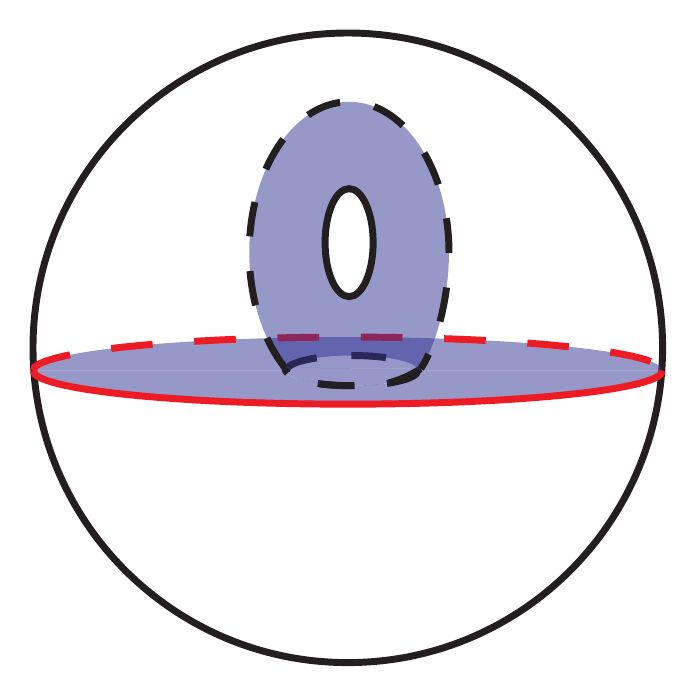}
				\vspace{-3mm}
				\caption{The embedding of a punctured torus into a 3-ball}
				\label{punctorus}
				\end{figure}
			\begin{lem}\label{punc_1}
 				Let $V$ be a solid torus and $D\cup P$ be the union of a disk and a thrice-punctured sphere properly embedded in $V$. If $D\cup P$ satisfies the following, there are exactly two embeddings of $D\cup P$ up to ambient isotopy.
				 \begin{enumerate}
				 \renewcommand{\theenumi}{(\arabic{enumi})}
  					   \item[(1)]  $D\cup P$ cuts open $V$ into two solid tori $V_1$ and $V_2$.
  					   \item[(2)]  $\partial (D\cup P)$ cuts open $\partial V$ into $D\cup P\cup A_1\cup A_2$, which satisfies the condition that $\partial A_1$ is inessential and $\partial A_2$ is essential in $\partial V$. 
				 \end{enumerate}
				 Furthermore, these ambient isotopy classes are taken to each other by the hyperelliptic involution on $V$.
			\end{lem}
				\begin{proof}
  					  Let $S_1=A_1\cup D\cup P$ and $S_2=A_2$.
  					  Then, $S_1\cup S_2=\partial V$.
    					  Let $\partial S_1=C_1\cup C_2$.
  					  If $\partial D$ is inessential in $\partial V$, this contradicts assumption (1).
					  Hence, $\partial D$ is essential in $\partial V$.
					  This implies that $D$ is a meridian disk of $V$.
  					  Therefore, we have  the following two cases by assumption (2).
  					  \begin{enumerate}
					  \renewcommand{\theenumi}{(\arabic{enumi})}
      							  \item[(1)]  $\partial D= C_1$.
       							 \item[(2)]  $\partial D= C_2$.
  					  \end{enumerate}
					  By Lemma \ref{punc}, the embedding of $D\cup P$, which satisfies cases (1) and (2), is unique up to ambient isotopy.
					  We suppose that each of the cases are ambient isotopic.
					  Then, the ambient isotopy $F$ satisfies the following.
 					   \[
  						      F(C_1, 1)=C_2, F(C_2, 1)=C_1, F(S_1, 1)=S_1, F(S_2, 1)=S_2.
  					  \]
					  We define $f_t(x)=F(x, t)$.
					  Suppose that $f_1|_{\partial V}:\partial V\to \partial V$ is an identity in the mapping class group of a torus.
					  Since $F(S_1, 1)=S_1$ and $F(S_2, 1)=S_2$, $F(C_1, 1)=C_1$ and  $F(C_2, 1)=C_2$.
					  This contradicts the definition of $F$.
					  Hence, $f_1|_{\partial V}:\partial V\to \partial V$ is not an identity in the mapping class group of a torus.
					  It is clear that the order of $f_1|_{\partial V}:\partial V\to \partial V$ is 2 in the mapping class group of a torus.
  					  This implies that $f_1|_{\partial V}:\partial V\to \partial V$ is a hyperelliptic involution in the mapping class group of a torus.
					  This contradicts the assumption that $F$ is an ambient isotopy.
					  Hence, each of the cases is not ambient isotopic in $V$ to each other.
				\end{proof}	
						\begin{figure}[httb]
							\centering
							\includegraphics[scale=0.8]{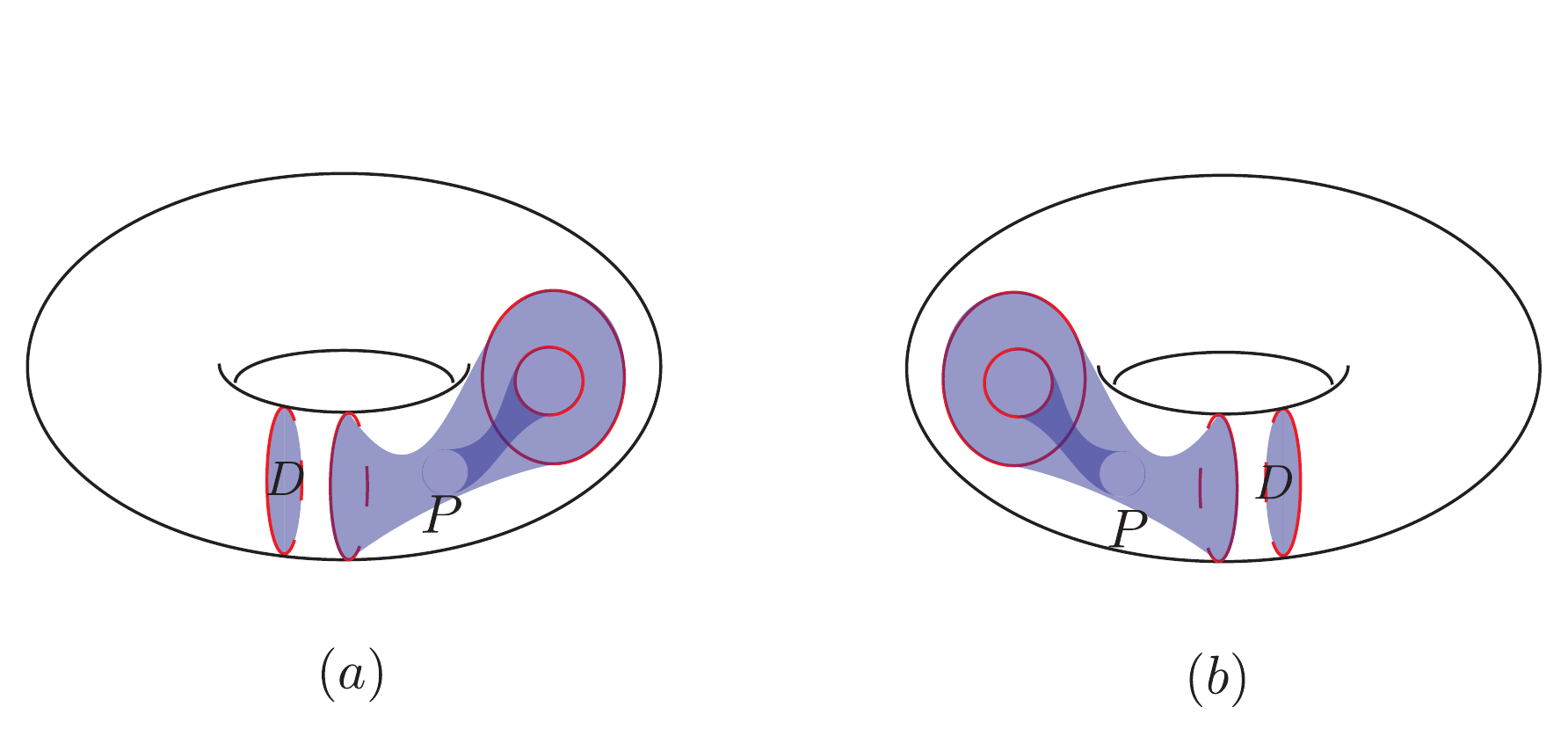}
							\vspace{-3mm}
							\caption{The embedding of a thrice-punctured 2-sphere and a disk  into a solid torus, satisfying the assumption of Lemma \ref{punc_1}. (a) in a figure is taken to (b) by the hyperelliptic involution.}
							\label{punc_1eps}
						\end{figure}
				\begin{lem}\label{punc_2}
					Let $V$ be a solid torus and $D\cup P$ be the union of a disk and a thrice-punctured sphere properly embedded in $V$. If $D\cup P$ satisfies the following, there is exactly one embedding of $D\cup P$ up to ambient isotopy.
				 	\begin{enumerate}
					\renewcommand{\theenumi}{(\arabic{enumi})}
  					   \item[(1)]  $D\cup P$ cuts open $V$ into two solid tori $V_1$ and $V_2$.
  					   \item[(2)]  $\partial (D\cup P)$ cuts open $\partial V$ into $A_1\cup A_2\cup A_3\cup A_4$. 
					 \end{enumerate}
				\end{lem}
					\begin{proof}
						From assumption (2),  $\partial D$ is essential.
						Hence, $D$ is a meridian disk of $V$.
						Then, each component of $\partial P$ is meridional in $\partial V$.
						Let $D'\cup P'$ be the union of a properly embedded disk and a thrice-punctured sphere satisfying assumption (2) and cutting open $V$ into two solid tori.
						We can isotope $D'\cup P'$ by taking $D'$ to $D$ since they are meridian disks of $V$.
						$\partial P$ and $\partial P'$ are isotopic since each component of $\partial P$ and $\partial P'$ is meridional.
						By Lemma \ref{punc}, $P$ is ambient isotopic to $P'$  in $V-D$.
						Hence, $D'\cup P'$ is ambient isotopic to $D'\cup P'$.
						This completes the proof (see figure \ref{punc3eps}).
					\end{proof}
					\begin{figure}[httb]
							\centering
							\includegraphics[scale=0.8]{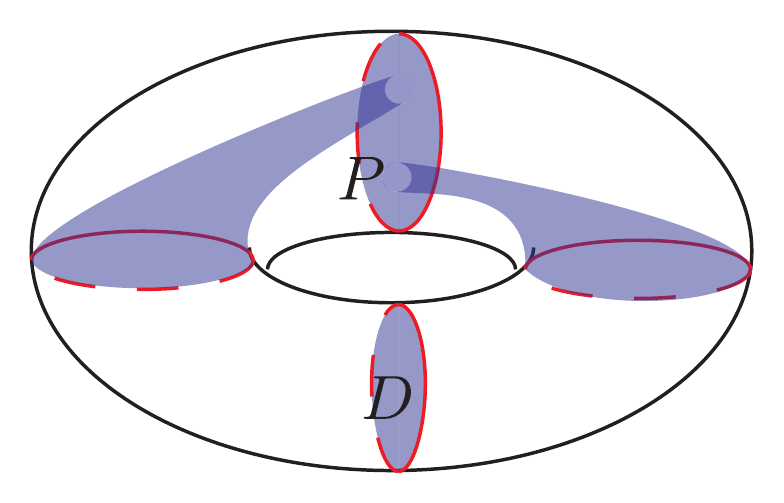}
							\vspace{-3mm}
							\caption{The embedding of a thrice-punctured 2-sphere and a disk  into a solid torus satisfying the assumption of Lemma \ref{punc_2}}
							\label{punc3eps}
							\end{figure}
					
				\begin{lem}\label{punc_3}
					Let $V$ be a solid torus and $D\cup P$ be the union of a disk and a thrice-punctured sphere properly embedded in $V$. If $D\cup P$ satisfies the following, there is exactly one embedding of $D\cup P$ up to ambient isotopy.
				 	\begin{enumerate}
					\renewcommand{\theenumi}{(\arabic{enumi})}
  					   \item[(1)]  $D\cup P$ cuts open $V$ into two solid tori $V_1$ and $V_2$.
  					   \item[(2)]  $\partial (D\cup P)$ cuts open $\partial V$ into $D_1\cup P_1\cup D_2\cup P_2$ such that $\partial D_1\subset \partial P_2$, $\partial D_2\subset \partial P_1$. 
					 \end{enumerate}
				\end{lem}
					\begin{proof}
						If $\partial D$ is inessential in $\partial V$, this contradicts assumption (1).
						Then, we can suppose that $\partial D$ is essential in $\partial V$.
						Hence, $D$ is a meridian disk of $V$.
						Then, exactly one component of $\partial P$ is meridional in $\partial V$ and the other components are inessential in $\partial V$.
						Let $D'\cup P'$ be the union of a properly embedded disk and a thrice-punctured sphere satisfying assumption (2) and cutting open $V$ into two solid tori.
						We can isotope $D'\cup P'$ by taking $D'$ to $D$ since they are meridian disks of $V$.
						Then, we can suppose that $\partial P=\partial P'$.
						By Lemma \ref{punc}, $P$ is ambient isotopic to $P'$  in $V-D$.
						Hence, $D'\cup P'$ is ambient isotopic to $D'\cup P'$.
						This completes the proof (see figure \ref{punc4eps}).
					\end{proof}

							\begin{figure}[httb]
							\centering
							\includegraphics[scale=0.8]{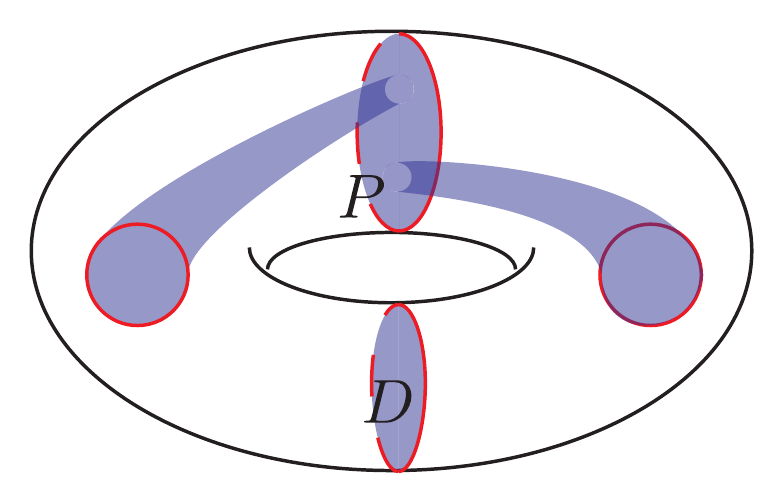}
							\vspace{-3mm}
							\caption{The embedding of a thrice-punctured 2-sphere and a disk  into a solid torus satisfying the assumption of Lemma \ref{punc_3}}
							\label{punc4eps}
							\end{figure}
Then, we shall start classifying the type-$(1, 1, 1)$ decompositions of the 3-sphere. 
			\begin{prop}\label{s3thecase2}
				Let $H_1\cup H_2\cup H_3$ be the type-$(1, 1, 1)$ decomposition of the 3-sphere $M$ that satisfies conclusion (2) of Theorem \ref{thm 3}.
				Then, the embedding of the branched surface $F_{12}\cup F_{13}\cup F_{23}$ into $M$ is unique up to ambient isotopy .
			\end{prop}
			\begin{proof}
			Let $H_1\cup H_2\cup H_3$ and $H_1'\cup H_2'\cup H_3'$ be the type-$(1, 1, 1)$ decompositions of the 3-sphere or a lens space that satisfies conclusion (2) of Theorem \ref{thm 3}.
			Furthermore, we denote $F_{ij}=H_i\cap  H_j$ and $F'_{ij}=H'_i\cap  H'_j$.
			We can assume that $F_{12} \cong F_{12}'\cong D^2\cup T^\circ$ without loss of generality.
			By Lemma \ref{heegaard2}, each of $\partial H_i$ is a genus-one Heegaard surface for $i=1, 2, 3$.
			Then, we have to consider only the cases where each of $\partial H_1$ and $\partial H_1'$ is a Heegaard surface of $M$.
			Since a Heegaard surface of $M$ is unique up to isotopy, we can assume that $\partial H_1=\partial H_1'$.
			By Lemma \ref{heegaard2}, $\partial H_2$ is also a Heegaard surface.
			
			We suppose that $\partial H_1$ is a Heegaard surface of $M$. 
			Let $V_1\cup V_2$ be a genus-one Heegaard splitting of $M$.
			Then, we can assume that $H_1=V_1$.
			Now, $F_{23}$ is an annulus embedded in $V_2$ that satisfies the condition that $\partial F_{23}$ is inessential in $\partial V_2$.
			$F_{23}$ is boundary-parallel in $V_2$ since $F_{23}$ cuts open $V_2$ into two solid tori and $\partial F_{23}$ is inessential in $\partial V_2$.
			Suppose that  $H_1'=V_1$.
			Then, $\partial F_{23}$ and $\partial F_{23}'$ are inessential in $V_2$ and these are boundary-parallel in $V_2$.
			Hence, $F_{23}$ is isotopic to $F_{23}'$.
			Therefore, if   $H_1'=V_1$, $F_{12}\cup F_{13}\cup F_{23}$ is isotopic to $F_{12}'\cup F_{13}'\cup F_{23}'$.
			If $H_1'=V_2$, $F_{23}$ is isotopic to $F_{13}'$ since $F_{23}$ is boundary-parallel and  $\partial F_{23}=\partial F_{13}'$. 
			Furthermore, $F_{13}\cup F_{12}$ is isotopic to  $F_{12}'\cup F_{23}'$ since $\partial H_2'$ is also a Heegaard surface.
			Then, the embedding of $F_{12}\cup F_{13}\cup F_{23}$ into $M$  such that $\partial H_1$ is a Heegaard surface is unique up to isotopy.
			Therefore, the type-$(1, 1, 1)$ decomposition of the 3-sphere or a lens space $M$ that satisfies conclusion (2) of Theorem \ref{thm 3} is unique up to ambient isotopy.
			\end{proof}
			
			\begin{prop}\label{s3thecase3}
				Let $H_1\cup H_2\cup H_3$ be the type-$(1, 1, 1)$ decomposition of the 3-sphere  $M$ that satisfies conclusion (3), (4), or (5) of Theorem \ref{thm 3}.
				Then, the embedding of $F_{12}\cup F_{13}\cup F_{23}$ into $M$ is unique up to isotopy.
			\end{prop}
			\begin{proof}
				Let $H_1\cup H_2\cup H_3$ and $H_1'\cup H_2'\cup H_3'$ be the type-$(1, 1, 1)$ decompositions of the 3-sphere or a lens space that satisfies conclusion (2) of Theorem \ref{thm 3}.
				Furthermore, we denote $F_{ij}=H_i\cap  H_j$ and $F'_{ij}=H'_i\cap  H'_j$.
				We can assume that $F_{12} \cong F_{12}'\cong D^2\cup P$  without loss of generality.
				By Lemma \ref{heegaard2}, each of $\partial H_i$ is a genus-one Heegaard surface for $i=1, 2, 3$
				Then, we have to consider only the case where each of $\partial H_3$ and $\partial H_3'$ is a Heegaard surface of $M$ .
				
				Let $V_1\cup V_2$ be a genus-one Heegaard splitting of $M$.
				Then, we can suppose that $H_3=V_1$.
				The embedding of $F_{12}$ into $V_2$ is unique up to ambient isotopy if $H_1\cup H_2\cup H_3$ is a decomposition of Theorem \ref{thm 3} (3), (4), or (5) by Lemmas \ref{punc_1}, \ref{punc_2}, and \ref{punc_3}.
				If $H_1\cup H_2\cup H_3$ is a decomposition of Theorem \ref{thm 3} (3), (4), or (5), two isotopy classes of the embedding of $F_{12}$ into $V_2$ are ambient isotopic to each other in $M$ by Lemma \ref{diff}.
				If $V_2=H'_3$, then $H'_3$ is isotopic to $H_3$ if and only if $M$ is homeomorphic to the 3-sphere or a lens space $L(p, q)$ with \p by Lemma \ref{core}.
				Hence, we can assume that $H_3=V_1$.
				Since $H_3=V_1$, $F_{12}\cup F_{13}\cup F_{23}$ is isotopic to $F_{12}'\cup F_{13}'\cup F_{23}'$.
				Hence,  the embedding of $F_{12}\cup F_{13}\cup F_{23}$ into $M$  is unique up to ambient isotopy.
			\end{proof}
			
			Next, we shall start classifying the type-$(1, 1, 1)$ decompositions of lens spaces. 
		\begin{prop}\label{thecase2}
			Let $H_1\cup H_2\cup H_3$ be a type-$(1, 1, 1)$ decomposition of a lens space $M$ that satisfies conclusion (2) of Theorem \ref{thm 3}.
			Then, there are two embeddings of $F_{12}\cup F_{13}\cup F_{23}$ into $M$  up to ambient isotopy  if $M$ is $L(p, q)$ with \p.
			Otherwise, there are three embeddings of $F_{12}\cup F_{13}\cup F_{23}$ into $M$ up to ambient isotopy.
		\end{prop}
		\begin{proof}
			Let $H_1\cup H_2\cup H_3$ and $H_1'\cup H_2'\cup H_3'$ be the type-$(1, 1, 1)$ decompositions of a lens space that satisfies conclusion (2) of Theorem \ref{thm 3}.
			Furthermore, we denote $F_{ij}=H_i\cap  H_j$ and $F'_{ij}=H'_i\cap  H'_j$.
			We can assume that  $F_{12} \cong F_{12}'\cong D^2\cup T^\circ$ without loss of generality.
			By \ref{heegaard}, if $\partial H_1$ (resp. $\partial H_1'$) is a Heegaard surface of $M$, $\partial H_2$ (resp. $\partial H_2'$) is also a Heegaard surface.
			Hence, there are two cases where each of $\partial H_1$ (resp. $\partial H_1'$) and $\partial H_2$ (resp. $\partial H_2'$) is a Heegaard surface of $M$ or $\partial H_3$ (resp. $\partial H_3'$) is a Heegaard surface of $M$.
			Then, we must consider the cases where $\partial H_1$ is a Heegaard surface of $M$ and $\partial H_3$ is a Heegaard surface of $M$.
			Let $V_1\cup V_2$ be a genus-one Heegaard splitting of $M$.
			If $\partial H_i$ is a Heegaard surface of $M$, then $\partial H_i$ is isotopic to $\partial V_1$ since a Heegaard splitting of $M$ is unique up to isotopy for $i=1,  3$.
			
			First, we suppose that each of $\partial H_1$ and $\partial H_1'$ is a Heegaard surface of $M$. 
			Then, we can assume that $H_1=V_1$.
			Now, $F_{23}$ is an annulus embedded in $V_2$ that satisfies the condition that $\partial F_{23}$ is inessential in $\partial V_2$.
			$F_{23}$ is boundary-parallel in $V_2$ since $F_{23}$ cuts open $V_2$ into two solid tori and $\partial F_{23}$ is inessential in $\partial V_2$.
			Suppose that  $H_1'=V_1$.
			Then, $\partial F_{23}$ and $\partial F_{23}'$ are inessential in $V_2$ and they are boundary-parallel in $V_2$.
			Hence, $F_{23}$ is isotopic to $F_{23}'$ by keeping $\partial H_1=\partial H_1'$.
			Therefore, if $H_1'=V_1$, $F_{12}\cup F_{13}\cup F_{23}$ is isotopic to $F_{12}\cup F_{13}\cup F_{23}$.
			If $H_1'=V_2$, $F_{23}$ is isotopic to $F_{13}'$ since $F_{23}$ is boundary-parallel and  $\partial F_{23}=\partial F_{13}'$. 
			Furthermore, $F_{13}\cup F_{12}$ is isotopic to  $F_{12}'\cup F_{23}'$ since $\partial H_2'$ is also a Heegaard surface.
			Then, the embedding of $F_{12}\cup F_{13}\cup F_{23}$ into $M$ is such that $\partial H_1$ is unique up to ambient isotopy.
			Therefore, the type-$(1, 1, 1)$ decomposition of a lens space $M$ that satisfies conclusion (2) of Theorem \ref{thm 3} that satisfies the condition that $\partial H_1$ is a Heegaard surface of $M$ has one isotopy class.
			
			Next, we suppose that $\partial H_3$ and $\partial H_3'$ are Heegaard surfaces of $M$. Then, we can assume that $H_3=V_1$.
			$F_{12}$ is the union of a disk and a thrice-punctured sphere properly embedded in $V_2$ that satisfies the assumption of Lemma \ref{punc_2}.
			Hence, if $H_3'=V_1$,  $F_{12}\cup F_{13}\cup F_{23}$ is isotopic to $F_{12}'\cup F_{13}'\cup F_{23}'$ by Lemma \ref{punc_2}.
			If $V_2=H'_3$, then $H'_3$ is isotopic to $H_3$ if and only if $M$ is $L(p, q)$ with \p by Lemma \ref{core}.
			Then, if $M$ is $L(p, q)$ with \p, $F_{12}\cup F_{13}\cup F_{23}$ is isotopic to $F'_{12}\cup F'_{13}\cup F'_{23}$.
			On the other hand, if $M$ is $L(p, q)$ with \np, $F_{12}\cup F_{13}\cup F_{23}$ cannot be isotopic to $F'_{12}\cup F'_{13}\cup F'_{23}$.
			Hence,  the embedding of $F_{12}\cup F_{13}\cup F_{23}$ into $M$  that satisfies the condition that  $\partial H_3$ is a Heegaard surface is unique up to ambient isotopy if $M$ is $L(p, q)$ with \p.
			On the other hand,  the embedding of $F_{12}\cup F_{13}\cup F_{23}$ into $M$   that satisfies the condition that $\partial H_3$ is a Heegaard surface has two isotopy classes if $M$ is  $L(p, q)$ with \np.
			
			%To summarize the above, the type-$(1, 1, 1)$ decomposition of a lens space $M$ which satisfies the conclusion (2) of Theorem \ref{thm 3} and $\partial H_3$ is a Heegaard surface of $M$  has one isotopy class if $M$ is $L(p, q)$ with $p=(-1+q)/(qn+s)$ or $p= (-1-q)/(qn+s)$ for some integers $n, s$.
			%On the other hand the type-$(1, 1, 1)$ decomposition of  a lens space $M$ which satisfies the conclusion (2) of Theorem \ref{thm 3} and $\partial H_3$ is a Heegaard surface of $M$ has two isotopy classes if $M$ is $L(p, q)$ with $p\neq(-1+q)/(qn+s)$ and $p\neq(-1-q)/(qn+s)$ for any integers $n, s$.
			
			We shall complete the proof by summing the number of isotopy classes of the case where $\partial H_1$ and $\partial H_3$ are Heegaard surfaces.
		\end{proof}
			
		\begin{prop}\label{thecase3}
			Let $H_1\cup H_2\cup H_3$ be a type-$(1, 1, 1)$ decomposition of  a lens space $M$ that satisfies conclusion (3) of Theorem \ref{thm 3}.
			If $M$ is  a lens space $L(p, q)$ with $p=2$,  there are two embeddings of $F_{12}\cup F_{13}\cup F_{23}$ into $M$ up to ambient isotopy.
			On the other hand, if $M$ is a lens space $L(p, q)$ with $p\neq 2$,  the type-$(1, 1, 1)$ decomposition $H_1\cup H_2\cup H_3$ of $M$ can be classified as follows.
			\begin{enumerate}
			\renewcommand{\theenumi}{(\arabic{enumi})}
				\item[(1)]  there are three embeddings of $F_{12}\cup F_{13}\cup F_{23}$ into $M$ up to ambient isotopy if $M$ is  $L(p, q)$ with \p.
				\item[(2)]  there are six embeddings of $F_{12}\cup F_{13}\cup F_{23}$ into $M$ up to ambient isotopy if $M$ is $L(p, q)$ with \np.
			\end{enumerate}
		\end{prop}
		\begin{proof}
			Let $H_1\cup H_2\cup H_3$ and $H_1'\cup H_2'\cup H_3'$ be the type-$(1, 1, 1)$ decompositions of a lens space that satisfies conclusion (3) of Theorem \ref{thm 3}.
			Furthermore, we denote $F_{ij}=H_i\cap  H_j$ and $F'_{ij}=H'_i\cap  H'_j$.
			We can assume that $F_{12} \cong F_{12}'\cong D^2\cup P$ without loss of generality.
			By Lemma \ref{heegaard}, there are three cases where $\partial H_i$ (resp. $\partial H_i'$) is a Heegaard surface of $M$ for $i=1, 2, 3$.
			If $\partial H_i$ is a Heegaard surface of $M$, then $\partial H_i$ is isotopic to $\partial V_1$ since the Heegaard splitting of $M$ is unique up to isotopy for $i=1, 2, 3$.
			
			If $\partial H_i$ is a Heegaard surface of $M$ and $H_i=V_1$, then the isotopy classes of the embedding $F_{12}\cup F_{13}\cup F_{23}$ into $M$ depend on only the embedding $F_{i3}\cong A\cup A$ into $V_2$ for $i=1, 2$. 
			Therefore, we can assume that the cases where $\partial H_1$ is a Heegaard surface and $\partial H_2$ is a Heegaard surface are the same. 
			Hence, we must consider the case where $\partial H_1$ is a Heegaard surface of $M$ and the case where  $\partial H_3$ is a Heegaard surface of $M$.
			Let $V_1\cup V_2$ be a genus-one Heegaard splitting of $M$.
			
			First, we consider the case where $M$  is a lens space $L(p, q)$ with $p=2$.
			We suppose that each of $\partial H_1$ and $\partial H_1'$ is a Heegaard surface.
			Then, we can assume that $H_1=V_1$ since a Heegaard surface of $M$ is unique up to isotopy.
			$F_{23}\cong A \cup A$ is properly embedded in $V_2$. $F_{23}$ cuts open $V_2$ into two solid tori $H_2$ and $H_3$ and $\partial V_2 - \partial F_{23}=F_{12}\cup F_{13}\cong P\cup D^2\cup A\cup A$.
			By Lemma \ref{annuiso1}, the embedding of $F_{23}$ into $V_2$ has two isotopy classes.
			By Lemma \ref{diff}, $M$ admits an ambient isotopy $F: M\times [0, 1]\to M$ such that $F( V_2, 1)=V_2$ and $f_1|_{\partial V_2}:\partial V_2\to \partial V_2$ is a hyperelliptic involution in the mapping class group of a torus $\partial V_2$, where $f_1(x)=F(x, 1)$.
			Since  $M$ admits an ambient isotopy $F: M\times [0, 1]\to M$ such that $F( V_2, 1)=V_2$ and $f_1|_{\partial V_2}:\partial V_2\to \partial V_2$ is a hyperelliptic involution in a mapping class group of a torus $\partial V_2$,  two of the isotopy classes of $F_{23}$ are ambient isotopic to each other by taking $F_{12}$ (resp. $F_{13}$) to $F_{12}$ (resp. $F_{13}$), where $f_1(x)=F(x, 1)$.
			If $H_1'=V_1$, $F_{23}'$  is properly embedded in $V_2$.
			$F_{23}'$ cuts open $V_2$ into two solid tori $H_2'$ and $H_3'$ and $\partial V_2 - \partial F_{23}'=F_{12}'\cup F_{13}'\cong P\cup D^2\cup A\cup A$.
			From the above discussion, $F_{12}\cup F_{13}\cup F_{23}$ is ambient isotopic to  $F_{12}'\cup F_{13}'\cup F_{23}'$.
			Suppose that $H_1'=V_2$.
			By Lemma \ref{core}, $V_1$ is isotopic to $V_2$ if and only if $M$ is $L(p, q)$ with \p.
			Since $M$  is a lens space $L(p, q)$ with $p=2$, $M$ satisfies \p.
			Then, we can isotope $H_1'$ to $H_1$.
			Hence, if each of $\partial H_1$ and $\partial H_1'$ is a Heegaard surface of $M$,  $F_{12}\cup F_{13}\cup F_{23}$ is ambient isotopic to  $F_{12}'\cup F_{13}'\cup F_{23}'$ in $M$.
			
			We suppose that $\partial H_3$ is a Heegaard surface of $M$.
			Then, we can assume that $H_3=V_1$ since a Heegaard surface of $M$ is unique up to isotopy.
			$F_{12}\cong D\cup P$ is properly embedded in $V_2$.
			By Lemma \ref{punc_2}, there is exactly one embedding of $F_{12}\cong D\cup P$ into $V_2$ up to ambient isotopy.
			If  $H_3'=V_2$, $H_3'$ is isotopic to $V_1$ if and only if $M$ is homeomorphic to $L(p, q)$ with \p by Lemma \ref{core}.
			Since $M$  is a lens space $L(p, q)$ with $p=2$, $M$ satisfies  \p.
			Then, we can isotope $H_3'$ to $V_1$.
			Hence, we can isotope $F_{12}\cup F_{13}\cup F_{23}$ to $F'_{12}\cup F'_{13}\cup F'_{23}$.
			Therefore,  if $\partial H_3$ is a Heegaard surface of $M$ and $M$ satisfies $p=2$, the type-$(1, 1, 1)$ decomposition of the 3-sphere or a lens space $M$ that satisfies conclusion (3) of Theorem \ref{thm 3} has one isotopy class. 
			
			Next, we consider the case where $M$ is a lens space $L(p, q)$ with $p\neq 2$.
			We suppose that $\partial H_1$ is a Heegaard surface.
			Then, we can assume that $H_1=V_1$ since a Heegaard surface of $M$ is unique up to isotopy.
			$F_{23}\cong A \cup A$ is properly embedded in $V_2$. $F_{23}$ cuts open $V_2$ into two solid tori $H_2$ and $H_3$ and $\partial V_2 - \partial F_{23}=F_{12}\cup F_{13}\cong P\cup D^2\cup A\cup A$.
			By Lemma \ref{annuiso1}, the embedding of $F_{23}$ into $V_2$ has two isotopy classes .
			By Lemma \ref{diff}, $M$ does not admit an ambient isotopy $F: M\times [0, 1]\to M$ such that $F( V_2, 1)=V_2$ and $f_1|_{\partial V_2}:\partial V_2\to \partial V_2$ is a hyperelliptic involution in the mapping class group of a torus $\partial V_2$, where $f_1(x)=F(x, 1)$.
			Since  $M$ does not admit an ambient isotopy $F: M\times [0, 1]\to M$ such that $F( V_2, 1)=V_2$ and $f_1|_{\partial V_2}:\partial V_2\to \partial V_2$ is a hyperelliptic involution in a mapping class group of a torus $\partial V_2$,  two of the isotopy classes of $F_{23}$ cannot be ambient isotopic to each other by sending $F_{12}$ (resp. $F_{13}$) to $F_{12}$ (resp. $F_{13}$), where $f_1(x)=F(x, 1)$. 
			Suppose that $H_1'=V_2$.
			By Lemma \ref{core}, $H_1'$ is isotopic to $H_1(=V_1)$ if and only if $M$ is homeomorphic to $L(p, q)$ with \p.
			Then, if $\partial H_1$ is a Heegaard surface of $M$ and $M$ is $L(p, q)$ with \p, the embedding $F_{12}\cup F_{13}\cup F_{23}$ into $M$ has two ambient isotopy classes.
			One the other hand, if $\partial H_1$ is a Heegaard surface of $M$ and $M$ is  $L(p, q)$ with \np, the embedding $F_{12}\cup F_{13}\cup F_{23}$ into $M$ has four ambient isotopy classes.
			
			We suppose that $\partial H_3$ is a Heegaard surface of $M$.
			Then, we can assume that $H_3=V_1$.
			$F_{12}\cong D\cup P$ is properly embedded in $V_2$.
			By Lemma \ref{punc_2}, there is exactly one embedding of $F_{12}\cong D\cup P$ into $V_2$ up to ambient isotopy.
			If  $H_3'=V_2$, $H_3'$ is isotopic to $H_3(=V_1)$ if and only if $M$ is homeomorphic to $L(p, q)$ with \p by Lemma \ref{core}.
			Hence, if  $H_3'=V_2$ and $M$ is $L(p, q)$ with \p, $F_{12}\cup F_{13}\cup F_{23}$ has exactly one isotopy class.
			On the other hand, if  $H_3'=V_2$ and $M$ is $L(p, q)$ with \np, $F_{12}\cup F_{13}\cup F_{23}$ cannot be isotopic to $F'_{12}\cup F'_{13}\cup F'_{23}$.
			Therefore,  if $\partial H_3$ is a Heegaard surface of $M$ and $M$ is $L(p, q)$ with \p, the type-$(1, 1, 1)$ decomposition of the 3-sphere or a lens space $M$ that satisfies conclusion (3) of Theorem \ref{thm 3} has one isotopy class. 
			One the other hand,  if $\partial H_3$ is a Heegaard surface of $M$ and $M$ is $L(p, q)$ with \np, the type-$(1, 1, 1)$ decomposition of the 3-sphere or a lens space $M$ that satisfies conclusion (3) of Theorem \ref{thm 3} has two isotopy classes. 

		\end{proof}
		
		\begin{prop}\label{thecase4}
			Let $H_1\cup H_2\cup H_3$ be a type-$(1, 1, 1)$ decomposition of a lens space $M$ that satisfies conclusion (4) of Theorem \ref{thm 3} and $V_1\cup V_2$ be a genus-one Heegaard splitting of $M$.
			If $M$ is the 3-sphere or a lens space $L(p, q)$ with $p=2$, there are two embeddings of $F_{12}\cup F_{13}\cup F_{23}$ into $M$ up to ambient isotopy.
			On the other hand, if $M$ is a lens space $L(p, q)$ with $p\neq 2$,  the type-$(1, 1, 1)$ decomposition $H_1\cup H_2\cup H_3$ of $M$ can be classified as follows.
			\begin{enumerate}
			\renewcommand{\theenumi}{(\arabic{enumi})}
				\item[(1)]  there are four embeddings of $F_{12}\cup F_{13}\cup F_{23}$ into $M$ up to ambient isotopy if $M$ is homeomorphic to $L(p, q)$ with \p.
				\item[(2)]  there are eight embeddings of $F_{12}\cup F_{13}\cup F_{23}$ into $M$ up to ambient isotopy if $M$ is not homeomorphic to $L(p, q)$ with \np.
			\end{enumerate}
		\end{prop}
		\begin{proof}
			Let $H_1\cup H_2\cup H_3$ and $H_1'\cup H_2'\cup H_3'$ be the type-$(1, 1, 1)$ decompositions of a lens space which satisfies the conclusion (4) of Theorem \ref{thm 3}.
			Furthermore, we denote $F_{ij}=H_i\cap  H_j$ and $F'_{ij}=H'_i\cap  H'_j$.
			We can assume that $F_{12} \cong F_{12}'\cong A\cup A$ without loss of generality.
			By Lemma \ref{heegaard}, there are three cases where $\partial H_i$ (resp. $\partial H_i'$) is a Heegaard surface of $M$ for $i=1, 2, 3$.
			Let $V_1\cup V_2$ be a genus-one Heegaard splitting of $M$.
			If $\partial H_i$ is a Heegaard surface of $M$, then $\partial H_i$ is isotopic to $\partial V_1$ since a Heegaard splitting of $M$ is unique up to isotopy for $i=1, 2, 3$.
			
			If $\partial H_i$ is a Heegaard surface of $M$ and $H_i=V_1$, then the isotopy classes of the embedding $F_{12}\cup F_{13}\cup F_{23}$ into $M$ depend on only the embedding $F_{i3}\cong D\cup P$ into $V_2$ for $i=1, 2$. 
			Therefore, we can assume that the cases where $\partial H_1$ is a Heegaard surface and $\partial H_2$ is a Heegaard surface are the same. 
			Hence, we must consider the cases where $\partial H_1$ is a Heegaard surface of $M$ and the case where  $\partial H_3$ is a Heegaard surface of $M$.
			
			First, we consider the case where $M$ is  a lens space $L(p, q)$ with $p=2$.
			We suppose that each of $\partial H_1$ and $\partial H_1'$ is a Heegaard surface.
			Then, we can assume that $H_1=V_1$ since a Heegaard surface of $M$ is unique up to isotopy.
			$F_{23}\cong D \cup P$ is properly embedded in $V_2$. 
			By Lemma \ref{punc_1}, there are two embeddings of $F_{12}$ into $V_2$ up to ambient isotopy.
			By Lemma \ref{diff}, the isotopy classes of $F_{12}\cup F_{13}\cup F_{23}$ are isotopic to each other in $M$.
			If  $H_1'=V_2$, $H_1'$ is isotopic to $H_1(=V_1)$ if and only if $M$ is homeomorphic to $L(p, q)$ with \p by Lemma \ref{core}.
			Since $M$  is a lens space $L(p, q)$ with $p=2$, $L(p, q)$ satisfies \p.
			Hence, if  $H_1'=V_2$, the embedding of $F_{12}\cup F_{13}\cup F_{23}$ is ambient isotopic to the embedding of $F'_{12}\cup F'_{13}\cup F'_{23}$.
			Therefore,  if $\partial H_1$ is a Heegaard surface of $M$, the type-$(1, 1, 1)$ decomposition of the 3-sphere or a lens space $M$ that satisfies conclusion (4) of Theorem \ref{thm 3} has one isotopy class. 
			
			We suppose that $\partial H_3$ is a Heegaard surface of $M$.
			$F_{12}\cong A\cup A$ is properly embedded in $V_2$. 
			$F_{12}$ cuts open $V_2$ into two solid tori $H_1$ and $H_2$ and $\partial V_2 - \partial F_{12}\cong P\cup D^2\cup P\cup D^2$ since $\partial V_2=\partial H_3=F_{13}\cup F_{23}$.
			By Lemma \ref{annuiso2}, there are exactly two such embeddings of $F_{12}$ into $V_2$ up to ambient isotopy.
			By Lemma \ref{diff}, $M$ admits an ambient isotopy $F: M\times [0, 1]\to M$ such that $F( V_2, 1)=V_2$ and $f_1|_{\partial V_2}:\partial V_2\to \partial V_2$ is a hyperelliptic involution in the mapping class group of a torus $\partial V_2$, where $f_1(x)=F(x, 1)$.
			Since  $M$ admits an ambient isotopy $F: M\times [0, 1]\to M$ such that $F( V_2, 1)=V_2$ and $f_1|_{\partial V_2}:\partial V_2\to \partial V_2$ is a hyperelliptic involution in a mapping class group of a torus $\partial V_2$,  two of the isotopy classes of $F_{12}$ are ambient isotopic to each other by sending $F_{23}$ (resp. $F_{13}$) to $F_{23}$ (resp. $F_{13}$). 
			By Lemma \ref{core}, $H_1$ is isotopic to $V_2$ if and only if $M$ is homeomorphic to $L(p, q)$ with \p.
			Since $M$  is a lens space $L(p, q)$ with $p=2$, $M$ satisfies \p.
			Then,  the embedding $F_{12}\cup F_{13}\cup F_{23}$  into $M$ has one isotopy class if $\partial H_1$ is a Heegaard surface of $M$.
			
			Next, we consider the case where $M$ is a lens space $L(p, q)$ with $p\neq2$.
			We suppose that $\partial H_1$ is a Heegaard surface.
			Then, we can assume that $H_1=V_1$.
			$F_{23}\cong D \cup P$ is properly embedded in $V_2$. 
			By Lemma \ref{punc_1}, the embedding of $F_{23}$ into $V_2$ has two ambient isotopy classes.
			Furthermore, each ambient isotopy class of  the embedding of $F_{23}$ into $V_2$ is not isotopic to each other by taking $\partial H_1$ to $\partial H_1$ if $p\neq 2$ by Lemma \ref{diff}.
			If $H_1'=V_2$, $H_1'$ is isotopic to $H_1(=V_1)$ if and only if $M$ is homeomorphic to $L(p, q)$ with \p by Lemma \ref{core}.
			Hence, if $M$ is  $L(p, q)$ with \p, there are exactly two embedding of $F_{12}\cup F_{13}\cup F_{23}$ up to ambient isotopy.
			On the other hand, if  $H_1'=V_2$ and $M$ is $L(p, q)$ with \np, $F_{12}\cup F_{13}\cup F_{23}$ cannot be isotopic to $F'_{12}\cup F'_{13}\cup F'_{23}$.
			Hence, if $M$ is $L(p, q)$ with \np, there are four embeddings of $F_{12}\cup F_{13}\cup F_{23}$ into $M$ up to ambient isotopy.
			
			We suppose that $\partial H_3$ is a Heegaard surface of $M$.
			Then, we can assume that $H_3=V_1$.
			$F_{12}\cong A\cup A$ is properly embedded in $V_2$. 
			$F_{12}$ cuts open $V_2$ into two solid tori $H_1$ and $H_2$ and $\partial V_2 - \partial F_{12}\cong P\cup D^2\cup P\cup D^2$ since $\partial V_2=\partial H_3=F_{13}\cup F_{23}$.
			By Lemma \ref{annuiso2}, there are exactly two such embeddings of $F_{12}$ into $V_2$ up to ambient isotopy.
			Furthermore, the ambient isotopy classes of  the embedding of $F_{12}$ into $V_2$ are not isotopic to each other by taking $\partial H_3$ to $\partial H_3$ if $p\neq 2$ by Lemma \ref{diff}.
			If $H_3'=V_2$, $H_3'$ is isotopic to $H_3(=V_1)$ if and only if $M$ is homeomorphic to $L(p, q)$ with \p by Lemma \ref{core}.
			This implies that if $M$ is $L(p, q)$ with \np, $F_{12}\cup F_{13}\cup F_{23}$ cannot be isotopic to $F_{12}'\cup F_{13}'\cup F_{23}'$.
			To summarize the above discussion,  there are exactly two embedding of  $F_{12}\cup F_{13}\cup F_{23}$  into $M$ up to ambient isotopy if $\partial H_3$ is a Heegaard surface of $M$ and $M$ is $L(p, q)$ with \p.
			One the other hand,  there are exactly four embeddings of $F_{12}\cup F_{13}\cup F_{23}$ into $M$ up to ambient isotopy if $\partial H_3$ is a Heegaard surface of $M$ and $M$ is $L(p, q)$ with \np.
		\end{proof}
		
		\begin{prop}\label{thecase5}
			Let $H_1\cup H_2\cup H_3$ be the type-$(1, 1, 1)$ decomposition of a lens space $M$ that satisfies conclusion (5) of Theorem \ref{thm 3}.
			The embedding of $F_{12}\cup F_{13}\cup F_{23}$ into $M$ is unique up to isotopy  if $M$ is $L(p, q)$ with \p.
			On the other hand, there are exactly two embeddings of $F_{12}\cup F_{13}\cup F_{23}$ into $M$ up to ambient isotopy if $M$ is $L(p, q)$ with \np.
		\end{prop}
		\begin{proof}
			Let $H_1\cup H_2\cup H_3$ and $H_1'\cup H_2'\cup H_3'$ be the type-$(1, 1, 1)$ decompositions of a lens space $M$ that satisfies conclusion (5) of Theorem \ref{thm 3}.
			Furthermore, we denote $F_{ij}=H_i\cap  H_j$ and $F'_{ij}=H'_i\cap  H'_j$.
			We can assume that $F_{12} \cong F_{12}'\cong D\cup P$ without loss of generality.
			By Lemma \ref{heegaard}, there are three cases where $\partial H_i$ (resp. $\partial H_i'$) is a Heegaard surface of $M$ for $i=1, 2, 3$.
			Let $V_1\cup V_2$ be a genus-one Heegaard splitting of $M$.
			If $\partial H_i$ is a Heegaard surface of $M$, then $\partial H_i$ is isotopic to $\partial V_1$ since a Heegaard splitting of $M$ is unique up to isotopy for $i=1, 2, 3$.
			
			If $\partial H_i$ is a Heegaard surface of $M$ and $H_i=V_1$, then the isotopy classes of the embedding $F_{12}\cup F_{13}\cup F_{23}$ into $M$ depend on only the embedding $F_{jk}\cong D\cup P$ into $V_2$ for $\{i,  j, k\}=\{1, 2, 3\}$. 
			Therefore, we can assume that the cases where $\partial H_i$ is a Heegaard surface are the same for $i=1, 2, 3$. 
			Hence, we have to consider only the cases where $\partial H_1$ is a Heegaard surface of $M$.
			
			We suppose that $\partial H_1$ and $\partial H_1'$ are Heegaard surface. 
			Then, we can suppose that $H_1=V_1$.
			$F_{23}\cong D \cup P$ is properly embedded in $V_2$, which satisfies the assumption of Lemma \ref{punc_3}. 
			Hence, if $H'_3=V_1$, $F_{23}$ is isotopic to $F_{23}'$ by taking $\partial H_1$ to $\partial H_1$ by Lemma \ref{punc_3}. 
			Then, $F_{12}\cup F_{13}\cup F_{23}$ is isotopic to $F'_{12}\cup F'_{13}\cup F'_{23}$.
			If  $H_1'=V_2$, $H_1'$ is isotopic to $V_1$ if and only if $M$ is  $L(p, q)$ with \p  by Lemma \ref{core}.
			Hence, if  $H_1'=V_2$ and $M$ is $L(p, q)$ with \p, $F_{12}\cup F_{13}\cup F_{23}$ is isotopic to $F'_{12}\cup F'_{13}\cup F'_{23}$.
			On the other hand, if  $H_1'=V_2$ and $M$ is $L(p, q)$ with \np, $F_{12}\cup F_{13}\cup F_{23}$ cannot be isotopic to $F'_{12}\cup F'_{13}\cup F'_{23}$.
			Therefore,  if $\partial H_1$ is a Heegaard surface of $M$ and $M$ is $L(p, q)$ with \p, the type-$(1, 1, 1)$ decomposition of the 3-sphere or a lens space $M$ that satisfies conclusion (3) of Theorem \ref{thm 3} has one isotopy class. 
			One the other hand,  if $\partial H_1$ is a Heegaard surface of $M$ and $M$ is $L(p, q)$ with \np, the type-$(1, 1, 1)$ decomposition of the 3-sphere or a lens space $M$ that satisfies conclusion (3) of Theorem \ref{thm 3} has two isotopy classes. 
		\end{proof}

		\begin{proof}[Proof of Theorem \ref{class111s3}]
		By Propositions \ref{s3thecase2} and \ref{s3thecase3}, we obtain Theorem \ref{class111s3}. 
		\end{proof}
		
		\begin{proof}[Proof of Theorem \ref{class111}]
		By Propositions \ref{thecase2}, \ref{thecase3}, \ref{thecase4}, and \ref{thecase5}, we obtain Theorem \ref{class111}.
		\end{proof}

		\setcounter{section}{4}

%=============================================================================================
\setcounter{section}{4}
\section{Stabilization of handlebody decomposition}
The stabilization of a Heegaard splitting is defined as a connected sum with a genus-one Heegaard splitting of the 3-sphere. Koenig defined the stabilization operation of handlebody decompositions. Reidemeister and  Singer showed  a stably equivalent theorem for a Heegaard splitting and its stabilization. 

We consider the stabilization of handlebody decompositions.  We focus on the handlebody decomposition of the 3-sphere and lens spaces.
 %In \cite{KOKK}, Koda, Ozawa, Ishihara and Shimokawa showed a stably equivalence of handlebody decomposition. 
%We consider stabilization about a handelbody decomposition of lens spaces.
%Stabilizations of handlebody decompositions are defined as below.
We restate  Theorem \ref{stab_2}.
	\setcounter{section}{2}
	\setcounter{thm}{6}
	\begin{thm}\label{stab_2}
	A type-$(0, 0, 1)$ decomposition of the 3-sphere and type-$(0, 1, 1)$ and  type-$(1, 1, 1)$ decompositions of the 3-sphere and lens spaces satisfy the following.
	\begin{enumerate}
	\renewcommand{\theenumi}{(\arabic{enumi})}
			\item[(1)]  a type-$(0,0,1)$ decomposition of the 3-sphere is obtained from a type-$(0,0,0)$ decomposition by a type-1 stabilization.
			\item[(2)]  a type-$(0,1,1)$ decomposition of the 3-sphere and lens spaces is obtained from a type-$(0,0,1)$ decomposition by a type-1 stabilization.
			\item[(3)]  a type-$(1,1,1)$ decomposition of the 3-sphere and  lens spaces is obtained from a type-$(0,0,1)$ decomposition by a sequence of  type-1 stabilizations and the case of Theorem \ref{thm 3} (6) is not stabilized. 
\end{enumerate}
	\end{thm}

\setcounter{section}{5}
%	\begin{lem}
%		Let $M=H_1\cup H_2 \cup H_3$ be a handlebody decomposition with $F_{ij}$ are annuli and $\partial F_{ij}$ are essential in $\partial H_i$ for $i=1, 2, 3$.
%		If $M$ is a lens space then one of the image $\partial F_{ij}$ in $\partial H_i$ is longitude or meridian for some $i$ and $j$. 
%	\end{lem}
%	\begin{proof}
%		Suppose $H_1$ and $H_2$ intersects along an annulus $A$. Let $f$ be a homeomorphism of identification of $H_1\cup_{A} H_2$. we can take a Seifert fibration of $H_1$ and $H_2$ such that $A$ is a fibered annulus in its fibration.
%		Put $S(2)\cong H_1\cup H_2$.
%		$\partial S(2)=A_{13}\cup A_{23}\cong T^2$. 
%		Homeomorphism $\partial S(2)$ to $\partial H_3$ send $\partial A_{13}=\partial A_{23}$ as essential simple closed curves $c_1$ and $c_2$ in $\partial H_3$.
%		Suppose $c_i$ are essential simple closed curve which is not longitude and also meridian for $i=1, 2$.
%		Then we can take a Seifert fibration of $H_3$ which $c_i$ is a one of the regular fiber.
%		The core of $H_3$ will be a singular fiber.
%		Then $M\cong S(3)$.
%		this contradicts a $M$ is a lens space. 
%		So $c_i$ is a longitude or meridian.
%	 \end{proof}

	\begin{proof}
	First, we consider a type-$(0, 0, 1)$ decomposition. 
Let $H_{1} \cup H_{2} \cup H_{3}$ be a type-$(0,0,1)$ decomposition of the 3-sphere.
There is a meridian disk of a solid torus $H_3$ that intersects the branched loci twice. 
Hence, we can destabilize the decomposition by a type-1 destabilization. 
% It is satisfied that $F_{12}$ has a disk component whose boundary is essential (resp. inessential) in $H_{3}$ (resp. $H_{1}$ and $H_{2}$). Hence we can do the type  2 destabilization and obtain the type-$(0,0,0)$ decomposition. This implies that this decomposition is obtained by the type I' stabilization from a type-$(0, 0, 0)$ decomposition.

		Next, we consider a type-$(0, 1, 1)$ decomposition. 
		the 3-sphere and lens spaces have two type-$(0, 1, 1)$ decompositions . 
		If the decomposition satisfies case (1) of the statements of Theorem \ref{s3011}, there is a meridian disk of $H_2$ that intersects the branched loci exactly twice. 
		Then, we can perform a type-1 destabilization. Subsequently, we can obtain a type-$(0, 0, 1)$ decomposition of a lens space.
		Hence this case is stabilized from a type-$(0, 0, 1)$ decomposition.
		
		If a handlebody decomposition satisfies conclusion (2) of Theorem \ref{s3011},
		exactly two branched loci in $\partial H_2$ or $\partial H_3$ form a Heegaard diagram of the 3-sphere.
		We can take a meridian disk of $H_2$ or $H_3$ whose boundary intersects them exactly once.
		This implies that there is a meridian disk of $H_2$ or $H_3$ that intersects the branched loci exactly twice.  
		Therefore, we can perform a type-1 destabilization.  
		Hence, we can obtain a type-$(0, 0, 1)$ decomposition of a lens space from a type-$(0, 1, 1)$ decomposition of case (2) of Theorem \ref{s3011}.

		Finally, we consider a type-$(1, 1, 1)$ decomposition.
		First, we consider the case where $F_{12}$ has a disk component. This case corresponds to conclusions (1), (2), (3), and (4) of Theorem \ref{thm 3}.
\begin{claim}\label{cl}
	If a handlebody decomposition satisfies conclusions (3), (4), and (5) of Theorem \ref{thm 3}, then there exist two handlebodies $H_i$ and $H_j$ whose boundaries satisfy the condition that exactly two branched loci are essential in both $\partial H_i$ and $\partial H_j$.
\end{claim}
\begin{proof}[Proof of Claim \ref{cl}]
	We consider conclusion (3) of Theorem \ref{thm 3}.  
	Conclusion (1) of  Theorem \ref{thm 3} satisfies $\partial H_1=D_{12}\cup P_{12}\cup A^1_{13}\cup A^2_{13}$, where $D_{12}$ and  $P_{12}$ are components of $F_{12}$ and $A^i_{13}$ is a component of $F_{13}$ for $i=1, 2$.
	One of the components of $F_{13}$ has a boundary component corresponding to $\partial D_{12}$. 
	We assume that $A^1_{13}$ is such a component. Then,  $\partial A^1_{13}$ is inessential in $\partial H_1$. 
	If $\partial A^2_{13}$ is also inessential in $\partial H_1$, $F_{12}$ or $F_{13}$ has a disk component that is not $D_{12}$. 
	This contradicts the assumption. Hence, $\partial A^2_{13}$ is essential in $\partial H_1$. Similarly, $\partial H_2$ satisfies the claim condition. 
	
	Next, we consider conclusion (4) of Theorem \ref{thm 3}. Conclusion (2) of Theorem \ref{thm 3} satisfies $\partial H_2=D_{12}\cup P_{12}\cup A^1_{23}\cup A^2_{23}$, where $D_{12}$ and  $P_{12}$ are components of $F_{12}$ and $A^i_{23}$ is a component of $F_{23}$ for $i=1, 2$. Hence, $\partial H_2$ satisfies the claim condition from the above discussion. Furthermore, $\partial H_3$ satisfies the claim condition.
	
	Finally, we consider conclusion (5) of Theorem \ref{thm 3}. Conclusion (3) of Theorem \ref{thm 3}  satisfies $\partial H_1=D_{12}\cup P_{12}\cup D_{13}\cup P_{13}$, where $D_{12}$ and  $P_{12}$ are components of $F_{12}$ and $D_{13}$ and $P_{13}$ are  components of $F_{23}$. One of the components of $\partial P_{23}$ corresponds to $\partial D_{12}$. Then, $P_{23}\cup D_{12}$ is an annulus in $\partial H_1$. Similarly, $P_{12}\cup D_{13}$ is also an annulus in $\partial H_1$. Hence, $\partial H_1$ is obtained by identifying the boundary components of two annuli. Then, $\partial H_1$ satisfies  the condition of the claim. Similarly, $\partial H_2$ satisfies  the condition of the claim.
\end{proof} 
\begin{claim}\label{cl2}
	One of the boundaries of a handlebody in the statement of Claim \ref{cl} contains exactly two branched loci that intersect the boundary of a meridian disk of a handlebody exactly once.
\end{claim}
\begin{proof}[Proof of Claim \ref{cl2}]
Suppose that both of the boundaries of the handlebodies in the statement of Claim \ref{cl} contain exactly two branched loci that intersect the boundary of a meridian disk of the handlebody more than twice.
Then, $M$ has two lens spaces as a connected summand. %longitude方向二回以上っているので%
This contradicts the assumption.
\end{proof}
\begin{claim}\label{1}
A handlebody decomposition that satisfies conclusion (2), (3), (4), or (5) of Theorem \ref{thm 3} can be destabilized by a type-1 destabilization.
\end{claim}
\begin{proof}
	By Claims \ref{cl} and \ref{cl2}, case (3), (4), or (5) is stabilized by a type-1 destabilization since there exists a meridian of a handlebody that intersects the branched loci exactly twice. 
	If a handlebody decomposition satisfies case (2), we can find a meridian of $H_1$ that intersects the branched loci exactly twice since all the branched loci in $\partial H_1$ are inessential.
	Then, we can perform a type-1 destabilization on the handlebody decomposition of case (4).
\end{proof}
\begin{claim}\label{2}
	A handlebody decomposition that satisfies conclusion (1) of Theorem \ref{thm 3} can be destabilized by a type-1 destabilization.
\end{claim}
\begin{proof}
		Since the Seifert fibered structures of the 3-sphere and lens spaces have at most two singular fibers, we can suppose that one of the $H_i$'s is a regular neighborhood of a regular fiber.  Suppose that $H_1$ is a regular neighborhood of a regular fiber.
		Hence, in case  (5) in Theorem \ref{thm 3},  $H_1$ has a meridian disk whose boundary intersects the branched loci exactly twice. 
		Therefore, we can perform a type-1 destabilization along such a meridian disk.
\end{proof}
\begin{claim}\label{3}
	The case of Theorem \ref{thm 3} (6) cannot be destabilized. 
\end{claim}
\begin{proof}
	Since each branched loci in this case is essential in each boundary of  the handlebodies and the number of branched loci is four, there is no meridian disk of handlebodies that intersects the branched loci exactly twice.
	Hence, we cannot perform a type-1 destabilization.
	If the case of Theorem \ref{thm 3} (6) is stabilized by a type-0 stabilization from a type-$(0, 0, 1)$ decomposition, the decomposition satisfies $F_{12}\cong F_{13}\cong F_{23}\cong A$. 
	This contradicts case (6) of Theorem \ref{thm 3}.
\end{proof}
		By Claims \ref{1}, \ref{2}, and \ref{3}, we complete the proof of Theorem \ref{stab_2}
	\end{proof}
	\begin{Rem}
	\begin{enumerate}
	\item A type-$(0, 1, 1)$   decomposition of the 3-sphere that satisfies conclusion (1) of Theorem \ref{s3011} is obtained by performing a type-0 stabilization on a type-$(0, 0, 0)$ decomposition of the 3-sphere. Furthermore, a type-$(1, 1, 1)$ decomposition of the 3-sphere or a lens space that satisfies conclusion (5) of Theorem \ref{thm 3} is obtained by performing a type-0 stabilization on a type-$(0, 0, 1)$ decomposition of the 3-sphere or a lens space.
	\item A type-$(1, 1, 1)$ decomposition of a Seifert fibered space $\mathbb{S}(3)$ that is not a lens space or  the 3-sphere or $S^2\times S^1$ is not stabilized. In fact, $\mathbb{S}(3)$ does not have type-$(0, 0, 0)$, type-$(0, 0, 1)$, and type-$(0, 1, 1)$ decompositions.
	\end{enumerate}
	\end{Rem}
	\section{Acknowledgements}
	We would like to express our gratitude to our supervisor, Professor Koya Shimokawa, for his valuable advice. We would also like to thank  Kai Ishihara for his insightful comments on the definition of stabilizations.
	The second author was partially supported by Grant-in-Aid for JSPS Research Fellow from JSPS KAKENHI Grant Number JP20J20545.
$ $\\
Yasuyoshi Ito

Department of Mathematics, Saitama University, 255 Shimo-Okubo, Sakura-ku, Saitama- shi, Saitama, 338-8570, Japan

E-mail address: 0104y.ito@gmail.com$ $\\

Masaki Ogawa 

Department of Mathematics, Saitama University, 255 Shimo-Okubo, Sakura-ku, Saitama- shi, Saitama, 338-8570, Japan

E-mail address: m.ogawa.691@ms.saitama-u.ac.jp

\end{document}